\documentclass[a4paper,reqno,11pt]{amsart}
\usepackage{amsmath}
\usepackage{amsfonts}
\usepackage{amssymb}
\usepackage{amsthm}
\usepackage[dvips]{graphicx}
\usepackage{hyperref}

\usepackage[mathscr]{euscript}
\usepackage{enumitem}

\newtheorem{theorem}{Theorem}[section]
\newtheorem{proposition}[theorem]{Proposition}
\newtheorem{lemma}[theorem]{Lemma}

\theoremstyle{definition}
\newtheorem{definition}[theorem]{Definition}
\newtheorem{example}[theorem]{Example}

\theoremstyle{remark}
\newtheorem{remark}[theorem]{Remark}

\newtheorem*{notation}{Notation}
\newtheorem*{acknowledgements}{Acknowledgements}

\numberwithin{equation}{section}

\def\enddoc{\end{document}}

\def\FRAME#1#2#3#4#5#6#7#8
{
 \begin{figure}[ptbh]
 \begin{center}
 \includegraphics[height=#3]{#7}
 \caption{#5}
 \label{#6}
 \end{center}
 \end{figure}
}

\newcommand{\ind}[1]{\mathbf{1}_{#1}}

\newcommand{\capa}{\mathop{\mathrm{cap}}\nolimits}
\newcommand{\Capa}{\mathop{\mathrm{Cap}}\nolimits}
\newcommand{\diam}{\mathop{\mathrm{diam}}\nolimits}

\newcommand{\interior}{\mathop{\mathrm{int}}\nolimits}

\newcommand{\supp}{\mathop{\mathrm{supp}}\nolimits}

\newcommand{\sigmafield}{\mathscr}
\newcommand{\functionspace}{\mathcal}
\newcommand{\cemetery}{\Delta}

\begin{document}

\title[Upper bounds of heat kernels for diffusions via Dynkin-Hunt formula]%
{Localized upper bounds of heat kernels for diffusions via a multiple Dynkin-Hunt formula}

\author{Alexander Grigor'yan}
\address{Fakult\"{a}t f\"{u}r Mathematik, Universit\"{a}t Bielefeld, Postfach 10 01 31, 33501 Bielefeld, Germany}
\curraddr{}
\email{grigor@math.uni-bielefeld.de}
\thanks{AG was supported by SFB 701 of the German Research Council (DFG)}

\author{Naotaka Kajino}
\address{Department of Mathematics, Graduate School of Science, Kobe University, Rokkodai-cho 1-1, Nada-ku, 657-8501 Kobe, Japan}
\curraddr{}
\email{nkajino@math.kobe-u.ac.jp}
\thanks{NK was supported by SFB 701 of the German Research Council (DFG)
and by JSPS KAKENHI Grant Number 26287017}



\date{July 3, 2015}

\dedicatory{}

\begin{abstract}
We prove that for a general diffusion process, certain assumptions on its behavior
\emph{only within a fixed open subset} of the state space imply the existence and
sub-Gaussian type off-diagonal upper bounds of the \emph{global} heat kernel on
the fixed open set. The proof is mostly probabilistic and is based on a
seemingly new formula, which we call a \emph{multiple Dynkin-Hunt formula},
expressing the transition function of a Hunt process in terms of that of the
part process on a given open subset. This result has an application to heat kernel
analysis for the \emph{Liouville Brownian motion}, the canonical diffusion in a
certain random geometry of the plane induced by a (massive) Gaussian free field.
\end{abstract}

\maketitle

\tableofcontents

\section{Introduction}\label{sec:intro}

Let $(M,d)$ be a locally compact separable metric space equipped with
a $\sigma$-finite Borel measure $\mu$ and let
$X=\bigl(\{X_{t}\}_{t\in[0,\infty]},\{\mathbb{P}_{x}\}_{x\in M_{\cemetery}}\bigr)$
be a diffusion on $M$, where $M_{\cemetery}:=M\cup\{\cemetery\}$
denotes the one-point compactification of $M$.
The themes of this paper are existence of the heat kernel
$p_{t}(x,y)$ (the transition density of $X$ with respect to $\mu$)
and off-diagonal upper bounds of $p_{t}(x,y)$ of the form
\begin{equation}\label{eq:UHK-intro}
p_{t}(x,y)\leq F_{t}(x,y)\exp\biggl(-\gamma\Bigl(\frac{d(x,y)^{\beta}}{t}\Bigr)^{\frac{1}{\beta-1}}\biggr)
\end{equation}
for some $\gamma\in(0,\infty)$, $\beta\in(1,\infty)$ and a positive function $F_{t}(x,y)$.
In most typical cases, $F_{t}(x,y)$ is given either by the power function
$F_{t}(x,y)=c_{0}t^{-\alpha}$ for some $c_{0},\alpha\in(0,\infty)$ or by the volume function
\begin{equation}\label{eq:F-vol-intro}
F_{t}(x,y)=c_{0}\mu\bigl(B(x,t^{1/\beta})\bigr)^{-1/2}\mu\bigl(B(y,t^{1/\beta})\bigr)^{-1/2},
\end{equation}
where $\beta$ is as in \eqref{eq:UHK-intro} and
$B(x,r):=\{y\in M\mid d(x,y)<r\}$ for $(x,r)\in M\times(0,\infty)$.

For $\beta=2$, \eqref{eq:UHK-intro} is called a \emph{Gaussian} upper bound and
has been extensively studied in the classical setting where $M$ is a complete
Riemannian manifold. For example, when $M$ has non-negative Ricci curvature,
the Gaussian bound \eqref{eq:UHK-intro} under \eqref{eq:F-vol-intro},
together with a matching lower bound, has been proved
for the Brownian motion on $M$ by Li and Yau \cite{LiYau:ActaMath86} and for
uniformly elliptic diffusions on $M$ by Saloff-Coste \cite{S-C:JDG92}. It is
also known by the results of Grigor'yan \cite{Gri:MatSb91,Gri:RevMatIberoam94}
and Saloff-Coste \cite{S-C:IMRN92,S-C:JDG92} that these bounds are
characterized or implied by certain scale-invariant functional inequalities,
such as Poincar\'{e}, local Sobolev and relative Faber-Krahn inequalities,
in conjunction with the volume doubling property
\begin{equation}\label{eq:VD-intro}
0<\mu(B(x,2r))\leq c_{\mathrm{vd}}\mu(B(x,r))<\infty.
\end{equation}
Saloff-Coste's proofs have developed from Moser's iteration argument in
\cite{Moser:CPAM64,Moser:CPAM71} combined with Davies' method in \cite{Dav:AJM87}
for making the constant $\gamma$ in \eqref{eq:UHK-intro} arbitrarily close to $\frac{1}{4}$,
and have been extended by Sturm \cite{Sturm:ALDS2,Sturm:ALDS3} to the framework of
a general strongly local regular Dirichlet space whose associated intrinsic metric
is non-degenerate. This last property basically means that for each relatively
compact ball $B(x,r)$ there exists a cutoff function $\varphi=\varphi_{x,r}$
satisfying $\ind{B(x,r)}\leq\varphi\leq\ind{B(x,2r)}$ and
``$|\nabla\varphi|\leq r^{-1}$" $\mu$-a.e., which makes it possible to apply
the methods developed for Riemannian manifolds to an abstract setting.
It should also be noted that such cutoff functions allow us to deduce
\emph{localized} Gaussian bounds from \emph{localized} assumptions; for example,
a Gaussian upper bound of $p_{t}(x,y)$ for \emph{given} $x,y\in M$ is implied by
a local Sobolev inequality \emph{on two balls $B(x,r_{x})$ and $B(y,r_{y})$ alone}.
See \cite{Dav:HK,Gri:HKAnalysisManifolds,S-C:AspectsSobolev,Sturm:ALDS2,Sturm:ALDS3}
and references therein for further details of Gaussian bounds.

The values of $\beta$ \emph{greater than $2$} naturally appear in the study of
diffusions on fractals. Barlow and Perkins have proved in their seminal work
\cite{BP} that the canonical diffusion on the two-dimensional Sierpi\'{n}ski
gasket satisfies \eqref{eq:UHK-intro} with \eqref{eq:F-vol-intro} and
$\beta=\log_{2}5>2$ as well as a matching lower bound, which indicate a lower
diffusion speed of the heat and are thereby called \emph{sub}-Gaussian bounds.
Such two-sided bounds with $\beta>2$ have been established also for nested fractals
by Kumagai \cite{Kum:nested}, affine nested fractals by Fitzsimmons, Hambly and
Kumagai \cite{FHK:affinenested} and Sierpi\'{n}ski carpets by Barlow and Bass
\cite{BB92,BB99} (see also \cite{BBKT}), which in turn have motivated a number
of recent studies on characterizing sub-Gaussian bounds, like
\cite{BBK:CS,BGK:PHI,GriHu:HKGreen,GriHuLau:GenCap,GT,Kig:RFQS,Kum:RF} for two-sided
and \cite{AB,Gri:HKfractal,GriHu:Upper,GriHuLau:GenCap,Kig:localNash} for upper.
A huge technical difficulty in the sub-Gaussian case is that, even though
we can construct good cutoff functions similar to the Gaussian case
\emph{a posteriori on the basis of sub-Gaussian bounds}
as has been done in \cite{AB,BBK:CS,GriHuLau:GenCap},
it is hopeless to have such functions \emph{a priori}; indeed, the natural
distance function may well even \emph{not} belong to the domain of the
Dirichlet form as proved in \cite[Proposition A.3]{K:aghSG} for
the two-dimensional Sierpi\'{n}ski gasket. Therefore in getting
sub-Gaussian bounds, practically we cannot use analytic methods developed for
Gaussian bounds, and most of the existent researches have made indispensable
use of arguments on the diffusion process instead.

While calculations with the diffusion enable us to estimate various
analytic quantities through probabilistic considerations, it is not clear
whether they admit localized implications similar to the analytic proofs
of Gaussian bounds, and there seems to be no result in the literature stating
such implications explicitly. In fact, unless the diffusion $X$ has a certain
prescribed local regularity property as in the case of Riemannian manifolds and
that of resistance forms treated in \cite{Kig:RFQS}, localizing \emph{existence} results
for the heat kernel $p_{t}(x,y)$ is already highly non-trivial, since its existence
on a given subset could be prevented by the possibly very bad behavior of the
diffusion outside the subset. These issues of localization have been carefully
avoided in the known probabilistic derivations of sub-Gaussian heat kernel bounds,
either by assuming as in \cite{Kig:localNash} the ultracontractivity of
the heat semigroup and thereby the existence and boundedness of
the heat kernel $p_{t}(x,y)$, or by assuming good situations everywhere
in every scale as in \cite{Gri:HKfractal,GriHu:Upper,GT} and
their descendants \cite{GriHu:HKGreen,GriHuLau:GenCap}.

The purpose of this paper is to provide a new probabilistic method of obtaining
\emph{localized} existence and sub-Gaussian upper bounds of the heat kernel
$p_{t}(x,y)$ of $X$ from \emph{localized} assumptions on $X$. Now we briefly
outline the statements of our main theorems.

The main localized existence theorem for the heat kernel (Theorem \ref{thm:HK-existence})
is proved for a Radon measure $\mu$ on $M$ with full support and a $\mu$-symmetric
Hunt process
$X=\bigl(\{X_{t}\}_{t\in[0,\infty]},\{\mathbb{P}_{x}\}_{x\in M_{\cemetery}}\bigr)$
on $M$ (\emph{not} necessarily with continuous sample paths) whose Dirichlet form
$(\functionspace{E},\functionspace{F})$ is regular on $L^{2}(M,\mu)$.
Let $U$ be a non-empty open subset of $M$, set
$\tau_{U}:=\inf\{t\in[0,\infty)\mid X_{t}\in M_{\cemetery}\setminus U\}$
($\inf\emptyset:=\infty$) and let
$\{T^{U}_{t}\}_{t\in(0,\infty)}$ denote the Dirichlet heat semigroup on $U$.
Then Theorem \ref{thm:HK-existence} states that for an interval
$I\subset(0,\infty)$ and open subsets $V,W$ of $M$, a
``\emph{$\mu$-almost everywhere} upper bound for $\{T^{U}_{t}\}_{t\in(0,\infty)}$ on
$I\times V\times W$ by a locally bounded upper semi-continuous kernel $H=H_{t}(x,y)$"
yields a Borel measurable function
$p^{U}=p^{U}_{t}(x,y)$ with $0\leq p^{U}_{t}(x,y)\leq H_{t}(x,y)$ such that
for \emph{$\functionspace{E}$-quasi-every} $x\in V$, for any $t\in I$,
\begin{equation}\label{eq:HK-existence-intro}
\mathbb{P}_{x}[X_{t}\in dy,\,t<\tau_{U}]=p^{U}_{t}(x,y)\,d\mu(y)\qquad\textrm{on }W.
\end{equation}
In fact, the same sort of results along with some additional regularity
properties of $p_{t}(x,y)$ have been obtained for $I=(0,\infty)$ and $U=V=W=M$
in \cite[Sections 7 and 8]{Gri:HKfractal} and \cite[Theorem 3.1]{BBCK}, but
our Theorem \ref{thm:HK-existence} should suffice for most applications
since it already guarantees the expected bound $p^{U}_{t}(x,y)\leq H_{t}(x,y)$
without requiring any regularity of the heat kernel $p^{U}_{t}(x,y)$.

The proof of Theorem \ref{thm:HK-existence} is mostly based on potential theory
for regular symmetric Dirichlet forms developed in \cite[Chapters 2 and 4]{FOT};
it should not be very difficult to generalize Theorem \ref{thm:HK-existence} to
a wider framework where the same kind of potential theory is still available.
As an intermediate step for the proof of Theorem \ref{thm:HK-existence},
we also prove in Proposition \ref{prop:HK-existence} that
``for \emph{$\functionspace{E}$-quasi-every} $x\in V$" in the above
statement can be improved to ``for \emph{any} $x\in V$" if the inequality
$\mathbb{P}_{x}[X_{t}\in dy,\,t<\tau_{U}]\leq H_{t}(x,y)\,d\mu(y)$
holds on $W$ for \emph{any} $(t,x)\in I\times V$.

Next we turn to our second main theorem on localized sub-Gaussian upper bounds
of heat kernels (Theorem \ref{thm:HKUB-localized}). For the reader's convenience,
we give here the precise statement of a simplified version of it. For $B\subset M$,
set $\tau_{B}:=\inf\{t\in[0,\infty)\mid X_{t}\in M_{\cemetery}\setminus B\}$
($\inf\emptyset:=\infty$) and let $\sigmafield{B}(B)$ denote its Borel
$\sigma$-field under the relative topology inherited from $M$.

\begin{theorem}\label{thm:HKUB-localized-intro}
Let $(M,d)$ be a locally compact separable metric space, let $\mu$ be a
$\sigma$-finite Borel measure $\mu$ on $M$ and let
$X=\bigl(\Omega,\sigmafield{M},\{X_{t}\}_{t\in[0,\infty]},
	\{\mathbb{P}_{x}\}_{x\in M_{\cemetery}}\bigr)$
be a Hunt process on $(M,\sigmafield{B}(M))$ with life time $\zeta$.
Let $N\in\sigmafield{B}(M)$ and assume that for any $x\in M\setminus N$,
\begin{equation}\label{eq:properly-exceptional-X-continuous-intro}
\mathbb{P}_{x}\bigl[X_{t}\in M_{\cemetery}\setminus N\textrm{ for any }t\in[0,\infty),\,
	[0,\zeta)\ni t\mapsto X_{t}\in M\textrm{ is continuous}\bigr]=1
\end{equation}
\textup{(namely, $M\setminus N$ is $X$-invariant and the restriction
$X|_{M\setminus N}$ of $X$ to $M\setminus N$ is a diffusion)}.

Let $\beta\in(1,\infty)$, let $R\in(0,\infty)$, let $U$ be a non-empty open subset of $M$
with $\diam U\leq R$ and let $F=F_{t}(x,y):(0,R^{\beta}]\times U\times U\to(0,\infty)$
be Borel measurable. Let $c_{F},\alpha_{F},c,\gamma\in(0,\infty)$
and assume that the following three conditions \textup{(DB)$_{\beta}$},
\textup{(DU)$_{F}^{U,R}$} and \textup{(P)$_{\beta}^{U,R}$} hold:
\begin{itemize}[label=\textup{(DU)$_{F}^{U,R}$},align=left,leftmargin=*]
\item[\textup{(DB)$_{\beta}$}]
	For any $(t,x,y),(s,z,w)\in(0,R^{\beta}]\times U\times U$ with $s\leq t$,
	\begin{equation}\label{eq:upper-bound-function-F-DBbeta}
	\frac{F_{s}(z,w)}{F_{t}(x,y)}
		\leq c_{F}\biggl(\frac{t\vee d(x,z)^{\beta}\vee d(y,w)^{\beta}}{s}\biggr)^{\alpha_{F}}.
	\end{equation}
\item[\textup{(DU)$_{F}^{U,R}$}]
	For any $(t,x)\in(0,R^{\beta})\times(U\setminus N)$ and any $A\in\sigmafield{B}(U)$,
	\begin{equation}\label{eq:HKUB-localized-on-diag-intro}
	\mathbb{P}_{x}[X_{t}\in A,\,t<\tau_{U}]\leq\int_{A}F_{t}(x,y)\,d\mu(y).
	\end{equation}
\item[\textup{(P)$_{\beta}^{U,R}$}]
	For any $(x,r)\in(U\setminus N)\times(0,R)$ with $B(x,r)\subset U$
	and any $t\in(0,\infty)$,
	\begin{equation}\label{eq:exit-probability-intro}
	\mathbb{P}_{x}[\tau_{B(x,r)}\leq t]
		\leq c\exp\bigl(-\gamma(r^{\beta}/t)^{\frac{1}{\beta-1}}\bigr).
	\end{equation}
\end{itemize}
Let $\varepsilon\in(0,1)$ and set
$U^{\circ}_{\varepsilon R}:=\{x\in M\mid\inf_{y\in M\setminus U}d(x,y)>\varepsilon R\}$
\textup{(note that $U^{\circ}_{\varepsilon R}$ is an open subset of $U$)}.
Then there exists a Borel measurable function
$p=p_{t}(x,y):(0,\infty)\times(M\setminus N)\times U^{\circ}_{\varepsilon R}\to[0,\infty)$
such that for any $(t,x)\in(0,\infty)\times(M\setminus N)$ the following hold:
\begin{equation}\label{eq:HKUB-localized-existence-intro}
\mathbb{P}_{x}[X_{t}\in A]=\int_{A}p_{t}(x,y)\,d\mu(y)
	\qquad\textrm{for any }A\in\sigmafield{B}(U^{\circ}_{\varepsilon R}),
\end{equation}
and furthermore for any $y\in U^{\circ}_{\varepsilon R}$,
\begin{equation}\label{eq:HKUB-localized-off-diag-U-intro}
p_{t}(x,y)\leq
	\begin{cases}
	c_{\varepsilon}F_{t}(x,y)\exp\bigl(-\gamma_{\varepsilon}(d(x,y)^{\beta}/t)^{\frac{1}{\beta-1}}\bigr) &\textrm{if }t<R^{\beta}\textrm{ and }x\in U,\\
	c_{\varepsilon}(\inf_{U\times U}F_{(2t)\wedge R^{\beta}})\exp\bigl(-\gamma_{\varepsilon}(R^{\beta}/t)^{\frac{1}{\beta-1}}\bigr) &\textrm{if }t<R^{\beta}\textrm{ and }x\not\in U,\\
	c_{\varepsilon}(\inf_{U\times U}F_{R^{\beta}}) &\textrm{if }t\geq R^{\beta}
	\end{cases}
\end{equation}
for some $c_{\varepsilon}\in(0,\infty)$ explicit in
$\beta,c_{F},\alpha_{F},c,\gamma,\varepsilon$ and
$\gamma_{\varepsilon}:=(\frac{1}{5}\varepsilon)^{\frac{\beta}{\beta-1}}\gamma$.
\end{theorem}

The strength of Theorem \ref{thm:HKUB-localized-intro} is that 
\emph{the conditions \textup{(DU)$_{F}^{U,R}$} and \textup{(P)$_{\beta}^{U,R}$}
are independent of the behavior of $X$ after exiting $U$ and thereby completely
localized within $U$} but assure nevertheless the existence and an upper bound
of the heat kernel $p=p_{t}(x,y)$ for the \emph{global} transition function
$\mathbb{P}_{x}[X_{t}\in dy]$.

The power function $F_{t}(x,y)=c_{0}t^{-\alpha}$ clearly satisfies \textup{(DB)$_{\beta}$},
and it is easy to see that \textup{(DB)$_{\beta}$} holds also for the volume function
\eqref{eq:F-vol-intro} provided \eqref{eq:VD-intro} is satisfied for any
$(x,r)\in U\times(0,R)$; see Example \ref{exmp:upper-bound-function} for some more details.
In view of these examples of $F=F_{t}(x,y)$, \textup{(DU)$_{F}^{U,R}$} amounts
to an \emph{on-diagonal} upper bound of the heat kernel $p^{U}=p^{U}_{t}(x,y)$ for
$\{T^{U}_{t}\}_{t\in(0,\infty)}$, which is known to be implied in the setting of
a regular symmetric Dirichlet form by the \emph{local Nash inequality} as shown
in \cite[Lemma 4.3]{Kig:localNash} and by the \emph{Faber-Krahn inequality}
as treated in \cite[Subsection 5.2 and (5.48)]{GriHu:Upper}.

The proof of Theorem \ref{thm:HKUB-localized-intro} relies essentially only on
two probabilistic iteration arguments based on the strong Markov property of $X$,
where the series in the resulting upper estimates are shown to converge to the
desired bounds by making heavy use of the condition \textup{(P)$_{\beta}^{U,R}$}.
In this sense, \emph{\textup{(P)$_{\beta}^{U,R}$} could be considered as the
probabilistic replacement for cutoff functions with well-controlled energy}.
One iteration argument involves the behavior of $X$ within $U$ alone and is
used in the first step of the proof of Theorem \ref{thm:HKUB-localized-intro} to
obtain an off-diagonal sub-Gaussian type upper bound of the Dirichlet heat kernel
$p^{U}=p^{U}_{t}(x,y)$ on $U$ \emph{without assuming the symmetry of $X$}
(Proposition \ref{prop:HKUB-localized}). The other iteration is
formulated as an equality, which we call a \emph{multiple Dynkin-Hunt formula},
expressing the global transition function $\mathbb{P}_{x}[X_{t}\in A]$
in terms of $\mathbb{P}_{y}[X_{s}\in A,\,s<\tau_{U}]$, $(s,y)\in[0,t]\times U$,
for each Borel subset $A$ of $M$ with $\overline{A}\subset U$
(Theorem \ref{thm:multiple-DH}) and thus enabling us to deduce upper bounds for
the former from those for the latter together with \textup{(P)$_{\beta}^{U,R}$}
(Proposition \ref{prop:HKUB-localized-difference}).

Note that the case of bounded $(M,d)$ has been excluded from the main results
of \cite{AB,Gri:HKfractal,GriHu:HKGreen,GriHu:Upper,GriHuLau:GenCap,GT},
mainly due to their construction of the global heat kernel $p_{t}(x,y)$ as
the limit as $U\uparrow M$ of the Dirichlet heat kernel $p^{U}_{t}(x,y)$ on $U$;
indeed, taking the limit as $U\uparrow M$ is not allowed for bounded $(M,d)$
since part of their conditions \textup{(FK)$_{\Psi}$} (Faber-Krahn inequality)
and \textup{(E)$_{\Psi}$} (mean exit time estimate, see
\eqref{eq:exit-time-upper} and \eqref{eq:exit-time-lower}
in Theorem \ref{thm:exit-time-probability} below)
must fail when the ball $B(x,r)$ coincides with $M$.
We expect that this difficulty can be overcome by applying the main results of this
paper, so that their results should be easily extended to the case of bounded $(M,d)$.
In fact, Barlow, Bass, Kumagai and Teplyaev \cite{BBKT:supplement} have used
an argument very similar to our proof of Theorem \ref{thm:multiple-DH}
and Proposition \ref{prop:HKUB-localized-difference}
in \cite[Proof of Proposition 2.12]{BBKT:supplement}
for the resolvent of the diffusion to extend part of
the main results of \cite{GriHu:HKGreen,GT} to the case of bounded $(M,d)$.
Our proof of Theorem \ref{thm:HKUB-localized-intro} has successfully localized
their idea by working directly with the transition function (semigroup)
rather than the resolvent.

Finally, we remark that Theorem \ref{thm:HKUB-localized-intro} has been recently
applied in \cite{AK} to prove the continuity and sub-Gaussian off-diagonal
upper bounds of the heat kernel of the \emph{Liouville Brownian motion}, the
canonical diffusion in a certain random geometry of $\mathbb{R}^{2}$ induced by
a (massive) Gaussian free field. These results in \cite{AK} have had to rely strongly
on Theorem \ref{thm:HKUB-localized-intro} due to the fact that the unboundedness
of $\mathbb{R}^{2}$ precludes any uniform estimates of volumes and exit times
over the whole $\mathbb{R}^{2}$ valid for almost every environment, as opposed
to the case of the two-dimensional torus, where the same kind of results have
been obtained independently and simultaneously in \cite{MRVZ}.

The rest of this paper is organized as follows.
In Section \ref{sec:Hunt-processes}, we collect basic definitions and facts
concerning Hunt processes. Section \ref{sec:multiple-DH} formulates one of our
two iteration arguments as a multiple Dynkin-Hunt formula and proves it for an
\emph{arbitrary} Hunt process (Theorem \ref{thm:multiple-DH}).
In Section \ref{sec:Dirichlet-form}, we recall the notions of the symmetry of
a Hunt process, the associated symmetric Dirichlet form and its regularity,
together with some basic potential theory that is needed in
Section \ref{sec:HK-existence} to state and prove our main localized
existence theorem for the heat kernel (Theorem \ref{thm:HK-existence}).
In Section \ref{sec:HKUB} we state our main theorem on localized sub-Gaussian
upper bounds of heat kernels (Theorem \ref{thm:HKUB-localized}) and a global version
of it (Theorem \ref{thm:HKUB-localized-global}) and prove them on the basis of
our other probabilistic iteration (Proposition \ref{prop:HKUB-localized}) and
the multiple Dynkin-Hunt formula combined with the condition
\textup{(P)$_{\beta}^{U,R}$} (Proposition \ref{prop:HKUB-localized-difference}).
Lastly, Section \ref{sec:exit-probability} is devoted to providing sufficient
conditions for \textup{(P)$_{\beta}^{U,R}$}
(Theorems \ref{thm:exit-probability} and \ref{thm:exit-time-probability}) as a
localized version of the (well-)known results in \cite{Bar,Gri:HKfractal,GriHu:Upper}.

\begin{notation}
In this paper, we adopt the following notation and conventions.
\begin{itemize}[label=\textup{(0)},align=left,leftmargin=*]
\item[\textup{(0)}]The symbols $\subset$ and $\supset$ for set inclusion
\emph{allow} the case of the equality.
\item[\textup{(1)}]$\mathbb{N}=\{n\in\mathbb{Z}\mid n>0\}$, i.e., $0\not\in\mathbb{N}$.
\item[\textup{(2)}]We set $\sup\emptyset:=0$ and $\inf\emptyset:=\infty$. We write
$a\vee b:=\max\{a,b\}$, $a\wedge b:=\min\{a,b\}$, $a^{+}:=a\vee 0$ and
$a^{-}:=-(a\wedge 0)$ for $a,b\in[-\infty,\infty]$, and
we use the same notation also for $[-\infty,\infty]$-valued functions
and equivalence classes of them. All numerical functions treated in this paper
are assumed to be $[-\infty,\infty]$-valued.
\item[\textup{(3)}]Let $E$ be a topological space.
The Borel $\sigma$-field of $E$ is denoted by $\sigmafield{B}(E)$.
We set
\begin{align*}
C(E)&:=\{u\mid u:E\to\mathbb{R},\,u\textrm{ is continuous}\},\\
C_{\mathrm{c}}(E)&:=\{u\in C(E)
	\mid\textrm{the closure of }u^{-1}(\mathbb{R}\setminus\{0\})\textrm{ in }E\textrm{ is compact}\},\\
\functionspace{B}(E)&:=\{u\mid u:E\to[-\infty,\infty],\,u\textrm{ is Borel measurable (i.e.,\ }\sigmafield{B}(E)\textrm{-measurable)}\},\\
\functionspace{B}^{+}(E)&:=\{u\in\functionspace{B}(E)\mid u\textrm{ is }[0,\infty]\textrm{-valued}\},\\
\functionspace{B}_{\mathrm{b}}(E)&:=\{u\in\functionspace{B}(E)\mid\|u\|_{\mathrm{sup}}<\infty\},
\end{align*}
where $\|u\|_{\mathrm{sup}}:=\|u\|_{\mathrm{sup},E}:=\sup_{x\in E}|u(x)|$ for $u:E\to[-\infty,\infty]$.
\end{itemize}
\end{notation}

\section{Basics on Hunt processes}\label{sec:Hunt-processes}

In this section, we introduce our framework of a Hunt process. To keep the main
results of this paper accessible to those who are not familiar with the theory
of Markov processes, we explain basic definitions and facts in some detail. See
\cite[Section A.2]{FOT} and \cite[Section A.1]{CF} for further details on Hunt processes.

Let $M$ be a locally compact separable metrizable topological space.
The interior, closure and boundary of $A\subset M$ in $M$ are denoted by
$\interior A$, $\overline{A}$ and $\partial A$, respectively.
Each $A\subset M$ is equipped with the relative topology inherited from $M$,
so that its Borel $\sigma$-field $\sigmafield{B}(A)$ can be expressed as
$\sigmafield{B}(A)=\{B\cap A\mid B\in\sigmafield{B}(M)\}$. Let
$M_{\cemetery}:=M\cup\{\cemetery\}$ denote the one-point compactification of $M$,
which satisfies 
$\sigmafield{B}(M_{\cemetery})=\sigmafield{B}(M)\cup\{A\cup\{\cemetery\}\mid A\in\sigmafield{B}(M)\}$.
In what follows, $[-\infty,\infty]$-valued functions on $M$ are always set
to be $0$ at $\cemetery$ unless their values at $\Delta$ are already defined:
$u(\cemetery):=0$ for $u:M\to[-\infty,\infty]$.

Let
$X=\bigl(\Omega,\sigmafield{M},\{X_{t}\}_{t\in[0,\infty]},
	\{\mathbb{P}_{x}\}_{x\in M_{\cemetery}}\bigr)$
be a Hunt process on $(M,\sigmafield{B}(M))$ with life time $\zeta$ and
shift operators $\{\theta_{t}\}_{t\in[0,\infty]}$.
By definition, $(\Omega,\sigmafield{M})$ is a measurable space,
$\{X_{t}\}_{t\in[0,\infty]}$ is a family of
$\sigmafield{M}/\sigmafield{B}(M_{\cemetery})$-measurable maps
$X_{t}:\Omega\to M_{\cemetery}$ such that
$X_{t}(\omega)=\cemetery$ for any $t\in[\zeta(\omega),\infty]$
for each $\omega\in\Omega$,
where $\zeta(\omega):=\inf\{t\in[0,\infty)\mid X_{t}(\omega)=\cemetery\}$,
and $\{\theta_{t}\}_{t\in[0,\infty]}$ is a family of maps
$\theta_{t}:\Omega\to\Omega$ satisfying
$X_{s}\circ\theta_{t}=X_{s+t}$ for any $s,t\in[0,\infty]$.
It is further assumed that for each $\omega\in\Omega$,
$[0,\infty)\ni t\mapsto X_{t}(\omega)\in M_{\cemetery}$ is right-continuous and
the limit $X_{t-}(\omega):=\lim_{s\to t,\,s<t}X_{s}(\omega)$ exists in $M_{\cemetery}$
for any $t\in(0,\infty)$; see \cite[Section A.2, (M.6)]{FOT}. The pair $X$
of such a stochastic process $\bigl(\Omega,\sigmafield{M},\{X_{t}\}_{t\in[0,\infty]}\bigr)$
and a family $\{\mathbb{P}_{x}\}_{x\in M_{\cemetery}}$ of probability measures
on $(\Omega,\sigmafield{M})$ is then called a
\emph{Hunt process on $(M,\sigmafield{B}(M))$}
if and only if it is a normal Markov process on
$(M,\sigmafield{B}(M))$ whose minimum completed admissible filtration
$\sigmafield{F}_{*}=\{\sigmafield{F}_{t}\}_{t\in[0,\infty]}$ is \emph{right-continuous}
and it is \emph{strong Markov} and \emph{quasi-left-continuous} with respect to
$\sigmafield{F}_{*}$; see
\cite[Section A.2, (M.2)--(M.5), the paragraph before Lemma A.2.2, (A.2.3) and (A.2.4)]{FOT}
for the precise definitions of these notions.

For $x\in M_{\cemetery}$, the expectation (that is, the integration on $\Omega$)
under the measure $\mathbb{P}_{x}$ is denoted by $\mathbb{E}_{x}[(\cdot)]$.
We remark that by \cite[Exercise A.1.20-(i)]{CF}, for each
$\sigmafield{F}_{\infty}$-measurable random variable $Y:\Omega\to[0,\infty]$
the function $M_{\cemetery}\ni x\mapsto\mathbb{E}_{x}[Y]\in[0,\infty]$ is
\emph{universally measurable}, i.e., measurable with respect to the
\emph{universal $\sigma$-field $\sigmafield{B}^{*}(M_{\cemetery})$ of $M_{\cemetery}$} defined as
$\sigmafield{B}^{*}(M_{\cemetery}):=\bigcap_{\nu}\sigmafield{B}^{\nu}(M_{\cemetery})$;
here $\nu$ runs through the set of probability (or equivalently, $\sigma$-finite)
measures on $(M_{\cemetery},\sigmafield{B}(M_{\cemetery}))$ and
$\sigmafield{B}^{\nu}(M_{\cemetery})$ denotes the $\nu$-completion
of $\sigmafield{B}(M_{\cemetery})$.

The Hunt process $X$ gives rise to a family $\{\functionspace{P}_{t}\}_{t\in[0,\infty)}$
of Markovian kernels on $(M,\sigmafield{B}(M))$ called the
\emph{transition function of $X$}, which is defined by
\begin{equation}\label{eq:transition-function}
\functionspace{P}_{t}(x,A):=\mathbb{P}_{x}[X_{t}\in A],
	\qquad t\in[0,\infty),\ x\in M,\ A\in\sigmafield{B}(M).
\end{equation}
Then for $t\in[0,\infty)$ and $u\in\functionspace{B}(M)$, we define
\begin{equation}\label{eq:PtEx}
\functionspace{P}_{t}u(x):=\int_{M}u(y)\functionspace{P}_{t}(x,dy)=\mathbb{E}_{x}[u(X_{t})]
\end{equation}%
for $x\in M$ satisfying $\mathbb{E}_{x}[u^{+}(X_{t})]\wedge\mathbb{E}_{x}[u^{-}(X_{t})]<\infty$,
so that $\functionspace{P}_{t}(\functionspace{B}^{+}(M))\subset\functionspace{B}^{+}(M)$ and
$\functionspace{P}_{t}(\functionspace{B}_{\mathrm{b}}(M))\subset\functionspace{B}_{\mathrm{b}}(M)$.
Note that our convention of setting $\functionspace{P}_{t}u(\cemetery):=0$ is
consistent with \eqref{eq:PtEx} for $x=\cemetery$ since
$\mathbb{E}_{\cemetery}[u(X_{t})]=\mathbb{E}_{\cemetery}[u(\cemetery)]=0$
by $\mathbb{P}_{\cemetery}[X_{t}=\cemetery]=1$.
Obviously, if $u\in\functionspace{B}(M)$ is $[0,1]$-valued then so is
$\functionspace{P}_{t}u$, and the Markov property of $X$
(see \cite[(A.2.2)]{FOT} or \cite[(A.1.3)]{CF}) easily implies the
semigroup property
\begin{equation}\label{eq:Pt-semigroup}
\functionspace{P}_{t}\functionspace{P}_{s}u=\functionspace{P}_{t+s}u,
	\qquad t,s\in[0,\infty),\ u\in\functionspace{B}^{+}(M)\cup\functionspace{B}_{\mathrm{b}}(M).
\end{equation}
Moreover,
it easily follows from the sample path right-continuity of $X$ and
the Dynkin class theorem \cite[Chapter 0, Theorem 2.2]{BG} that
\begin{equation}\label{eq:Pt-measurable}
[0,\infty)\times M\ni(t,x)\mapsto\functionspace{P}_{t}u(x)
\textrm{ is Borel measurable for any }u\in\functionspace{B}^{+}(M)\cup\functionspace{B}_{\mathrm{b}}(M).
\end{equation}

Recall that $\sigma:\Omega\to[0,\infty]$ is called an
\emph{$\sigmafield{F}_{*}$-stopping time} if and only if
$\{\sigma\leq t\}\in\sigmafield{F}_{t}$ for any $t\in[0,\infty)$.
For $B\subset M_{\cemetery}$, we define its \emph{entrance time $\dot{\sigma}_{B}$}
and \emph{exit time $\tau_{B}$ for $X$} by
\begin{equation}\label{eq:entrance-exit}
\dot{\sigma}_{B}(\omega):=\inf\{t\in[0,\infty)\mid X_{t}(\omega)\in B\},
	\quad\omega\in\Omega\qquad\textrm{and}\qquad
	\tau_{B}:=\dot{\sigma}_{M_{\cemetery}\setminus B},
\end{equation}
and we also set $\hat{\sigma}_{B}(\omega):=\inf\{t\in(0,\infty)\mid X_{t-}(\omega)\in B\}$
for $\omega\in\Omega$. If $B\in\sigmafield{B}(M_{\cemetery})$, then
$\dot{\sigma}_{B},\tau_{B},\hat{\sigma}_{B}$ are
$\sigmafield{F}_{*}$-stopping times and
\begin{equation}\label{eq:entrance-leq-entrance-left-limit}
\mathbb{P}_{x}[\dot{\sigma}_{B}\leq\hat{\sigma}_{B}]=1
	\qquad\textrm{for any }x\in M_{\cemetery}
\end{equation}
by \cite[Theorem A.2.3]{FOT}, where the case of $\cemetery\in B$ is easily
deduced from that of $B\in\sigmafield{B}(M)$ by using the equalities
$\dot{\sigma}_{B\cup\{\cemetery\}}=\dot{\sigma}_{B}\wedge\zeta$ and
$\hat{\sigma}_{B\cup\{\cemetery\}}=\hat{\sigma}_{B}\wedge\hat{\sigma}_{\{\cemetery\}}$
for $B\subset M$ and the quasi-left-continuity \cite[(A.2.4)]{FOT} of $X$
(see also \cite[Theorem A.1.19 and Exercise A.1.26-(ii)]{CF}).
Note that if $B\subset M_{\cemetery}$, $t\in[0,\infty]$ and
$\omega\in\{\dot{\sigma}_{B}\geq t\}$ then
$\dot{\sigma}_{B}(\omega)=t+\dot{\sigma}_{B}(\theta_{t}(\omega))$.

Next we introduce the part of $X$ on open sets. Let $U$ be a non-empty open
subset of $M$, let $U_{\cemetery}:=U\cup\{\cemetery_{U}\}$ denote its one-point
compactification and define
\begin{equation}\label{eq:part-process}
X^{U}_{t}(\omega):=
	\begin{cases}
	X_{t}(\omega)&\textrm{if }t<\tau_{U}(\omega),\\
	\cemetery_{U}&\textrm{if }t\geq\tau_{U}(\omega),\\
	\end{cases}
\qquad(t,\omega)\in[0,\infty]\times\Omega
\end{equation}
and $\mathbb{P}_{\cemetery_{U}}:=\mathbb{P}_{\cemetery}$. Then
$X^{U}:=\bigl(\Omega,\sigmafield{M},\{X^{U}_{t}\}_{t\in[0,\infty]},\{\mathbb{P}_{x}\}_{x\in U_{\cemetery}}\bigr)$,
called the \emph{part of $X$ on $U$}, is a Hunt process on $(U,\sigmafield{B}(U))$
by \cite[Theorem A.2.10]{FOT}. Its transition function is naturally extended to
$(M,\sigmafield{B}(M))$ as a family $\{\functionspace{P}^{U}_{t}\}_{t\in[0,\infty)}$
of Markovian kernels on $(M,\sigmafield{B}(M))$ given by
(with the obvious convention that $\cemetery_{U}\not\in M$)
\begin{equation}\label{eq:transition-function-part}
\functionspace{P}^{U}_{t}(x,A):=\mathbb{P}_{x}[X^{U}_{t}\in A]
	=\mathbb{P}_{x}[X_{t}\in A,\,t<\tau_{U}],
	\quad t\in[0,\infty),\ x\in M,\ A\in\sigmafield{B}(M).
\end{equation}
Also for $t\in[0,\infty)$ and $u\in\functionspace{B}(M)$,
similarly to \eqref{eq:PtEx} we further define
\begin{equation}\label{eq:PtEx-part}
\functionspace{P}^{U}_{t}u(x):=\int_{M}u(y)\functionspace{P}^{U}_{t}(x,dy)
	=\int_{U}u(y)\functionspace{P}^{U}_{t}(x,dy)
	=\mathbb{E}_{x}[u(X_{t})\ind{\{t<\tau_{U}\}}]
\end{equation}
for $x\in M$ satisfying
$\mathbb{E}_{x}[u^{+}(X_{t})\ind{\{t<\tau_{U}\}}]\wedge\mathbb{E}_{x}[u^{-}(X_{t})\ind{\{t<\tau_{U}\}}]<\infty$,
where $\ind{A}:\Omega\to\{0,1\}$ denotes the indicator function of $A\subset\Omega$
given by $\ind{A}|_{A}:=1$ and $\ind{A}|_{\Omega\setminus A}:=0$. Then
$\functionspace{P}^{U}_{t}u(x)=0$ for $x\in M\setminus U$,
$\functionspace{P}^{U}_{t}(\functionspace{B}^{+}(M))\subset\functionspace{B}^{+}(M)$,
$\functionspace{P}^{U}_{t}(\functionspace{B}_{\mathrm{b}}(M))\subset\functionspace{B}_{\mathrm{b}}(M)$,
and \eqref{eq:Pt-semigroup} and \eqref{eq:Pt-measurable} hold
with $\{\functionspace{P}^{U}_{t}\}_{t\in[0,\infty)}$
in place of $\{\functionspace{P}_{t}\}_{t\in[0,\infty)}$.

\section{A multiple Dynkin-Hunt formula for Hunt processes}\label{sec:multiple-DH}

As in Section \ref{sec:Hunt-processes}, let $M$ be a locally compact separable
metrizable topological space and let $X$ be a Hunt process on $(M,\sigmafield{B}(M))$
with life time $\zeta$ and shift operators $\{\theta_{t}\}_{t\in[0,\infty]}$.
\emph{Throughout the rest of this paper, we fix this setting and follow
the notation introduced in Section \textup{\ref{sec:Hunt-processes}}}.

In this section, we state and prove a \emph{multiple Dynkin-Hunt formula}
(Theorem \ref{thm:multiple-DH} below) which gives an expression of
$\functionspace{P}_{t}u$ in terms of $\functionspace{P}^{U}_{s}u$, $s\in[0,t]$,
for a non-empty open subset $U$ of $M$ and functions
$u\in\functionspace{B}^{+}(M)\cup\functionspace{B}_{\mathrm{b}}(M)$
supported in $U$. It will be used later in Section \ref{sec:HKUB}
to deduce upper bounds for $\{\functionspace{P}_{t}\}_{t\in(0,\infty)}$
from those for $\{\functionspace{P}^{U}_{t}\}_{t\in(0,\infty)}$.

The statement of Theorem \ref{thm:multiple-DH} requires the following definition and proposition.

\begin{definition}\label{dfn:entrance-exit-after-sigma}
For $\sigma:\Omega\to[0,\infty]$ and $B\subset M_{\cemetery}$,
the \emph{entrance time $\dot{\sigma}_{B,\sigma}$} and
\emph{exit time $\tau_{B,\sigma}$ of $B$ after $\sigma$ for $X$}
are defined by (with the convention that $[\infty,\infty):=\emptyset$)
\begin{equation}\label{eq:entrance-exit-after-sigma}
\dot{\sigma}_{B,\sigma}(\omega):=\inf\{t\in[\sigma(\omega),\infty)\mid X_{t}(\omega)\in B\},
	\quad\omega\in\Omega\quad\textrm{and}\quad
	\tau_{B,\sigma}:=\dot{\sigma}_{M_{\cemetery}\setminus B,\sigma},
\end{equation}
so that
$\dot{\sigma}_{B,\sigma}(\omega)=\sigma(\omega)+\dot{\sigma}_{B}(\theta_{\sigma(\omega)}(\omega))$
and
$\tau_{B,\sigma}(\omega)=\sigma(\omega)+\tau_{B}(\theta_{\sigma(\omega)}(\omega))$
for any $\omega\in\Omega$.
\end{definition}

\begin{proposition}\label{prop:entrance-exit-after-sigma}
For any $\sigmafield{F}_{*}$-stopping time $\sigma$ and any
$B\in\sigmafield{B}(M_{\cemetery})$,
the entrance time $\dot{\sigma}_{B,\sigma}$ and exit time $\tau_{B,\sigma}$ of
$B$ after $\sigma$ for $X$ are $\sigmafield{F}_{*}$-stopping times.
\end{proposition}

\begin{proof}
This proposition should be well-known, but we give an explicit proof for completeness.
We follow \cite[Proof of Theorem A.1.19]{CF}. For each $t\in(0,\infty)$, the set
$\{\dot{\sigma}_{B,\sigma}<t\}=\{\omega\in\Omega\mid\dot{\sigma}_{B,\sigma}(\omega)<t\}$
is equal to the projection on $\Omega$ of
\begin{equation*}
\{(s,\omega)\in[0,t)\times\Omega\mid\sigma(\omega)\leq s,\,X_{s}(\omega)\in B\},
\end{equation*}
which is easily shown to belong to the product $\sigma$-field
$\sigmafield{B}([0,t])\otimes\sigmafield{F}_{t}$ by the sample path
right-continuity of $X$ and the assumption that
$\sigma$ is an $\sigmafield{F}_{*}$-stopping time.
Therefore \cite[Chapter III, 13 and 33]{DM} imply that
$\{\dot{\sigma}_{B,\sigma}<t\}\in\sigmafield{F}_{t}$,
which means that $\dot{\sigma}_{B,\sigma}$, and hence also
$\tau_{B,\sigma}$, are $\sigmafield{F}_{*}$-stopping times
since $\sigmafield{F}_{*}$ is right-continuous.
\end{proof}

Now we state the main theorem of this section.
Recall for $\sigma:\Omega\to[0,\infty]$ that the map
$X_{\sigma}:\Omega\to M_{\cemetery}$ is defined as
$X_{\sigma}(\omega):=X_{\sigma(\omega)}(\omega)$ and that $X_{\sigma}$ is
$\sigmafield{F}_{\infty}/\sigmafield{B}(M_{\cemetery})$-measurable
if $\sigma$ is $\sigmafield{F}_{\infty}$-measurable
by the sample path right-continuity of $X$.

\begin{theorem}[A multiple Dynkin-Hunt formula]\label{thm:multiple-DH}
Let $U$ be a non-empty open subset of $M$, let $B\in\sigmafield{B}(M)$ satisfy
$\overline{B}\subset U$ and define $\sigmafield{F}_{*}$-stopping times
$\tau_{n}$ and $\sigma_{n}$, $n\in\mathbb{N}$, by
\begin{equation}\label{eq:multiple-DH-stopping-times}
\tau_{1}:=\tau_{U}\qquad\textrm{and inductively}\qquad
\sigma_{n}:=\dot{\sigma}_{B,\tau_{n}}\quad\textrm{and}\quad
\tau_{n+1}:=\tau_{U,\sigma_{n}},\quad n\in\mathbb{N}.
\end{equation}
Then for any $u\in\functionspace{B}^{+}(M)\cup\functionspace{B}_{\mathrm{b}}(M)$
with $u|_{M\setminus B}=0$ and any $(t,x)\in[0,\infty)\times M$,
\begin{equation}\label{eq:multiple-DH}
\functionspace{P}_{t}u(x)=\functionspace{P}^{U}_{t}u(x)
	+\sum_{n\in\mathbb{N}}\mathbb{E}_{x}\bigl[\ind{\{\sigma_{n}\leq t\}}
		\functionspace{P}^{U}_{t-\sigma_{n}}u(X_{\sigma_{n}})\bigr].
\end{equation}
\end{theorem}

Note that by \eqref{eq:Pt-measurable} for
$\{\functionspace{P}^{U}_{t}\}_{t\in[0,\infty)}$, the random variable
$\ind{\{\sigma_{n}\leq t\}}\functionspace{P}^{U}_{t-\sigma_{n}}u(X_{\sigma_{n}})$
in \eqref{eq:multiple-DH} is $\sigmafield{F}_{\infty}$-measurable
for any $u\in\functionspace{B}^{+}(M)\cup\functionspace{B}_{\mathrm{b}}(M)$,
any $t\in[0,\infty)$ and any $n\in\mathbb{N}$.

Recall that the \emph{Dynkin-Hunt formula} refers to
(the heat kernel version of) the following equality,
which is an easy consequence of Proposition \ref{prop:strong-Markov} below:
for any non-empty open subset $U$ of $M$,
any $u\in\functionspace{B}^{+}(M)\cup\functionspace{B}_{\mathrm{b}}(M)$
and any $(t,x)\in[0,\infty)\times M$,
\begin{equation}\label{eq:DH}
\functionspace{P}_{t}u(x)=\functionspace{P}^{U}_{t}u(x)
	+\mathbb{E}_{x}\bigl[\ind{\{\tau_{U}\leq t\}}
		\functionspace{P}_{t-\tau_{U}}u(X_{\tau_{U}})\bigr].
\end{equation}
\eqref{eq:multiple-DH} can be regarded as an indefinite iteration of
\eqref{eq:DH} through restarting $X$ at the entrance time
$\dot{\sigma}_{B,\tau_{U}}$ of $B$ after $\tau_{U}$, which is why we call
\eqref{eq:multiple-DH} a \emph{multiple Dynkin-Hunt formula}.

For the proof of Theorem \ref{thm:multiple-DH} we need a variation of
the strong Markov property of $X$ as in the following proposition.
Recall for each $\sigmafield{F}_{*}$-stopping time $\sigma$ that the collection
\begin{equation}\label{eq:Fsigma}
\sigmafield{F}_{\sigma}:=\{A\in\sigmafield{F}_{\infty}\mid
	A\cap\{\sigma\leq t\}\in\sigmafield{F}_{t}\textrm{ for any }t\in[0,\infty)\}
\end{equation}
is a $\sigma$-field in $\Omega$ with respect to which $\sigma$ is measurable,
that $X_{\sigma}$ is $\sigmafield{F}_{\sigma}/\sigmafield{B}^{*}(M_{\cemetery})$-measurable
by \cite[Exercise A.1.20-(ii)]{CF}, and that the map $\theta_{\sigma}:\Omega\to\Omega$,
$\theta_{\sigma}(\omega):=\theta_{\sigma(\omega)}(\omega)$, is
$\sigmafield{F}_{\infty}/\sigmafield{F}_{\infty}$-measurable by \cite[Theorem A.1.21]{CF}.

\begin{proposition}\label{prop:strong-Markov}
Let $\sigma$ be an $\sigmafield{F}_{*}$-stopping time, let $\tau:\Omega\to[0,\infty]$
be $\sigmafield{F}_{\infty}$-measurable and let $T:\Omega\to[0,\infty]$ be
$\sigmafield{F}_{\sigma}$-measurable and satisfy $\sigma(\omega)\leq T(\omega)$
for any $\omega\in\Omega$. Then for any $x\in M_{\cemetery}$ and any
$u\in\functionspace{B}_{\mathrm{b}}(M_{\cemetery})$,
it holds that for $\mathbb{P}_{x}$-a.e.\ $\omega\in\{\sigma<\infty\}$,
\begin{equation}\label{eq:strong-Markov}
\mathbb{E}_{x}\bigl[u(X_{T})\ind{\{T<\sigma+\tau\circ\theta_{\sigma}\}}\bigm|\sigmafield{F}_{\sigma}\bigr](\omega)
	=\mathbb{E}_{X_{\sigma}(\omega)}\bigl[u(X_{T(\omega)-\sigma(\omega)})\ind{\{T(\omega)-\sigma(\omega)<\tau\}}\bigr].
\end{equation}
\end{proposition}

\begin{proof}
We follow \cite[Proofs of Proposition 2.6.17 and Corollary 2.6.18]{KS}.
For $u\in\sigmafield{B}_{\mathrm{b}}(M_{\cemetery})$, let
$Y_{u}(\omega)$ denote the right-hand side of \eqref{eq:strong-Markov} for
$\omega\in\{\sigma<\infty\}$ and set $Y_{u}(\omega):=0$ for $\omega\in\{\sigma=\infty\}$.
Let $x\in M_{\cemetery}$. For the proof of \eqref{eq:strong-Markov} it suffices
to show that $Y_{u}:\Omega\to\mathbb{R}$ possesses the following properties:
\begin{equation}\label{eq:strong-Markov-proof}
Y_{u}\textrm{ is }\sigmafield{F}_{\sigma}\textrm{-measurable}\mspace{15mu}\textrm{and}\mspace{15mu}
\mathbb{E}_{x}\bigl[u(X_{T})\ind{\{T<\sigma+\tau\circ\theta_{\sigma}\}}\ind{A}\bigr]
	=\mathbb{E}_{x}[Y_{u}\ind{A}]
	\textrm{ for any }A\in\sigmafield{F}_{\sigma}.
\end{equation}

We first prove \eqref{eq:strong-Markov-proof} for $u\in C(M_{\cemetery})$.
Let $n\in\mathbb{N}$ and define $T_{n}:\Omega\to[0,\infty]$ by
\begin{equation}\label{eq:strong-Markov-Tn}
T_{n}|_{\{\sigma+(k-1)2^{-n}\leq T<\sigma+k2^{-n}\}}:=\sigma+k2^{-n},
	\quad k\in\mathbb{N}\qquad\textrm{and}\qquad T_{n}|_{\{T=\infty\}}:=\infty,
\end{equation}
so that $T_{n}$ is $\sigmafield{F}_{\sigma}$-measurable and
$T_{n}-2^{-n}\leq T\leq T_{n}$. Also define $Y_{u,n}$
in the same way as $Y_{u}$ with $T_{n}$ in place of $T$.
Then $Y_{u,n}|_{\{T=\infty\}}=0=Y_{u}|_{\{T=\infty\}}$, and
$\lim_{n\to\infty}Y_{u,n}=Y_{u}$ on $\{T<\infty\}$ by
$T_{n}-2^{-n}\leq T\leq T_{n}$, the sample path right-continuity of $X$
and dominated convergence. Also for $k\in\mathbb{N}$,
on $\{\sigma+(k-1)2^{-n}\leq T<\sigma+k2^{-n}\}\in\sigmafield{F}_{\sigma}$ we have
$Y_{u,n}=\mathbb{E}_{X_{\sigma}}[u(X_{k2^{-n}})\ind{\{k2^{-n}<\tau\}}]$,
and since the latter is $\sigmafield{F}_{\sigma}$-measurable by
\cite[Exercise A.1.20]{CF} so are $Y_{u,n}$ and $Y_{u}=\lim_{n\to\infty}Y_{u,n}$.
Now for $A\in\sigmafield{F}_{\sigma}$, thanks to dominated convergence,
\begin{align*}
\mathbb{E}_{x}\bigl[u(X_{T_{n}})\ind{\{T_{n}<\sigma+\tau\circ\theta_{\sigma}\}}\ind{A}\bigr]
	&=\sum_{k\in\mathbb{N}}\mathbb{E}_{x}\bigl[\ind{A\cap\{T_{n}=\sigma+k2^{-n}<\infty\}}\bigl((u(X_{k2^{-n}})\ind{\{k2^{-n}<\tau\}})\circ\theta_{\sigma}\bigr)\bigr]\\
&=\sum_{k\in\mathbb{N}}\mathbb{E}_{x}\bigl[\ind{A\cap\{T_{n}=\sigma+k2^{-n}<\infty\}}\mathbb{E}_{X_{\sigma}}[u(X_{k2^{-n}})\ind{\{k2^{-n}<\tau\}}]\bigr]\\
	&=\mathbb{E}_{x}[Y_{u,n}\ind{A}]
\end{align*}
by the strong Markov property \cite[Theorem A.1.21]{CF} of $X$ at time $\sigma$,
and we conclude \eqref{eq:strong-Markov-proof} by using $T_{n}-2^{-n}\leq T\leq T_{n}$
and the sample path right-continuity of $X$ to let $n\to\infty$.

Note that for $u\in\functionspace{B}_{\mathrm{b}}(M_{\cemetery})$ and
$\{u_{n}\}_{n\in\mathbb{N}}\subset\functionspace{B}_{\mathrm{b}}(M_{\cemetery})$
such that $\sup_{n\in\mathbb{N}}\|u_{n}\|_{\mathrm{sup}}<\infty$ and
$\lim_{n\to\infty}u_{n}(y)=u(y)$ for any $y\in M_{\cemetery}$,
if $u_{n}$ satisfies \eqref{eq:strong-Markov-proof} for any $n\in\mathbb{N}$
then so does $u$ by dominated convergence. Therefore it follows from
the previous paragraph that \eqref{eq:strong-Markov-proof} holds for
$u=\ind{B}$ with $B\subset M_{\cemetery}$ closed in $M_{\cemetery}$, hence
also with $B\in\sigmafield{B}(M_{\cemetery})$ by the Dynkin class theorem
\cite[Chapter 0, Theorem 2.2]{BG}, and thus for any
$u\in\functionspace{B}_{\mathrm{b}}(M_{\cemetery})$.
\end{proof}

\begin{proof}[Proof of Theorem \textup{\ref{thm:multiple-DH}}]
For $n\in\mathbb{N}$, $\tau_{n}\leq\sigma_{n}\leq\tau_{n+1}$ by
\eqref{eq:entrance-exit-after-sigma} and \eqref{eq:multiple-DH-stopping-times},
and the sample path right-continuity of $X$ implies that
$X_{\tau_{n}}\in M\setminus U$ and $\tau_{n}<\sigma_{n}$ on $\{\tau_{n}<\zeta\}$ and that
$X_{\sigma_{n}}\in\overline{B}$ and $\sigma_{n}<\tau_{n+1}\wedge\zeta$ on $\{\sigma_{n}<\infty\}$.
Moreover, setting $\tau:=\lim_{n\to\infty}\tau_{n}=\lim_{n\to\infty}\sigma_{n}$,
we see from the quasi-left-continuity \cite[(A.2.4)]{FOT} of $X$ that
for any $x\in M$,
\begin{equation}\label{eq:tau-sigma-lim}
\mathbb{P}_{x}[\tau<\zeta]
	=\mathbb{P}_{x}[\tau<\zeta,\,\lim\nolimits_{n\to\infty}X_{\tau_{n}}=X_{\tau}=\lim\nolimits_{n\to\infty}X_{\sigma_{n}}]
	=\mathbb{P}_{x}[\emptyset]=0.
\end{equation}

Let $(t,x)\in[0,\infty)\times M$.
Then for each $\omega\in\{\textrm{$X_{t}\in B$, $\zeta\leq\tau$}\}$,
$t<\zeta(\omega)\leq\tau(\omega)$ and hence either $t<\tau_{1}(\omega)$,
or $\tau_{n}(\omega)\leq t<\tau_{n+1}(\omega)$ for some $n\in\mathbb{N}$, whence
$\sigma_{n}(\omega)\leq t<\tau_{n+1}(\omega)$ by $X_{t}(\omega)\in B$; namely
$\{\textrm{$X_{t}\in B$, $\zeta\leq\tau$}\}\subset
	\{t<\tau_{1}\}\cup\bigcup_{n\in\mathbb{N}}\{\sigma_{n}\leq t<\tau_{n+1}\}$,
and this union is disjoint. Therefore for any
$u\in\functionspace{B}_{\mathrm{b}}(M)$ with $u|_{M\setminus B}=0$, noting that
$\tau_{n+1}=\sigma_{n}+\tau_{U}\circ\theta_{\sigma_{n}}$ for any $n\in\mathbb{N}$
and using \eqref{eq:tau-sigma-lim}, dominated convergence and
Proposition \ref{prop:strong-Markov}, we obtain
\begin{align*}
\functionspace{P}_{t}u(x)&=\mathbb{E}_{x}[u(X_{t})]
	=\mathbb{E}_{x}[u(X_{t})\ind{\{X_{t}\in B,\,\zeta\leq\tau\}}]\\
&=\mathbb{E}_{x}\biggl[u(X_{t})\ind{\{X_{t}\in B,\,\zeta\leq\tau\}}\biggl(\ind{\{t<\tau_{1}\}}
	+\sum_{n\in\mathbb{N}}\ind{\{\sigma_{n}\leq t<\tau_{n+1}\}}\biggr)\biggr]\\
&=\mathbb{E}_{x}[u(X_{t})\ind{\{t<\tau_{U}\}}]
	+\sum_{n\in\mathbb{N}}\mathbb{E}_{x}[u(X_{t})\ind{\{\sigma_{n}\leq t<\sigma_{n}+\tau_{U}\circ\theta_{\sigma_{n}}\}}]\\
&=\functionspace{P}^{U}_{t}u(x)
	+\sum_{n\in\mathbb{N}}\mathbb{E}_{x}\bigl[\ind{\{\sigma_{n}\leq t\}}\mathbb{E}_{x}\bigl[u(X_{t})\ind{\{t<\sigma_{n}\wedge t+\tau_{U}\circ\theta_{\sigma_{n}\wedge t}\}}\bigm|\sigmafield{F}_{\sigma_{n}\wedge t}\bigr]\bigr]\\
&=\functionspace{P}^{U}_{t}u(x)
	+\sum_{n\in\mathbb{N}}\int_{\{\sigma_{n}\leq t\}}\mathbb{E}_{X_{\sigma_{n}\wedge t}(\omega)}\bigl[u(X_{t-\sigma_{n}(\omega)\wedge t})\ind{\{t-\sigma_{n}(\omega)\wedge t<\tau_{U}\}}\bigr]\,d\mathbb{P}_{x}(\omega)\\
&=\functionspace{P}^{U}_{t}u(x)
	+\sum_{n\in\mathbb{N}}\mathbb{E}_{x}\bigl[\ind{\{\sigma_{n}\leq t\}}\functionspace{P}^{U}_{t-\sigma_{n}}u(X_{\sigma_{n}})\bigr],
\end{align*}
where the equality in the fourth line holds since $\{\sigma_{n}\leq t\}\in\sigmafield{F}_{\sigma_{n}\wedge t}$
by \cite[Lemma 1.2.16]{KS}. Thus we have proved \eqref{eq:multiple-DH} for
$u\in\functionspace{B}_{\mathrm{b}}(M)$ with $u|_{M\setminus B}=0$,
which easily implies \eqref{eq:multiple-DH} for $u\in\functionspace{B}^{+}(M)$
with $u|_{M\setminus B}=0$ by monotone convergence.
\end{proof}

\section{Symmetry of a Hunt process and the associated Dirichlet form}\label{sec:Dirichlet-form}

In this section, assuming the symmetry of our Hunt process $X$, we first recall
that such $X$ naturally gives rise to a symmetric Dirichlet form, and then
introduce related potential theoretic notions.
We refer the reader to \cite{FOT,CF} for further details.

\subsection{The Dirichlet form of a symmetric Hunt process}\label{ssec:Dirichlet-form}

In the rest of this paper, we fix a metric $d$ on $M$ compatible with
the topology of $M$, and a Radon measure $\mu$ on $M$ with full support,
i.e., a Borel measure on $M$ such that $\mu(K)<\infty$ for any
$K\subset M$ compact and $\mu(U)>0$ for any $U\subset M$ non-empty open.
We set $B(x,r):=\{y\in M\mid d(x,y)<r\}$ for $(x,r)\in M\times(0,\infty)$
and $\diam A:=\sup_{x,y\in A}d(x,y)$ for $A\subset M$.
For $q\in[1,\infty)$, we set
$\|u\|_{q}:=(\int_{M}|u|^{q}d\mu)^{1/q}$ for $u\in\functionspace{B}(M)$ and
$\functionspace{B}L^{q}(M,\mu):=\{u\in\functionspace{B}(M)\mid\|u\|_{q}<\infty\}$,
and we also set $\langle u,v\rangle:=\int_{M}uv\,d\mu$ for
$u,v\in\functionspace{B}^{+}(M)$ and for $u,v\in\functionspace{B}(M)$
with $\|uv\|_{1}<\infty$. For $\|\cdot\|_{q}$ and $\langle\cdot,\cdot\rangle$,
we use the same notation for $\mu$-equivalence classes of functions as well.

Now we assume that $X$ is \emph{$\mu$-symmetric},
i.e., $\langle\functionspace{P}_{t}u,v\rangle=\langle u,\functionspace{P}_{t}v\rangle$
for any $t\in(0,\infty)$ and any $u,v\in\functionspace{B}^{+}(M)$.
Then for each $t\in(0,\infty)$, as in \cite[(1.4.13)]{FOT} we can easily verify that
$\|\functionspace{P}_{t}u\|_{2}\leq\|u\|_{2}$ for any $u\in\functionspace{B}^{+}(M)$,
so that $\functionspace{P}_{t}u$ is defined $\mu$-a.e.\ and determines an
element $T_{t}u$ of $L^{2}(M,\mu)$ for each $u\in L^{2}(M,\mu)$ independently of a
particular choice of a $\mu$-version of $u$. Thus the transition function
$\{\functionspace{P}_{t}\}_{t\in[0,\infty)}$ of $X$ canonically induces a
symmetric contraction semigroup $\{T_{t}\}_{t\in(0,\infty)}$ on $L^{2}(M,\mu)$
which is also \emph{Markovian}, i.e., $0\leq T_{t}u\leq 1$ $\mu$-a.e.\ for any
$t\in(0,\infty)$ and any $u\in L^{2}(M,\mu)$ with $0\leq u\leq 1$ $\mu$-a.e.
This semigroup $\{T_{t}\}_{t\in(0,\infty)}$ is
in fact strongly continuous thanks to the sample path right-continuity of $X$
as shown in \cite[Lemma 1.4.3-(i)]{FOT} and hence determines a symmetric
Dirichlet form $(\functionspace{E},\functionspace{F})$ on $L^{2}(M,\mu)$
by \cite[Lemma 1.3.4-(i) and Theorem 1.4.1]{FOT}. Namely, we have a dense linear
subspace $\functionspace{F}$ of $L^{2}(M,\mu)$ and a non-negative definite symmetric
bilinear form $\functionspace{E}:\functionspace{F}\times\functionspace{F}\to\mathbb{R}$
given by
\begin{equation}\label{eq:Dirichlet-form}
\begin{split}
\functionspace{F}&:=\Bigl\{u\in L^{2}(M,\mu)\Bigm|\lim_{t\downarrow 0}t^{-1}\langle u-T_{t}u,u\rangle<\infty\Bigr\},\\
&\functionspace{E}(u,v):=\lim_{t\downarrow 0}t^{-1}\langle u-T_{t}u,v\rangle,\qquad u,v\in\functionspace{F},
\end{split}
\end{equation}
respectively, and $(\functionspace{E},\functionspace{F})$ is \emph{closed}
(i.e., $\functionspace{F}$ forms a Hilbert space with inner product
$\functionspace{E}_{1}:=\functionspace{E}+\langle\cdot,\cdot\rangle$) and
\emph{Markovian} (i.e., $u^{+}\wedge 1\in\functionspace{F}$ and
$\functionspace{E}(u^{+}\wedge 1,u^{+}\wedge 1)\leq\functionspace{E}(u,u)$
for any $u\in\functionspace{F}$). $(\functionspace{E},\functionspace{F})$
is called the \emph{Dirichlet form of the $\mu$-symmetric Hunt process $X$}.
Note that by \cite[Lemma 1.3.3-(i)]{FOT},
\begin{equation}\label{eq:TtL2-F}
T_{t}(L^{2}(M,\mu))\subset\functionspace{F}\qquad\textrm{for any }t\in(0,\infty).
\end{equation}

In what follows we further assume that the Dirichlet form
$(\functionspace{E},\functionspace{F})$ of $X$ is \emph{regular on $L^{2}(M,\mu)$},
i.e., that $\functionspace{F}\cap C_{\mathrm{c}}(M)$ is dense
both in $(\functionspace{F},\functionspace{E}_{1})$
and in $(C_{\mathrm{c}}(M),\|\cdot\|_{\mathrm{sup}})$.
Note that this framework actually contains any regular symmetric Dirichlet form
on any locally compact separable metric space $(M,d)$ equipped with a Radon measure
$\mu$ with full support, since any such form can be realized as the Dirichlet
form of some $\mu$-symmetric Hunt process on $(M,\sigmafield{B}(M))$
by the fundamental result \cite[Theorem 7.2.1]{FOT} from Dirichlet form theory.

\subsection{Capacity, quasi-continuity and exceptional sets}\label{ssec:capacity}

The following potential theoretic notions are adopted from
\cite[Section 2.1]{FOT} and \cite[Sections 1.2 and 1.3]{CF}.

\begin{definition}\label{dfn:capacity}
\begin{itemize}[label=\textup{(1)},align=left,leftmargin=*]
\item[\textup{(1)}]We define the \emph{$1$-capacity $\Capa_{1}$ associated with $(M,\mu,\functionspace{E},\functionspace{F})$} by
\begin{align}
\capa_{1}(U)
	&:=\inf\{\functionspace{E}_{1}(u,u)\mid u\in\functionspace{F},\ u\geq 1\ \mu\textrm{-a.e.\ on }U\}
	&&\textrm{for }U\subset M\textrm{ open in }M,\notag\\
\Capa_{1}(A)
	&:=\inf\{\capa_{1}(U)\mid U\subset M\textrm{ open in }M,\ A\subset U\}
	&&\textrm{for }A\subset M
\label{eq:capacity}
\end{align}
(recall $\functionspace{E}_{1}:=\functionspace{E}+\langle\cdot,\cdot\rangle$).
Clearly, $\Capa_{1}$ extends $\capa_{1}$ and
$\mu(A)\leq \Capa_{1}(A)$ for $A\in\sigmafield{B}(M)$.
\item[\textup{(2)}]A subset $N$ of $M$ is called \emph{$\functionspace{E}$-polar} if and only if
$\Capa_{1}(N)=0$. Moreover, if $A\subset M$ and $\functionspace{S}(x)$ is
a statement in $x\in A$, then we say that
\emph{$\functionspace{S}$ holds $\functionspace{E}$-q.e.\ on $A$} if and only if
$\{x\in A\mid\textrm{$\functionspace{S}(x)$ fails}\}$ is $\functionspace{E}$-polar.
When $A=M$ we simply say \emph{``$\functionspace{S}$ holds $\functionspace{E}$-q.e."}\ instead.
\item[\textup{(3)}]Let $U\subset M$ be open in $M$. A function $u:U\setminus N\to[-\infty,\infty]$,
with $N\subset M$ $\functionspace{E}$-polar, is called
\emph{$\functionspace{E}$-quasi-continuous on $U$} if and only if
for any $\varepsilon\in(0,\infty)$ there exists an open subset
$V$ of $M$ with $U\cap N\subset V$ and $\Capa_{1}(V)<\varepsilon$ such that
$u|_{U\setminus V}$ is $\mathbb{R}$-valued and continuous.
When $U=M$, such $u$ is simply called
\emph{$\functionspace{E}$-quasi-continuous} instead.
\end{itemize}
\end{definition}

\begin{remark}\label{rmk:capacity}
There are several equivalent ways of defining the notions of
$\functionspace{E}$-polar sets and $\functionspace{E}$-quasi-continuous functions.
See \cite[Section 1.2 and Theorem 1.3.14]{CF} in this connection.
\end{remark}

Note that $\Capa_{1}$ is countably subadditive by \cite[Lemma 2.1.2 and Theorem A.1.2]{FOT}.

Let $U\subset M$ be open in $M$. By \cite[Lemma 2.1.4]{FOT}, if $u,v$ are
$\functionspace{E}$-quasi-continuous functions
on $U$ and $u\leq v$ $\mu$-a.e.\ on $U$,
then $u\leq v$ $\functionspace{E}$-q.e.\ on $U$.
In particular, for each $u\in L^{2}(M,\mu)$, an $\functionspace{E}$-quasi-continuous
$\mu$-version of $u$, if it exists, is unique up to $\functionspace{E}$-q.e.
By \cite[Theorem 2.1.3]{FOT}, each $u\in\functionspace{F}$
admits an $\functionspace{E}$-quasi-continuous $\mu$-version,
which is denoted as $\widetilde{u}$.

For each $t\in(0,\infty)$, while $T_{t}u=\functionspace{P}_{t}u$ $\mu$-a.e.\ for
any $u\in L^{2}(M,\mu)$ by the definition of $T_{t}$, more strongly it actually
holds by \cite[Theorem 4.2.3-(i)]{FOT} that for any $u\in\functionspace{B}L^{2}(M,\mu)$,
\begin{equation}\label{eq:Ptu-quasi-continuous}
\functionspace{P}_{t}u\textrm{ is an }\functionspace{E}\textrm{-quasi-continuous }\mu\textrm{-version of }T_{t}u.
\end{equation}

The following definition gives a probabilistic counterpart of the notion of
$\functionspace{E}$-polar sets.

\begin{definition}\label{dfn:properly-exceptional}
A Borel set $N\in\sigmafield{B}(M)$ is called
\emph{properly exceptional for $X$} if and only if $\mu(N)=0$ and
for any $x\in M\setminus N$,
$\mathbb{P}_{x}[\dot{\sigma}_{N}\wedge\hat{\sigma}_{N}=\infty]=1$
or, by \eqref{eq:entrance-leq-entrance-left-limit}, equivalently 
\begin{equation}\label{eq:properly-exceptional}
\mathbb{P}_{x}[\dot{\sigma}_{N}=\infty]=1.
\end{equation}
\end{definition}

Note that
$\{\dot{\sigma}_{N}\wedge\hat{\sigma}_{N}=\infty\}
	=\{\textrm{$X_{0},X_{t},X_{t-}\in M_{\cemetery}\setminus N$ for any $t\in(0,\infty)$}\}
	\in\sigmafield{F}_{\infty}$
and that $\{\dot{\sigma}_{N}=\infty\}
	=\{\textrm{$X_{t}\in M_{\cemetery}\setminus N$ for any $t\in[0,\infty)$}\}
	\in\sigmafield{F}_{\infty}$.
Every properly exceptional set for $X$ is $\functionspace{E}$-polar by
\cite[Theorem 4.2.1-(ii)]{FOT}, and conversely any $\functionspace{E}$-polar
set is included in a Borel properly exceptional set for $X$ by \cite[Theorem 4.1.1]{FOT}.

\subsection{The Dirichlet form of the part process on open sets}\label{ssec:part-form}

Let $U$ be a non-empty open subset of $M$ and set $\mu|_{U}:=\mu|_{\sigmafield{B}(U)}$.
Recall that the part $X^{U}$ of $X$ on $U$ is a Hunt process on
$(U,\sigmafield{B}(U))$ defined in \eqref{eq:part-process} and that
its transition function naturally extends to $(M,\sigmafield{B}(M))$ as
a family $\{\functionspace{P}^{U}_{t}\}_{t\in[0,\infty)}$ of Markovian kernels
on $(M,\sigmafield{B}(M))$ given by \eqref{eq:transition-function-part}.
In the present situation, the assumed $\mu$-symmetry of $X$ implies that
$X^{U}$ is $\mu|_{U}$-symmetric. More precisely,
for any $t\in(0,\infty)$ and any $u,v\in\functionspace{B}^{+}(M)$, we have
$\langle\functionspace{P}^{U}_{t}u,v\rangle=\langle u,\functionspace{P}^{U}_{t}v\rangle$
by \cite[Lemma 4.1.3]{FOT} and hence also
$\|\functionspace{P}^{U}_{t}u\|_{2}\leq\|u\|_{2}$ as in \cite[(1.4.13)]{FOT}.
Thus we obtain a Markovian symmetric contraction semigroup $\{T^{U}_{t}\}_{t\in(0,\infty)}$
on $L^{2}(M,\mu)$ canonically induced by $\{\functionspace{P}^{U}_{t}\}_{t\in(0,\infty)}$
in the same way as for $\{\functionspace{P}_{t}\}_{t\in(0,\infty)}$. Moreover,
under the natural identification of $L^{2}(U,\mu|_{U})$ with the closed linear
subspace $\{u\in L^{2}(M,\mu)\mid\textrm{$u=0$ $\mu$-a.e.\ on $M\setminus U$}\}$
of $L^{2}(M,\mu)$, the strongly continuous Markovian semigroup on $L^{2}(U,\mu|_{U})$
induced by the transition function of $X^{U}$ is easily shown to be given by
$\{T^{U}_{t}|_{L^{2}(U,\mu|_{U})}\}_{t\in(0,\infty)}$, and hence
\eqref{eq:Dirichlet-form} with $T^{U}_{t}$ in place of $T_{t}$ gives the Dirichlet
form $(\functionspace{E}^{U},\functionspace{F}_{U})$ of $X^{U}$. In fact,
\begin{equation}\label{eq:part-form}
\functionspace{F}_{U}
	=\{u\in\functionspace{F}\mid\widetilde{u}=0\ \functionspace{E}\textrm{-q.e.\ on }M\setminus U\}
	\qquad\textrm{and}\qquad
	\functionspace{E}^{U}=\functionspace{E}|_{\functionspace{F}_{U}\times\functionspace{F}_{U}}
\end{equation}
by \cite[Theorem 4.4.2]{FOT} and $(\functionspace{E}^{U},\functionspace{F}_{U})$
is regular on $L^{2}(U,\mu|_{U})$ by \cite[Lemma 1.4.2-(ii) and Corollary 2.3.1]{FOT}.
$(\functionspace{E}^{U},\functionspace{F}_{U})$ is called the
\emph{part of the Dirichlet form $(\functionspace{E},\functionspace{F})$ on $U$}.

For $t\in(0,\infty)$ and $u\in\functionspace{B}L^{2}(M,\mu)$, while
$T^{U}_{t}u=\functionspace{P}^{U}_{t}u$ $\mu$-a.e.\ by definition, more strongly
\begin{equation}\label{eq:Ptu-quasi-continuous-part}
\functionspace{P}^{U}_{t}u\textrm{ is an }\functionspace{E}\textrm{-quasi-continuous }
	\mu\textrm{-version of }T^{U}_{t}u,
\end{equation}
similarly to \eqref{eq:Ptu-quasi-continuous}.
Indeed, since $v:=T^{U}_{t}u\in\functionspace{F}_{U}\subset\functionspace{F}$
by \eqref{eq:TtL2-F} and \eqref{eq:part-form},
$v$ admits an $\functionspace{E}$-quasi-continuous $\mu$-version
$\widetilde{v}$ and then $\widetilde{v}=0=\functionspace{P}^{U}_{t}u$
$\functionspace{E}$-q.e.\ on $M\setminus U$ by \eqref{eq:part-form}.
On the other hand, $(\functionspace{P}^{U}_{t}u)|_{U}$ is a $\mu$-version
of $v|_{U}$ which is $\functionspace{E}$-quasi-continuous on $U$ by
\cite[Theorem 4.4.3]{FOT} and therefore $(\functionspace{P}^{U}_{t}u)|_{U}=\widetilde{v}|_{U}$
$\functionspace{E}$-q.e.\ on $U$ by \cite[Lemma 2.1.4]{FOT}.
Thus $\functionspace{P}^{U}_{t}u=\widetilde{v}$ $\functionspace{E}$-q.e.,
which together with the $\functionspace{E}$-quasi-continuity of $\widetilde{v}$
yields \eqref{eq:Ptu-quasi-continuous-part}.

\section{Localized quasi-everywhere existence of the heat kernel}\label{sec:HK-existence}

As in Section \textup{\ref{sec:Dirichlet-form}}, let $(M,d)$ be a locally compact
separable metric space equipped with a Radon measure $\mu$ with full support, and
let $X$ be a $\mu$-symmetric Hunt process on $(M,\sigmafield{B}(M))$ whose
Dirichlet form $(\functionspace{E},\functionspace{F})$ is regular on $L^{2}(M,\mu)$.
\emph{Throughout the rest of this paper, we fix this setting and follow the notation
introduced in Section \textup{\ref{sec:Dirichlet-form}} in addition to that from
Sections \textup{\ref{sec:Hunt-processes}} and \textup{\ref{sec:multiple-DH}}.}

The purpose of this section is to prove Theorem \ref{thm:HK-existence}
below on the existence of the heat kernel $p^{U}=p^{U}_{t}(x,y)$
for $\{\functionspace{P}^{U}_{t}\}_{t\in(0,\infty)}$
on a given subset of $(0,\infty)\times M\times M$ under a suitable
upper bound on the Markovian semigroup $\{T^{U}_{t}\}_{t\in(0,\infty)}$
\emph{which is assumed only on the given subset}.
In the case where the subset is the whole $(0,\infty)\times M\times M$,
similar results have been obtained, e.g., in
\cite[Sections 7 and 8]{Gri:HKfractal} and \cite[Theorem 3.1]{BBCK}.

\begin{remark}\label{rmk:Dirichlet-form}
The $\mu$-symmetry of $X$ and the regularity of its Dirichlet form are assumed mostly
for the sake of simplicity of the framework. In fact, we need these assumptions
\emph{only in order to use potential theoretic results from
\cite[\emph{Chapters \textup{2} and \textup{4}}]{FOT}
in the proof of Theorem \textup{\ref{thm:HK-existence}}};
it should be possible to extend Theorem \ref{thm:HK-existence} to a more
general framework where the same kind of potential theory remains available,
and the reader is referred to
Remarks \ref{rmk:HK-existence}, \ref{rmk:HKUB} and \ref{rmk:exit-probability}
for the precise settings actually required for the (other) results in
Sections \ref{sec:HK-existence}, \ref{sec:HKUB} and \ref{sec:exit-probability}, respectively.
\end{remark}

For $A\subset M$, let $\ind{A}:M\to\{0,1\}$ denote its indicator function
given by $\ind{A}|_{A}:=1$ and $\ind{A}|_{M\setminus A}:=0$.
In what follows we allow an interval $I\subset\mathbb{R}$ to be a one-point set.

\begin{definition}\label{dfn:upper-bound-function}
Let $I\subset(0,\infty)$ be an interval, $V$ an open subset of $M$ and
$W\in\sigmafield{B}(M)$. A Borel measurable function
$H=H_{t}(x,y):I\times V\times W\to[0,\infty]$ is called a
\emph{$\mu$-upper bound function on $I\times V\times W$} if and only if
the following three conditions are satisfied:
\begin{itemize}[label=\textup{(UB1)},align=left,leftmargin=*]
\item[\textup{(UB1)}]$\limsup_{s\downarrow t}H_{s}(x,y)\leq H_{t}(x,y)$
	for any $(t,x,y)\in I\times V\times W$ with $t<\sup I$.
\item[\textup{(UB2)}]$H_{t}(\cdot,y):V\to[0,\infty]$ is upper semi-continuous
	for any $(t,y)\in I\times W$.
\item[\textup{(UB3)}]There exist $\{h_{n}\}_{n\in\mathbb{N}}\subset\functionspace{B}^{+}(M)$
	and non-decreasing sequences $\{I_{n}\}_{n\in\mathbb{N}}$ of open subsets of $I$,
	$\{V_{n}\}_{n\in\mathbb{N}}$ of open subsets of $V$ and
	$\{W_{n}\}_{n\in\mathbb{N}}$ of Borel subsets of $W$
	with $I=\bigcup_{n\in\mathbb{N}}I_{n}$, $V=\bigcup_{n\in\mathbb{N}}V_{n}$
	and $W=\bigcup_{n\in\mathbb{N}}W_{n}$ such that for any $n\in\mathbb{N}$,
	\begin{equation}\label{eq:upper-bound-function}
		\int_{W_{n}}h_{n}\,d\mu<\infty\quad\textrm{and}\quad H_{t}(x,y)\leq h_{n}(y)
		\textrm{ for any }(t,x,y)\in I_{n}\times V_{n}\times W_{n}.
	\end{equation}
\end{itemize}
\end{definition}

\begin{remark}\label{rmk:upper-bound-function}
\begin{itemize}[label=\textup{(1)},align=left,leftmargin=*]
\item[\textup{(1)}]In \textup{(UB3)} we may assume that
	$\mu(V_{n}\cup W_{n})<\infty$ for any $n\in\mathbb{N}$, by taking a
	non-decreasing sequence $\{M_{n}\}_{n\in\mathbb{N}}$ of open subsets of $M$
	with $\overline{M_{n}}$ compact and $M=\bigcup_{n\in\mathbb{N}}M_{n}$ and replacing
	$V_{n}$ and $W_{n}$ with $V_{n}\cap M_{n}$ and $W_{n}\cap M_{n}$, respectively.
\item[\textup{(2)}]It is easy to see that the condition \textup{(UB3)} in
	Definition \ref{dfn:upper-bound-function} is satisfied if $W$ is open in $M$ and
	$\|H\|_{\mathrm{sup},K}=\sup_{(t,x,y)\in K}H_{t}(x,y)<\infty$ for any compact subset $K$ of $I\times V\times W$.
\end{itemize}
\end{remark}

\begin{theorem}\label{thm:HK-existence}
Let $I\subset(0,\infty)$ be an interval, $V$ an open subset of $M$,
$W\in\sigmafield{B}(M)$ and let $H=H_{t}(x,y)$ be a $\mu$-upper bound function
on $I\times V\times W$. Let $U$ be a non-empty open subset of $M$.
Then for each countable dense subset $J$ of $I$ satisfying $\max I\in J$
if $\max I$ exists, the following three conditions are equivalent:
\begin{itemize}[label=\textup{(1)},align=left,leftmargin=*]
\item[\textup{(1)}]For any $t\in J$ and any $v,w\in L^{2}(M,\mu)$
	with $(v\ind{V})\wedge(w\ind{W})\geq 0$ $\mu$-a.e.,
	\begin{equation}\label{eq:HK-existence1}
	\langle v\ind{V},T^{U}_{t}(w\ind{W})\rangle\leq\int_{V\times W}v(x)H_{t}(x,y)w(y)\,d(\mu\times\mu)(x,y).
	\end{equation}
\item[\textup{(2)}]For each $t\in J$ and each $w\in L^{2}(M,\mu)$
	with $w\ind{W}\geq 0$ $\mu$-a.e.,
	\begin{equation}\label{eq:HK-existence2}
	T^{U}_{t}(w\ind{W})(x)\leq\int_{W}H_{t}(x,y)w(y)\,d\mu(y)
		\quad\textrm{for }\mu\textrm{-a.e.\ }x\in V.
	\end{equation}
\item[\textup{(3)}]There exist a properly exceptional set $N\in\sigmafield{B}(M)$
	for $X$ and a Borel measurable function
	$p^{U}=p^{U}_{t}(x,y):I\times(V\setminus N)\times W\to[0,\infty]$
	such that for any $(t,x)\in I\times(V\setminus N)$,
	\begin{gather}\label{eq:HK-existence3-1}
	\functionspace{P}^{U}_{t}(x,A)=\int_{A}p^{U}_{t}(x,y)\,d\mu(y)
	\qquad\textrm{for any }A\in\sigmafield{B}(W),\\
	p^{U}_{t}(x,y)\leq H_{t}(x,y)\qquad\textrm{for any }y\in W.
	\label{eq:HK-existence3-2}
	\end{gather}
\end{itemize}
\end{theorem}

We first show the following proposition, which is of independent interest and
will be used in the proof of the implication \textup{(2)}$\Rightarrow$\textup{(3)}
of Theorem \ref{thm:HK-existence} and also in the proof of
Theorems \ref{thm:HKUB-localized} and \ref{thm:HKUB-localized-global}
in the next section.

\begin{remark}\label{rmk:HK-existence}
In fact, \emph{Proposition \textup{\ref{prop:HK-existence}} below applies,
without any changes in the proof, to any locally compact separable metrizable
topological space $M$, any $\sigma$-finite Borel measure $\mu$ on $M$
and any Hunt process $X$ on $(M,\sigmafield{B}(M))$}.
\end{remark}

\begin{proposition}\label{prop:HK-existence}
Let $I\in\sigmafield{B}([0,\infty))$, let $V,W\in\sigmafield{B}(M)$ and let
$H=H_{t}(x,y):I\times V\times W\to[0,\infty]$ be Borel measurable. Let $U$ be a
non-empty open subset of $M$. Then the following two conditions are equivalent:
\begin{itemize}[label=\textup{(1)},align=left,leftmargin=*]
\item[\textup{(1)}]For any $(t,x)\in I\times V$ and any $A\in\sigmafield{B}(W)$,
	\begin{equation}\label{eq:HK-existence4}
	\functionspace{P}^{U}_{t}(x,A)\leq\int_{A}H_{t}(x,y)\,d\mu(y).
	\end{equation}
\item[\textup{(2)}]There exists a Borel measurable function
	$p^{U}=p^{U}_{t}(x,y):I\times V\times W\to[0,\infty]$ such that
	\eqref{eq:HK-existence3-1} and \eqref{eq:HK-existence3-2} hold for any $(t,x)\in I\times V$.
\end{itemize}
\end{proposition}

\begin{proof}
Since the implication \textup{(2)}$\Rightarrow$\textup{(1)} is immediate,
it suffices to show the converse \textup{(1)}$\Rightarrow$\textup{(2)}.
By the $\sigma$-finiteness of $\mu$, we can choose
$\{W_{n}\}_{n\in\mathbb{N}}\subset\sigmafield{B}(W)$ with
$W=\bigcup_{n\in\mathbb{N}}W_{n}$ so that
$W_{n}\subset W_{n+1}$ and $\mu(W_{n})<\infty$ for any $n\in\mathbb{N}$.
We will construct for each $n\in\mathbb{N}$ a function
$p^{U,n}=p^{U,n}_{t}(x,y):I\times V\times W_{n}\to[0,\infty]$
possessing the required properties with $W_{n}$ in place of $W$.
If $\mu(W_{n})=0$ then it suffices to set $p^{U,n}:=0$ in view of
\eqref{eq:HK-existence4}, and therefore we may assume $\mu(W_{n})>0$.
Let $\sigmafield{U}=\{A_{k}\}_{k\in\mathbb{N}}$ be a countable open base for
the topology of $M$, set $A_{k}^{0}:=M\setminus A_{k}$ and $A_{k}^{1}:=A_{k}$
for $k\in\mathbb{N}$, and define
\begin{equation}\label{eq:HK-existence-filtration}
\sigmafield{A}_{k}:=\{{\textstyle\bigcup_{\alpha\in\functionspace{I}}}A_{k}^{\alpha}\mid
	\functionspace{I}\subset\{0,1\}^{k}\},\quad k\in\mathbb{N},
\end{equation}
where $A_{k}^{\alpha}:={\textstyle\bigcap_{i=1}^{k}}A_{i}^{\alpha_{i}}$
for $\alpha=(\alpha_{i})_{i=1}^{k}\in\{0,1\}^{k}$, so that
$\{\sigmafield{A}_{k}\}_{k\in\mathbb{N}}$ is a non-decreasing sequence of
$\sigma$-fields in $M$ with $\bigcup_{k\in\mathbb{N}}\sigmafield{A}_{k}$
generating $\sigmafield{B}(M)$. For $k\in\mathbb{N}$, noting that
$M=\bigcup_{\alpha\in\{0,1\}^{k}}A_{k}^{\alpha}$ and that
$A_{k}^{\alpha}\cap A_{k}^{\beta}=\emptyset$ for $\alpha,\beta\in\{0,1\}^{k}$
with $\alpha\not=\beta$, define $p^{U,n,k}=p^{U,n,k}_{t}(x,y):I\times V\times M\to[0,\infty)$
by, for $\alpha\in\{0,1\}^{k}$ and $(t,x,y)\in I\times V\times A_{k}^{\alpha}$,
\begin{equation}\label{eq:HK-existence-conditional-expectation}
p^{U,n,k}_{t}(x,y):=
	\begin{cases}
		\mu(A_{k}^{\alpha}\cap W_{n})^{-1}\functionspace{P}^{U}_{t}\ind{A_{k}^{\alpha}\cap W_{n}}(x)
		&\textrm{if }\mu(A_{k}^{\alpha}\cap W_{n})>0,\\
		0&\textrm{if }\mu(A_{k}^{\alpha}\cap W_{n})=0.
	\end{cases}
\end{equation}
Then $p^{U,n,k}$ is Borel measurable by \eqref{eq:Pt-measurable}
for $\{\functionspace{P}^{U}_{t}\}_{t\in[0,\infty)}$.
Furthermore for each $(t,x)\in I\times V$, since
$\functionspace{P}^{U}_{t}(x,(\cdot)\cap W_{n})$
is absolutely continuous with respect to $\mu((\cdot)\cap W_{n})$ and
$f^{t,x}_{n}:=\frac{d\functionspace{P}^{U}_{t}(x,(\cdot)\cap W_{n})}{d\mu((\cdot)\cap W_{n})}
	\leq H_{t}(x,\cdot)$ $\mu$-a.e.\ on $W_{n}$ by \eqref{eq:HK-existence4},
$p^{U,n,k}_{t}(x,\cdot)$ is a version of the $\sigmafield{A}_{k}$-conditional
$\frac{\mu((\cdot)\cap W_{n})}{\mu(W_{n})}$-expectation of $f^{t,x}_{n}$
and hence $\lim_{k\to\infty}p^{U,n,k}_{t}(x,y)=f^{t,x}_{n}(y)\leq H_{t}(x,y)$
for $\mu$-a.e.\ $y\in W_{n}$ by the martingale convergence theorem
\cite[Theorem 10.5.1]{Dud:RAP}. Therefore the function
$p^{U,n}_{t}(x,y):=H_{t}(x,y)\wedge\liminf_{k\to\infty}p^{U,n,k}_{t}(x,y)$,
$(t,x,y)\in I\times V\times W_{n}$, has the desired properties.
Now the proof of \textup{(2)} is completed by setting $p^{U}_{t}(x,y):=p^{U,n}_{t}(x,y)$
for $n\in\mathbb{N}$ and $(t,x,y)\in I\times V\times(W_{n}\setminus W_{n-1})$
($W_{0}:=\emptyset$) and using monotone convergence.
\end{proof}

\begin{proof}[Proof of Theorem \textup{\ref{thm:HK-existence}}]
The implication \textup{(2)}$\Rightarrow$\textup{(1)} is immediate, and it is
easy to see from \textup{(UB3)} of Definition \ref{dfn:upper-bound-function}
and Remark \ref{rmk:upper-bound-function}-(1) that \textup{(1)} implies
\textup{(2)}. The implication \textup{(3)}$\Rightarrow$\textup{(2)} also
follows easily since $T^{U}_{t}u=\functionspace{P}^{U}_{t}u$ $\mu$-a.e.\ for
any $t\in(0,\infty)$ and any $u\in\functionspace{B}L^{2}(M,\mu)$.

Therefore it remains to prove \textup{(2)}$\Rightarrow$\textup{(3)}. Let
$\{h_{n}\}_{n\in\mathbb{N}},\{I_{n}\}_{n\in\mathbb{N}},
	\{V_{n}\}_{n\in\mathbb{N}},\{W_{n}\}_{n\in\mathbb{N}}$
be as in (UB3) with $\mu(W_{n})<\infty$ for any
$n\in\mathbb{N}$ as noted in Remark \ref{rmk:upper-bound-function}-(1).
Let $\sigmafield{A}_{k}$ be as in \eqref{eq:HK-existence-filtration} for each
$k\in\mathbb{N}$ and set $\sigmafield{A}:=\bigcup_{k\in\mathbb{N}}\sigmafield{A}_{k}$,
so that $\sigmafield{A}$ is countable, generates $\sigmafield{B}(M)$ and satisfies
$\emptyset\in\sigmafield{A}$, $M\setminus A\in\sigmafield{A}$ for any $A\in\sigmafield{A}$
and $A\cup B\in\sigmafield{A}$ for any $A,B\in\sigmafield{A}$.
By \eqref{eq:Ptu-quasi-continuous-part} and \cite[Theorem 2.1.2-(i)]{FOT},
there exists a non-decreasing sequence $\{F_{k}\}_{k\in\mathbb{N}}$
of closed subsets of $M$ such that $\lim_{k\to\infty}\Capa_{1}(M\setminus F_{k})=0$
and for each $k\in\mathbb{N}$,
$\mu(G\cap F_{k})>0$ for any open subset $G$ of $M$ with $G\cap F_{k}\not=\emptyset$ and
$\{\functionspace{P}^{U}_{t}\ind{A\cap W_{n}}|_{F_{k}}\mid\textrm{$n\in\mathbb{N}$, $t\in J$, $A\in\sigmafield{A}$}\}
	\subset C(F_{k})$.
Moreover, since
\begin{equation*}
\functionspace{P}^{U}_{t}\ind{A\cap W_{n}}(x)
	=\functionspace{P}^{U}_{t-l^{-1}}(\functionspace{P}^{U}_{l^{-1}}\ind{A\cap W_{n}})(x)
	=\mathbb{E}_{x}\bigl[\functionspace{P}^{U}_{l^{-1}}\ind{A\cap W_{n}}(X_{t-l^{-1}})\ind{\{t-l^{-1}<\tau_{U}\}}\bigr]
\end{equation*}
for $l\in\mathbb{N}$ and $t\in[l^{-1},\infty)$ by \eqref{eq:Pt-semigroup} for
$\{\functionspace{P}^{U}_{t}\}_{t\in[0,\infty)}$, an application of
\eqref{eq:Ptu-quasi-continuous-part} and \cite[Theorem 4.2.2]{FOT}
to $\functionspace{P}^{U}_{l^{-1}}\ind{A\cap W_{n}}$
with $l,n\in\mathbb{N}$ and $A\in\sigmafield{A}$ yields an
$\functionspace{E}$-polar set $N_{0}\in\sigmafield{B}(M)$ such that
$(0,\infty)\ni t\mapsto\functionspace{P}^{U}_{t}\ind{A\cap W_{n}}(x)\in\mathbb{R}$
is right-continuous for any $x\in M\setminus N_{0}$, any $n\in\mathbb{N}$ and
any $A\in\sigmafield{A}$. Then $(M\setminus\bigcup_{k\in\mathbb{N}}F_{k})\cup N_{0}$
is $\functionspace{E}$-polar and therefore by \cite[Theorem 4.1.1]{FOT} we can
take a properly exceptional set $N\in\sigmafield{B}(M)$ for $X$ satisfying
$(M\setminus\bigcup_{k\in\mathbb{N}}F_{k})\cup N_{0}\subset N$.

Let $n\in\mathbb{N}$ and $(t,x)\in I\times(V\setminus N)$.
We claim that for any $A\in\sigmafield{B}(M)$,
\begin{equation}\label{eq:HK-existence2-qe}
\functionspace{P}^{U}_{t}\ind{A\cap W_{n}}(x)\leq\int_{A\cap W_{n}}H_{t}(x,y)\,d\mu(y),
\end{equation}
whose limit as $n\to\infty$ results in \eqref{eq:HK-existence4} with $V\setminus N$
in place of $V$ by monotone convergence, thereby proving
\textup{(2)}$\Rightarrow$\textup{(3)} by virtue of Proposition \ref{prop:HK-existence}.
Thus it remains to show \eqref{eq:HK-existence2-qe}. To this end, let
$A\in\sigmafield{A}$ and choose $k\in\mathbb{N}$ with $k\geq n$ so that
$t\in I_{k}$ and $x\in V_{k}\cap F_{k}$.

First we assume $t\in J$. Then
$\functionspace{P}^{U}_{t}\ind{A\cap W_{n}}\leq\int_{A\cap W_{n}}H_{t}(\cdot,y)\,d\mu(y)$
$\mu$-a.e.\ on $V$ by (2), and since $\mu(G\cap V_{k}\cap F_{k})>0$
for any open subset $G$ of $M$ with $x\in G$ we can take
$\{x_{l}\}_{l\in\mathbb{N}}\subset V_{k}\cap F_{k}$ such that
$\lim_{l\to\infty}x_{l}=x$ in $M$ and
$\functionspace{P}^{U}_{t}\ind{A\cap W_{n}}(x_{l})\leq\int_{A\cap W_{n}}H_{t}(x_{l},y)\,d\mu(y)$
for any $l\in\mathbb{N}$. Now \eqref{eq:HK-existence2-qe} follows by utilizing
$\functionspace{P}^{U}_{t}\ind{A\cap W_{n}}|_{F_{k}}\in C(F_{k})$,
Fatou's lemma and \textup{(UB2)} to let $l\to\infty$, where the use of Fatou's
lemma is justified by \eqref{eq:upper-bound-function} with $k$ in place of $n$.

Next for $t\in I\setminus J$, with $k\in\mathbb{N}$ as above, we can take
a strictly decreasing sequence $\{t_{l}\}_{l\in\mathbb{N}}\subset I_{k}\cap J$
satisfying $\lim_{l\to\infty}t_{l}=t$, and then
$\functionspace{P}^{U}_{t_{l}}\ind{A\cap W_{n}}(x)\leq\int_{A\cap W_{n}}H_{t_{l}}(x,y)\,d\mu(y)$
for any $l\in\mathbb{N}$ by the previous paragraph. Now letting $l\to\infty$
yields \eqref{eq:HK-existence2-qe} for this case by the right-continuity of
$\functionspace{P}^{U}_{{\scriptscriptstyle(\cdot)}}\ind{A\cap W_{n}}(x)$,
Fatou's lemma and \textup{(UB1)}, where \eqref{eq:upper-bound-function} with
$k$ in place of $n$ is used again to verify the applicability of Fatou's lemma
to the right-hand side.

Thus \eqref{eq:HK-existence2-qe} has been proved for any $A\in\sigmafield{A}$.
Further, we easily see from \eqref{eq:upper-bound-function} with $k$
in place of $n$ and the dominated convergence theorem that
$\{A\in\sigmafield{B}(M)\mid\textrm{$A$ satisfies \eqref{eq:HK-existence2-qe}}\}$
is closed under monotone countable unions and intersections, and hence
the monotone class theorem \cite[Theorem 4.4.2]{Dud:RAP} implies that
\eqref{eq:HK-existence2-qe} holds for any $A\in\sigmafield{B}(M)$.
\end{proof}

The rest of this section is devoted to presenting examples of $\mu$-upper bound
functions. We start with a lemma which is mostly due to \cite[Subsection 3.4]{GT}.

\begin{lemma}\label{lem:upper-bound-function-Psi}
Let $\Psi:[0,\infty)\to[0,\infty)$ be a homeomorphism satisfying
\begin{equation}\label{eq:upper-bound-function-Psi}
c_{\Psi}^{-1}\Bigl(\frac{R}{r}\Bigr)^{\beta_{1}}\leq\frac{\Psi(R)}{\Psi(r)}
	\leq c_{\Psi}\Bigl(\frac{R}{r}\Bigr)^{\beta_{2}}
	\qquad\textrm{for any }r,R\in(0,\infty)\textrm{ with }r\leq R
\end{equation}
for some $c_{\Psi},\beta_{1},\beta_{2}\in(0,\infty)$ with
$1<\beta_{1}\leq\beta_{2}$, and for $(R,t)\in[0,\infty)\times(0,\infty)$ define
\begin{equation}\label{eq:upper-bound-function-Phi}
\Phi(R,t):=\Phi_{\Psi}(R,t)
	:=\sup_{r\in(0,\infty)}\Bigl\{\frac{R}{r}-\frac{t}{\Psi(r)}\Bigr\}
	=\sup_{\lambda\in(0,\infty)}\Bigl\{\frac{R}{\Psi^{-1}(\lambda^{-1})}-\lambda t\Bigr\}.
\end{equation}
Then $\Phi=\Phi_{\Psi}$ is a $[0,\infty)$-valued lower semi-continuous function
such that for any $R,t\in(0,\infty)$, $\Phi(\cdot,t)$ is non-decreasing,
$\Phi(R,\cdot)$ is non-increasing, $\Phi(0,t)=0<\Phi(R,t)$,
\begin{align}\label{eq:upper-bound-function-Phi-constant}
a\Phi(R,t)&\leq\Phi(aR,t)\quad\textrm{for any }a\in[1,\infty),\\
(c_{\Psi}2^{\beta_{1}})^{-\frac{1}{\beta_{1}-1}}\min_{k\in\{1,2\}}\Bigl(\frac{\Psi(R)}{t}\Bigr)^{\frac{1}{\beta_{k}-1}}
	&\leq\Phi(R,t)\leq c_{\Psi}^{\frac{1}{\beta_{1}-1}}\max_{k\in\{1,2\}}\Bigl(\frac{\Psi(R)}{t}\Bigr)^{\frac{1}{\beta_{k}-1}}.
\label{eq:upper-bound-function-Phi-ULE}
\end{align}
\end{lemma}

\begin{proof}
The lower semi-continuity of $\Phi=\Phi_{\Psi}$ is clear from
\eqref{eq:upper-bound-function-Phi}, and the other assertions
except the upper inequality in \eqref{eq:upper-bound-function-Phi-ULE}
have been verified in \cite[Remark 3.16 and Lemma 3.19]{GT}.
To see the upper inequality in \eqref{eq:upper-bound-function-Phi-ULE},
let $R,t,r\in(0,\infty)$ and set $a:=R\Psi(r)/(rt)$. Noting that
$R/r-t/\Psi(r)=(a-1)t/\Psi(r)\leq 0$ if $a\leq 1$, we assume $a>1$,
and set $\beta:=\beta_{1}$ if $r\leq R$ and $\beta:=\beta_{2}$ if $r>R$.
Then $at/\Psi(r)=R/r\leq(c_{\Psi}\Psi(R)/\Psi(r))^{1/\beta}$
by \eqref{eq:upper-bound-function-Psi}, hence
$\Psi(r)\geq at(c_{\Psi}^{-1}at/\Psi(R))^{\frac{1}{\beta-1}}$, and therefore
\begin{equation*}
\frac{R}{r}-\frac{t}{\Psi(r)}\leq\frac{R}{r}=\frac{at}{\Psi(r)}
	\leq\Bigl(\frac{c_{\Psi}\Psi(R)}{at}\Bigr)^{\frac{1}{\beta-1}}
	\leq c_{\Psi}^{\frac{1}{\beta_{1}-1}}\max_{k\in\{1,2\}}\Bigl(\frac{\Psi(R)}{t}\Bigr)^{\frac{1}{\beta_{k}-1}},
\end{equation*}
where the last inequality follows by $a\geq 1$, $1<\beta_{1}\leq\beta_{2}$
and the fact that $c_{\Psi}\geq 1$ by \eqref{eq:upper-bound-function-Psi}.
Now taking the supremum in $r\in(0,\infty)$ yields the desired inequality.
\end{proof}

\begin{example}\label{exmp:upper-bound-function-Psi}
An important special case of Lemma \ref{lem:upper-bound-function-Psi}
is that of $\Psi(r)=r^{\beta}$ for some $\beta\in(1,\infty)$ treated in
\cite[Example 3.17]{GT}, where $\Phi=\Phi_{\Psi}$ is easily evaluated as
\begin{equation}\label{eq:upper-bound-function-beta}
\Phi(R,t)=\beta^{-\frac{\beta}{\beta-1}}(\beta-1)\Bigl(\frac{R^{\beta}}{t}\Bigr)^{\frac{1}{\beta-1}}.
\end{equation}
\end{example}

The following lemma provides a class of typical $\mu$-upper bound functions,
which has essentially appeared in \cite[(6.10)]{GriHuLau:comp}. Note that
\emph{Lemma \textup{\ref{lem:upper-bound-function}} and Example
\textup{\ref{exmp:upper-bound-function}} below, as well as
Remark \textup{\ref{rmk:upper-bound-function}} above, apply to any locally
compact separable metric space $(M,d)$ and any Radon measure $\mu$ on $M$}
(i.e., any Borel measure on $M$ that is finite on compact sets).

\begin{lemma}\label{lem:upper-bound-function}
Let $\Psi$ and $\Phi=\Phi_{\Psi}$ be as in Lemma \textup{\ref{lem:upper-bound-function-Psi}}.
Let $I\subset(0,\infty)$ be an interval, let $V,W$ be open subsets of $M$ and
let $F=F_{t}(x,y):I\times V\times W\to(0,\infty)$ be a Borel measurable function
satisfying \textup{(UB1)} and \textup{(UB2)} of Definition \textup{\ref{dfn:upper-bound-function}}
and the following \emph{$\Psi$-doubling condition \textup{(DB)$_{\Psi}$}}:
\begin{itemize}[label=\textup{(DB)$_{\Psi}$},align=left,leftmargin=*]
\item[\textup{(DB)$_{\Psi}$}]There exist $\alpha_{F},c_{F}\in(0,\infty)$ such that
	for any $(t,x,y),(s,z,w)\in I\times V\times W$ with $s\leq t$,
	\begin{equation}\label{eq:upper-bound-function-F-DBPsi}
		\frac{F_{s}(z,w)}{F_{t}(x,y)}
		\leq c_{F}\Bigl(\frac{t\vee\Psi(d(x,z))\vee\Psi(d(y,w))}{s}\Bigr)^{\alpha_{F}}.
	\end{equation}
\end{itemize}
Also let $c_{1},c_{2}\in(0,\infty)$ and define
$H=H_{t}(x,y):I\times V\times W\to(0,\infty)$ by
\begin{equation}\label{eq:upper-bound-function-H}
H_{t}(x,y):=F_{t}(x,y)\exp\bigl(-c_{1}\Phi(c_{2}d(x,y),t)\bigr).
\end{equation}
Then $F=F_{t}(x,y)$ and $H=H_{t}(x,y)$ are $\mu$-upper bound functions on $I\times V\times W$.
\end{lemma}

\begin{proof}
It is immediate to see that $H=H_{t}(x,y)$ is Borel measurable and satisfies
\textup{(UB1)} and \textup{(UB2)}, from the corresponding properties of $F=F_{t}(x,y)$
and the lower semi-continuity of $\Phi$. Also \textup{(DB)$_{\Psi}$} easily implies
that $F=F_{t}(x,y)$ and hence $H=H_{t}(x,y)$ are bounded on each compact subset of
$I\times V\times W$, so that they satisfy \textup{(UB3)}
by Remark \ref{rmk:upper-bound-function}-(2).
\end{proof}

\begin{example}\label{exmp:upper-bound-function}
Let $\Psi$ be as in Lemma \textup{\ref{lem:upper-bound-function-Psi}}.
\begin{itemize}[label=\textup{(1)},align=left,leftmargin=*]
\item[\textup{(1)}]A continuous function
	$F=F_{t}(x,y):(0,\infty)\times M\times M\to(0,\infty)$ of the form
	\begin{equation}\label{eq:upper-bound-function-F-ex1}
	F_{t}(x,y)=c_{3}t^{-\alpha_{1}}\bigl(\log(2+t^{-1})\bigr)^{\alpha_{2}}\bigl(\log(2+t)\bigr)^{\alpha_{3}}
	\end{equation}
	for some $c_{3},\alpha_{1}\in(0,\infty)$ and $\alpha_{2},\alpha_{3}\in\mathbb{R}$ clearly
	satisfies \textup{(UB1)}, \textup{(UB2)} and \textup{(DB)$_{\Psi}$}.
\item[\textup{(2)}]
	Let $R\in(0,\infty]$, let $V,W$ be open subsets of $M$ with
	$(\diam V)\vee(\diam W)\leq R$ and let $\nu$ be a Borel measure on $M$
	satisfying the \emph{volume doubling property}
	\begin{equation}\label{eq:VD}
	0<\nu(B(x,2r))\leq c_{\mathrm{vd}}\nu(B(x,r))<\infty
	\end{equation}
	for any $(x,r)\in(V\cup W)\times(0,R)$ for some $c_{\mathrm{vd}}\in(0,\infty)$.
	Then for each $c_{4}\in(0,\infty)$, the function
	$F=F_{t}(x,y):(0,\Psi(R)]\times V\times W\to(0,\infty)$
	($(0,\infty)$ in place of $(0,\Psi(R)]$ for $R=\infty$) defined by
	\begin{equation}\label{eq:upper-bound-function-F-vol}
	F_{t}(x,y):=c_{4}\nu\bigl(B(x,\Psi^{-1}(t))\bigr)^{-1/2}\nu\bigl(B(y,\Psi^{-1}(t))\bigr)^{-1/2}
	\end{equation}
	is easily proved to be upper semi-continuous and satisfy \textup{(DB)$_{\Psi}$}
	thanks to \eqref{eq:VD} and \eqref{eq:upper-bound-function-Psi}, and in particular
	it is Borel measurable and satisfies \textup{(UB1)} and \textup{(UB2)}.
\end{itemize}
\end{example}

\section{Localized upper bounds of heat kernels for diffusions}\label{sec:HKUB}

In this section, we state and prove the main theorem of this paper on deducing
heat kernel upper bounds for $\{\functionspace{P}_{t}\}_{t\in(0,\infty)}$
from those for $\{\functionspace{P}^{U}_{t}\}_{t\in(0,\infty)}$
(Theorem \ref{thm:HKUB-localized} below).
The arguments heavily rely on the decay estimate \eqref{eq:exit-probability}
for the exit probabilities $\mathbb{P}_{x}[\tau_{B(x,r)}\leq t]$, for which
reasonable sufficient conditions will be presented in the next section.
\emph{In the rest of this paper, we fix a homeomorphism
$\Psi:[0,\infty)\to[0,\infty)$ and $c_{\Psi},\beta_{1},\beta_{2}\in(0,\infty)$
with $1<\beta_{1}\leq\beta_{2}$ satisfying \eqref{eq:upper-bound-function-Psi},
and $\Phi=\Phi_{\Psi}$ denotes the function given by \eqref{eq:upper-bound-function-Phi}.}

\emph{Throughout this section, we fix an arbitrary properly exceptional set
$N\in\sigmafield{B}(M)$ for $X$ such that for any
$x\in M\setminus N$,
\begin{equation}\label{eq:X-continuous}
\mathbb{P}_{x}\bigl[[0,\zeta)\ni t\mapsto X_{t}\in M\textrm{ is continuous}\bigr]=1,
\end{equation}}%
where
$\{\textrm{$[0,\zeta)\ni t\mapsto X_{t}\in M$ is continuous}\}\in\sigmafield{F}_{\infty}$
by \cite[Chapter III, 13 and 33]{DM}. According to \cite[Theorem 4.5.1]{FOT},
such $N$ exists if and only if
$(\functionspace{E},\functionspace{F})$ is \emph{local},
i.e., $\functionspace{E}(u,v)=0$ for any $u,v\in\functionspace{F}$ with
$\supp_{\mu}[u],\supp_{\mu}[v]$ compact and $\supp_{\mu}[u]\cap\supp_{\mu}[v]=\emptyset$.
Here for $u\in\functionspace{B}(M)$ or its $\mu$-equivalence class,
$\supp_{\mu}[u]$ denotes its \emph{$\mu$-support} defined as
the smallest closed subset of $M$ such that
$u=0$ $\mu$-a.e.\ on $M\setminus\supp_{\mu}[u]$, which exists since
$M$ has a countable open base for its topology. Note that
$\supp_{\mu}[u]=\overline{u^{-1}(\mathbb{R}\setminus\{0\})}$ for $u\in C(M)$.

\begin{remark}\label{rmk:HKUB}
In fact, \emph{Theorems \textup{\ref{thm:HKUB-localized}},
\textup{\ref{thm:HKUB-localized-global}},
Propositions \textup{\ref{prop:HKUB-localized}} and
\textup{\ref{prop:HKUB-localized-difference}} below apply, without
any changes in the proofs, to any locally compact separable metric space $(M,d)$,
any $\sigma$-finite Borel measure $\mu$ on $M$, any Hunt process $X$ on
$(M,\sigmafield{B}(M))$ and any $N\in\sigmafield{B}(M)$ satisfying
\eqref{eq:properly-exceptional} and \eqref{eq:X-continuous}
for any $x\in M\setminus N$}.
\end{remark}

\begin{theorem}\label{thm:HKUB-localized}
Let $R\in(0,\infty)$, let $U$ be a non-empty open subset of $M$ with $\diam U\leq R$
and let $F=F_{t}(x,y):(0,\Psi(R)]\times U\times U\to(0,\infty)$ be a Borel measurable
function satisfying \textup{(DB)$_{\Psi}$} of Lemma \textup{\ref{lem:upper-bound-function}}
with $I=(0,\Psi(R)]$ and $V=W=U$.
Let $c,\gamma\in(0,\infty)$ and assume that the following two conditions
\textup{(DU)$_{F}^{U,R}$} and \textup{(P)$_{\Psi}^{U,R}$} are fulfilled:
\begin{itemize}[label=\textup{(DU)$_{F}^{U,R}$},align=left,leftmargin=*]
\item[\textup{(DU)$_{F}^{U,R}$}]
	For any $(t,x)\in(0,\Psi(R))\times(U\setminus N)$ and any $A\in\sigmafield{B}(U)$,
	\begin{equation}\label{eq:HKUB-localized-on-diag}
	\functionspace{P}^{U}_{t}(x,A)\leq\int_{A}F_{t}(x,y)\,d\mu(y).
	\end{equation}
\item[\textup{(P)$_{\Psi}^{U,R}$}]
	For any $(x,r)\in(U\setminus N)\times(0,R)$ with $B(x,r)\subset U$
	and any $t\in(0,\infty)$,
	\begin{equation}\label{eq:exit-probability}
	\mathbb{P}_{x}[\tau_{B(x,r)}\leq t]\leq c\exp(-\Phi(\gamma r,t)).
	\end{equation}
\end{itemize}
Let $\varepsilon\in(0,1)$ and set
$U^{\circ}_{\varepsilon R}:=\{x\in M\mid\inf_{y\in M\setminus U}d(x,y)>\varepsilon R\}$
\textup{(note that $U^{\circ}_{\varepsilon R}$ is an open subset of $U$)}.
Then there exists a Borel measurable function
$p=p_{t}(x,y):(0,\infty)\times(M\setminus N)\times U^{\circ}_{\varepsilon R}\to[0,\infty)$
such that for any $(t,x)\in(0,\infty)\times(M\setminus N)$ the following hold:
\begin{equation}\label{eq:HKUB-localized-existence}
\functionspace{P}_{t}(x,A)=\int_{A}p_{t}(x,y)\,d\mu(y)
	\qquad\textrm{for any }A\in\sigmafield{B}(U^{\circ}_{\varepsilon R}),
\end{equation}
and furthermore for any $y\in U^{\circ}_{\varepsilon R}$,
\begin{equation}\label{eq:HKUB-localized-off-diag-U}
p_{t}(x,y)\leq
	\begin{cases}
	c_{\varepsilon}F_{t}(x,y)\exp\bigl(-\Phi(\gamma_{\varepsilon}d(x,y),t)\bigr) &\textrm{if }t<\Psi(R)\textrm{ and }x\in U,\\
	c_{\varepsilon}(\inf_{U\times U}F_{(2t)\wedge\Psi(R)})\exp(-\Phi(\gamma_{\varepsilon}R,t)) &\textrm{if }t<\Psi(R)\textrm{ and }x\not\in U,\\
	c_{\varepsilon}(\inf_{U\times U}F_{\Psi(R)}) &\textrm{if }t\geq\Psi(R)
	\end{cases}
\end{equation}
for some $c_{\varepsilon}\in(0,\infty)$ explicit in
$c_{\Psi},\beta_{1},\beta_{2},c_{F},\alpha_{F},c,\gamma,\varepsilon$ and
$\gamma_{\varepsilon}:=\frac{1}{5}\varepsilon\gamma$.
\end{theorem}

In light of the equivalence stated in Proposition \ref{prop:HK-existence} and
the examples of Borel measurable functions $F=F_{t}(x,y)$ satisfying
\textup{(DB)$_{\Psi}$} in Example \ref{exmp:upper-bound-function},
\textup{(DU)$_{F}^{U,R}$} of Theorem \ref{thm:HKUB-localized} amounts to an
\emph{on-diagonal} upper bound of the heat kernel $p^{U}=p^{U}_{t}(x,y)$ for
$\{\functionspace{P}^{U}_{t}\}_{t\in(0,\infty)}$.
Note that \emph{the two conditions \textup{(DU)$_{F}^{U,R}$} and
\textup{(P)$_{\Psi}^{U,R}$} involve only the part $X^{U}$ of $X$ on $U$
and hence are independent of the behavior of $X$ after exiting $U$},
on account of \eqref{eq:transition-function-part} and the obvious fact that
$\tau_{B}(\omega)=\inf\{t\in[0,\infty)\mid X^{U}_{t}(\omega)\in U_{\cemetery}\setminus B\}$
for $B\subset U$ and $\omega\in\Omega$.

\begin{remark}\label{rmk:HKUB-localized}
Theorem \ref{thm:HK-existence} tells us that \textup{(DU)$_{F}^{U,R}$}
of Theorem \ref{thm:HKUB-localized} is implied,
at the price of replacing $N$ with a larger properly exceptional set for $X$,
by its \emph{``$\mu$-a.e."} counterpart for the Markovian semigroup
$\{T^{U}_{t}\}_{t\in(0,\infty)}$ provided $F=F_{t}(x,y)$ is a
$\mu$-upper bound function on $(0,\Psi(R))\times U\times U$.
Remember, though, that \emph{we have proved Theorem \textup{\ref{thm:HK-existence}}
only for a Radon measure $\mu$ on $M$ with full support and a $\mu$-symmetric
Hunt process $X$ on $(M,\sigmafield{B}(M))$ whose Dirichlet form
$(\functionspace{E},\functionspace{F})$ is regular on $L^{2}(M,\mu)$};
recall Remark \ref{rmk:Dirichlet-form} in this connection.
\end{remark}

We also have a global version of Theorem \ref{thm:HKUB-localized}
for the case where its assumptions are valid on $B(y_{0},\frac{R'}{2})$ for any
$(y_{0},R')\in M\times(0,\infty)$ with $R'\leq R$, as follows. Set $\Psi(\infty):=\infty$.

\begin{theorem}\label{thm:HKUB-localized-global}
Let $\delta\in(0,1]$, let $R\in(0,\infty]$ satisfy $R\geq\delta\diam M$
and let $F=F_{t}(x,y):(0,\Psi(R)]\times M\times M\to(0,\infty)$ be a Borel
measurable function satisfying \textup{(DB)$_{\Psi}$} of
Lemma \textup{\ref{lem:upper-bound-function}} with $I=(0,\Psi(R)]$ and $V=W=M$
\textup{($(0,\infty)$ in place of $(0,\Psi(R)]$ for $R=\infty$)}.
Let $c,\gamma\in(0,\infty)$ and assume that the two conditions
\textup{(DU)$_{F}^{B(y_{0},R'/2),R'}$} and \textup{(P)$_{\Psi}^{B(y_{0},R'/2),R'}$}
from Theorem \textup{\ref{thm:HKUB-localized}} are fulfilled for any
$(y_{0},R')\in M\times(0,\infty)$ with $R'\leq R$.
Then there exists a Borel measurable function
$p=p_{t}(x,y):(0,\infty)\times(M\setminus N)\times M\to[0,\infty)$
such that for any $(t,x)\in(0,\infty)\times(M\setminus N)$,
\eqref{eq:HKUB-localized-existence} with $\sigmafield{B}(M)$ in place of
$\sigmafield{B}(U^{\circ}_{\varepsilon R})$ holds and
\begin{equation}\label{eq:HKUB-localized-off-diag}
p_{t}(x,y)\leq
	\begin{cases}
	c'\delta^{-\beta_{2}\alpha_{F}}F_{t}(x,y)\exp\bigl(-\Phi(\gamma'_{\delta}d(x,y),t)\bigr) &\textrm{if }t<\Psi(R),\\
	c'\delta^{-\beta_{2}\alpha_{F}}(\inf_{M\times M}F_{\Psi(R)}) &\textrm{if }t\geq\Psi(R)
	\end{cases}
\end{equation}
for any $y\in M$ for some $c'\in(0,\infty)$ explicit in
$c_{\Psi},\beta_{1},\beta_{2},c_{F},\alpha_{F},c,\gamma$
and $\gamma'_{\delta}:=\frac{1}{40}\delta\gamma$.
\end{theorem}

The rest of this section is devoted to the proof of
Theorems \ref{thm:HKUB-localized} and \ref{thm:HKUB-localized-global}.
We start with the proof of the following proposition, which, in view of
Proposition \ref{prop:HK-existence}, can be considered
as a localized version of \cite[Theorem 6.3]{GriHuLau:comp}.
Its proof in \cite{GriHuLau:comp} is based on a general comparison inequality
\cite[Theorem 5.1]{GriHuLau:comp} among the heat kernels on different open sets
\emph{which heavily relies on the symmetry of the Markovian semigroups
$\{T^{U}_{t}\}_{t\in(0,\infty)}$}; see also \cite[Theorem 10.4]{Gri:HKfractal}
for an alternative probabilistic proof of the same comparison inequality.
Here we give a new proof \emph{which does not require the $\mu$-symmetry of $X$}.

\begin{proposition}\label{prop:HKUB-localized}
Under the same assumptions as those of Theorem \textup{\ref{thm:HKUB-localized}},
there exists $c'_{\varepsilon}\in(0,\infty)$ explicit in
$c_{\Psi},\beta_{1},\beta_{2},c_{F},\alpha_{F},c,\gamma,\varepsilon$ such that,
with $\gamma_{\varepsilon}:=\frac{1}{5}\varepsilon\gamma$,
for any $(t,x)\in(0,\Psi(R))\times(U\setminus N)$ and
any $A\in\sigmafield{B}(U^{\circ}_{\varepsilon R})$,
\begin{equation}\label{eq:HKUB-localized-off-diag-local}
\functionspace{P}^{U}_{t}(x,A)
	\leq\int_{A}c'_{\varepsilon}F_{t}(x,y)\exp\bigl(-\Phi(\gamma_{\varepsilon}d(x,y),t)\bigr)\,d\mu(y).
\end{equation}
\end{proposition}

\begin{proof}
Let $(t,x)\in(0,\Psi(R))\times(U\setminus N)$.
Let $y_{0}\in U^{\circ}_{\varepsilon R}\setminus\{x\}$, set
$r:=\frac{1}{4}\varepsilon d(x,y_{0})\in(0,\frac{1}{4}\varepsilon R]$
and let $A\in\sigmafield{B}(B(y_{0},r))$.
We first verify \eqref{eq:HKUB-localized-off-diag-local} for such $A$.
Since $A\subset B(y_{0},r)\subset B(y_{0},4r)\subset U$ by
$y_{0}\in U^{\circ}_{\varepsilon R}$ and $r\in(0,\frac{1}{4}\varepsilon R]$, if
$t\geq\Psi((2+4/\varepsilon)r)$ then \eqref{eq:HKUB-localized-off-diag-local}
is immediate from \textup{(DU)$_{F}^{U,R}$}
and the upper inequality in \eqref{eq:upper-bound-function-Phi-ULE},
and therefore we may assume $t<\Psi((2+4/\varepsilon)r)$.
We set $r_{n}:=r+2^{-n/(2\beta_{2})}r$ and
$\sigma_{n}:=\dot{\sigma}_{B(y_{0},r_{n})}$ for $n\in\mathbb{N}$, so that
$B(y_{0},r)\subset B(y_{0},r_{n})\subset B(y_{0},r_{k})$ and hence
$\sigma_{k}\leq\sigma_{n}\leq\dot{\sigma}_{B(y_{0},r)}$ for any $k\in\{1,\dots,n\}$.

Let $\omega\in\{\textrm{$[0,\zeta)\ni s\mapsto X_{s}\in M$ is continuous}\}$.
It is easy to see that for $B\subset M$,
\begin{equation}\label{eq:X-continuous-entrance}
\textrm{if}\quad X_{0}(\omega)\not\in\interior B\quad\textrm{and}\quad
	\dot{\sigma}_{B}(\omega)<\infty\qquad\textrm{then}\qquad
	X_{\dot{\sigma}_{B}}(\omega)\in\partial B.
\end{equation}
Assume further that $\omega\in\{X_{t}\in B(y_{0},r),\,X_{0}=x\}$. Then since
$X_{0}(\omega)=x\not\in B(y_{0},4r)$ by $d(x,y_{0})>\varepsilon d(x,y_{0})=4r$
and $\dot{\sigma}_{B(y_{0},r)}(\omega)\leq t$ by $X_{t}(\omega)\in B(y_{0},r)$,
it follows from \eqref{eq:X-continuous-entrance} that
$X_{\dot{\sigma}_{B(y_{0},r)}}(\omega)\in\partial B(y_{0},r)$ and hence that
\begin{equation}\label{eq:sigma-n-leq-t}
\sigma_{n}(\omega)\leq\dot{\sigma}_{B(y_{0},r)}(\omega)<t
	\qquad\textrm{for any }n\in\mathbb{N}.
\end{equation}
In particular, $\sigma_{n+1}(\omega)\leq\frac{1}{2}(\sigma_{n}(\omega)+t)$
for some $n\in\mathbb{N}$; indeed, otherwise for any $n\in\mathbb{N}$
we would have $\sigma_{n+1}(\omega)\geq\frac{1}{2}(\sigma_{n}(\omega)+t)$, or equivalently
$t-\sigma_{n+1}(\omega)\leq\frac{1}{2}(t-\sigma_{n}(\omega))$, and hence
$0<t-\dot{\sigma}_{B(y_{0},r)}(\omega)\leq t-\sigma_{n}(\omega)\leq 2^{1-n}(t-\sigma_{1}(\omega))$
by \eqref{eq:sigma-n-leq-t}, contradicting $\lim_{n\to\infty}2^{-n}=0$.
Thus, setting $\Omega_{1}:=\Omega$ and
\begin{equation}\label{eq:Omega-n}
\Omega_{n}:=\{\sigma_{k+1}>{\textstyle\frac{1}{2}}(\sigma_{k}+t)\textrm{ for any }k\in\{1,\dots,n-1\}\},
	\quad n\in\mathbb{N}\setminus\{1\},
\end{equation}
we obtain
\begin{multline}\label{eq:sigma-n-t-gap}
\{X_{t}\in B(y_{0},r),\,X_{0}=x,\,[0,\zeta)\ni s\mapsto X_{s}\in M\textrm{ is continuous}\}\\
\subset\bigcup\nolimits_{n\in\mathbb{N}}\bigl(\Omega_{n}\cap\{\sigma_{n+1}\leq{\textstyle\frac{1}{2}}(\sigma_{n}+t)\}\bigr),
	\quad\textrm{where the union is disjoint.}
\end{multline}
Note that $\Omega_{n}\in\sigmafield{F}_{\sigma_{n}}$ for any $n\in\mathbb{N}$
since $\sigmafield{F}_{\sigma_{k}}\subset\sigmafield{F}_{\sigma_{n}}$ by
$\sigma_{k}\leq\sigma_{n}$ and \cite[Lemma 1.2.15]{KS} for any $k\in\{1,\dots,n\}$.
Now by \eqref{eq:sigma-n-t-gap} along with $A\subset B(y_{0},r)$,
$\mathbb{P}_{x}[X_{0}=x]=1$ and \eqref{eq:X-continuous},
\begin{align}
\functionspace{P}^{U}_{t}(x,A)&=\mathbb{P}_{x}[X_{t}\in A,\,t<\tau_{U}]\notag\\
&=\mathbb{P}_{x}\Bigl[\{X_{t}\in A,\,t<\tau_{U}\}\cap
	\bigcup\nolimits_{n\in\mathbb{N}}\bigl(\Omega_{n}\cap\{\sigma_{n+1}\leq{\textstyle\frac{1}{2}}(\sigma_{n}+t)\}\bigr)\Bigr]\notag\\
&=\sum_{n\in\mathbb{N}}\mathbb{P}_{x}\bigl[\{X_{t}\in A,\,t<\tau_{U}\}
	\cap\bigl(\Omega_{n}\cap\{\sigma_{n+1}\leq{\textstyle\frac{1}{2}}(\sigma_{n}+t)\}\bigr)\bigr].
\label{eq:PUtxA-expansion}
\end{align}

Let $n\in\mathbb{N}$, set $\sigma_{n,t}:=\sigma_{n}\wedge t$ and
$\Omega'_{n}:=\Omega_{n}\cap\{\sigma_{n}\leq t,\,\sigma_{n,t}\leq\tau_{U}\}$,
so that $\Omega'_{n}\in\sigmafield{F}_{\sigma_{n,t}}$ by
$\Omega_{n}\in\sigmafield{F}_{\sigma_{n}}$ and \cite[Lemma 1.2.16]{KS}.
Then $\{X_{t}\in A\}\subset\{\sigma_{n}\leq t\}$ by $A\subset B(y_{0},r_{n})$,
clearly $\tau_{U}=\sigma_{n,t}+\tau_{U}\circ\theta_{\sigma_{n,t}}$
on $\{\sigma_{n,t}\leq\tau_{U}\}$,
and by $\sigma_{n}\leq\sigma_{n+1}$
we also have $\sigma_{n+1}=\sigma_{n}+\sigma_{n+1}\circ\theta_{\sigma_{n}}$,
which easily implies that
$\{\sigma_{n}\leq t,\,\sigma_{n+1}\leq\frac{1}{2}(\sigma_{n}+t)\}
	=\{\sigma_{n}+2\sigma_{n+1}\circ\theta_{\sigma_{n}}\leq t\}$.
Therefore, 
\begin{align}
&\{X_{t}\in A,\,t<\tau_{U}\}\cap\bigl(\Omega_{n}\cap\{\sigma_{n+1}\leq{\textstyle\frac{1}{2}}(\sigma_{n}+t)\}\notag\\
&=\{X_{t}\in A\}\cap\Omega_{n}\cap\{\sigma_{n}\leq t<\tau_{U},\,\sigma_{n+1}\leq{\textstyle\frac{1}{2}}(\sigma_{n}+t)\}\notag\\
&=\{X_{t}\in A\}\cap\Omega_{n}\cap\{\sigma_{n}\leq t<\tau_{U},\,\sigma_{n,t}\leq\tau_{U},\,\sigma_{n}+2\sigma_{n+1}\circ\theta_{\sigma_{n}}\leq t\}\notag\\
&=\{X_{t}\in A\}\cap\Omega'_{n}
	\cap\{\sigma_{n,t}+2\sigma_{n+1}\circ\theta_{\sigma_{n,t}}
		\leq t<\sigma_{n,t}+\tau_{U}\circ\theta_{\sigma_{n,t}}\}.
\label{eq:PUtxA-expansion-n-event}
\end{align}
Noting that $(\sigma_{n,t}+2\sigma_{n+1}\circ\theta_{\sigma_{n,t}})
	\wedge(\sigma_{n,t}+\tau_{U}\circ\theta_{\sigma_{n,t}})
	=\sigma_{n,t}+((2\sigma_{n+1})\wedge\tau_{U})\circ\theta_{\sigma_{n,t}}$,
we see from \eqref{eq:PUtxA-expansion-n-event},
$\Omega'_{n}\in\sigmafield{F}_{\sigma_{n,t}}$ and Proposition \ref{prop:strong-Markov} that
\begin{align}
&\mathbb{P}_{x}\bigl[\{X_{t}\in A,\,t<\tau_{U}\}\cap\bigl(\Omega_{n}\cap\{\sigma_{n+1}\leq{\textstyle\frac{1}{2}}(\sigma_{n}+t)\}\bigr)\bigr]\notag\\
&=\mathbb{E}_{x}\bigl[\ind{A}(X_{t})
	\ind{\Omega'_{n}\cap\{\sigma_{n,t}+2\sigma_{n+1}\circ\theta_{\sigma_{n,t}}
		\leq t<\sigma_{n,t}+\tau_{U}\circ\theta_{\sigma_{n,t}}\}}\bigr]\notag\\
&=\mathbb{E}_{x}\bigl[\ind{\Omega'_{n}}\ind{A}(X_{t})\bigl(\ind{\{t<\sigma_{n,t}+\tau_{U}\circ\theta_{\sigma_{n,t}}\}}
	-\ind{\{t<\sigma_{n,t}+((2\sigma_{n+1})\wedge\tau_{U})\circ\theta_{\sigma_{n,t}}\}}\bigr)\bigr]\notag\\
&=\mathbb{E}_{x}\bigl[\ind{\Omega'_{n}}\mathbb{E}_{x}\bigl[\ind{A}(X_{t})\ind{\{t<\sigma_{n,t}+\tau_{U}\circ\theta_{\sigma_{n,t}}\}}\bigm|\sigmafield{F}_{\sigma_{n,t}}\bigr]\bigr]\notag\\
	&\qquad\qquad-\mathbb{E}_{x}\bigl[\ind{\Omega'_{n}}\mathbb{E}_{x}\bigl[\ind{A}(X_{t})\ind{\{t<\sigma_{n,t}+((2\sigma_{n+1})\wedge\tau_{U})\circ\theta_{\sigma_{n,t}}\}}\bigm|\sigmafield{F}_{\sigma_{n,t}}\bigr]\bigr]\notag\\
&=\int_{\Omega'_{n}}\mathbb{E}_{X_{\sigma_{n,t}}(\omega)}\bigl[\ind{A}(X_{t-\sigma_{n,t}(\omega)})\ind{\{t-\sigma_{n,t}(\omega)<\tau_{U}\}}\bigr]\,d\mathbb{P}_{x}(\omega)\notag\\
	&\qquad\qquad-\int_{\Omega'_{n}}\mathbb{E}_{X_{\sigma_{n,t}}(\omega)}\bigl[\ind{A}(X_{t-\sigma_{n,t}(\omega)})\ind{\{t-\sigma_{n,t}(\omega)<(2\sigma_{n+1})\wedge\tau_{U}\}}\bigr]\,d\mathbb{P}_{x}(\omega)\notag\\
&=\int_{\Omega'_{n}}\mathbb{E}_{X_{\sigma_{n,t}}(\omega)}\bigl[\ind{A}(X_{t-\sigma_{n,t}(\omega)})\bigl(\ind{\{t-\sigma_{n,t}(\omega)<\tau_{U}\}}-\ind{\{t-\sigma_{n,t}(\omega)<(2\sigma_{n+1})\wedge\tau_{U}\}}\bigr)\bigr]\,d\mathbb{P}_{x}(\omega)\notag\\
&=\int_{\Omega'_{n}}\mathbb{E}_{X_{\sigma_{n,t}}(\omega)}\bigl[\ind{A}(X_{t-\sigma_{n,t}(\omega)})\ind{\{2\sigma_{n+1}\leq t-\sigma_{n,t}(\omega)<\tau_{U}\}}\bigr]\,d\mathbb{P}_{x}(\omega)\notag\\
&=\int_{\Omega'_{n}\cap\{X_{\sigma_{n}}\in(\partial B(y_{0},r_{n}))\setminus N\}}\mathbb{E}_{X_{\sigma_{n,t}}(\omega)}\bigl[\ind{A}(X_{t-\sigma_{n,t}(\omega)})\ind{\{2\sigma_{n+1}\leq t-\sigma_{n,t}(\omega)<\tau_{U}\}}\bigr]\,d\mathbb{P}_{x}(\omega),
\label{eq:PUtxA-expansion-n-step1}
\end{align}
where the equality in the last line follows since
$\ind{\{\sigma_{n}\leq t\}}=\ind{\{\sigma_{n}\leq t,\,X_{\sigma_{n}}\in(\partial B(y_{0},r_{n}))\setminus N\}}$
$\mathbb{P}_{x}$-a.s.\ by $x\in M\setminus(N\cup B(y_{0},4r))$, $\mathbb{P}_{x}[X_{0}=x]=1$,
\eqref{eq:X-continuous}, \eqref{eq:X-continuous-entrance} and \eqref{eq:properly-exceptional}.

Let $\omega\in\Omega'_{n}\cap\{X_{\sigma_{n}}\in(\partial B(y_{0},r_{n}))\setminus N\}$,
set $s:=t-\sigma_{n,t}(\omega)$ and $z:=X_{\sigma_{n,t}}(\omega)$, so that
$\sigma_{n,t}(\omega)=\sigma_{n}(\omega)$, $s=t-\sigma_{n}(\omega)\in[0,t]$
and $z=X_{\sigma_{n}}(\omega)\in(\partial B(y_{0},r_{n}))\setminus N$ by
$\sigma_{n}(\omega)\leq t$. The integrand in \eqref{eq:PUtxA-expansion-n-step1}
is $\mathbb{E}_{z}[\ind{A}(X_{s})\ind{\{2\sigma_{n+1}\leq s<\tau_{U}\}}]$,
which is $0$ if $s=0$ by $\mathbb{P}_{z}[X_{0}=z]=1$ and
$z\not\in B(y_{0},r_{n})\supset B(y_{0},r)\supset A$.
Assume $s>0$ and set $\sigma_{n+1,s}:=\sigma_{n+1}\wedge s$.
Noting that $\tau_{U}=\sigma_{n+1,s}+\tau_{U}\circ\theta_{\sigma_{n+1,s}}$
on $\{\sigma_{n+1,s}\leq\tau_{U}\}$ and that
$\{\sigma_{n+1}\leq\frac{s}{2},\,\sigma_{n+1,s}\leq\tau_{U}\}=\{\sigma_{n+1}\leq\tau_{U}\wedge\frac{s}{2}\}\in\sigmafield{F}_{\sigma_{n+1,s}}$
by \cite[Lemma 1.2.16]{KS}, we see from Proposition \ref{prop:strong-Markov} that
\begin{align}
&\mathbb{E}_{z}[\ind{A}(X_{s})\ind{\{2\sigma_{n+1}\leq s<\tau_{U}\}}]\notag\\
&=\mathbb{E}_{z}\bigl[\ind{A}(X_{s})\ind{\{s<\tau_{U}\}}\ind{\{\sigma_{n+1}\leq s/2,\,\sigma_{n+1,s}\leq\tau_{U}\}}\bigr]\notag\\
&=\mathbb{E}_{z}\bigl[\ind{A}(X_{s})\ind{\{s<\sigma_{n+1,s}+\tau_{U}\circ\theta_{\sigma_{n+1,s}}\}}\ind{\{\sigma_{n+1}\leq\tau_{U}\wedge(s/2)\}}\bigr]\notag\\
&=\mathbb{E}_{z}\bigl[\ind{\{\sigma_{n+1}\leq\tau_{U}\wedge(s/2)\}}
	\mathbb{E}_{z}\bigl[\ind{A}(X_{s})\ind{\{s<\sigma_{n+1,s}+\tau_{U}\circ\theta_{\sigma_{n+1,s}}\}}\bigm|\sigmafield{F}_{\sigma_{n+1,s}}\bigr]\bigr]\notag\\
&=\int_{\{\sigma_{n+1}\leq\tau_{U}\wedge(s/2)\}}\mathbb{E}_{X_{\sigma_{n+1,s}}(\omega')}
	\bigl[\ind{A}(X_{s-\sigma_{n+1,s}(\omega')})\ind{\{s-\sigma_{n+1,s}(\omega')<\tau_{U}\}}\bigr]\,d\mathbb{P}_{z}(\omega')\notag\\
&=\mathbb{E}_{z}\bigl[\ind{\{\sigma_{n+1}\leq\tau_{U}\wedge(s/2),\,X_{\sigma_{n+1}}\in(\partial B(y_{0},r_{n+1}))\setminus N\}}
	\functionspace{P}^{U}_{s-\sigma_{n+1}}(X_{\sigma_{n+1}},A)\bigr],
\label{eq:PUtxA-expansion-n-step2}
\end{align}
where again the last equality follows since
$\ind{\{\sigma_{n+1}\leq s/2\}}=\ind{\{\sigma_{n+1}\leq s/2,\,X_{\sigma_{n+1}}\in(\partial B(y_{0},r_{n+1}))\setminus N\}}$
$\mathbb{P}_{z}$-a.s.\ by $z\in M\setminus(N\cup B(y_{0},r_{n}))$, $\mathbb{P}_{z}[X_{0}=z]=1$,
\eqref{eq:X-continuous}, \eqref{eq:X-continuous-entrance} and \eqref{eq:properly-exceptional}.

Further let
$\omega'\in\{\sigma_{n+1}\leq\tau_{U}\wedge\frac{s}{2},\,X_{\sigma_{n+1}}\in(\partial B(y_{0},r_{n+1}))\setminus N\}$,
set $u:=s-\sigma_{n+1}(\omega')$ and $w:=X_{\sigma_{n+1}}(\omega')$, so that
$0<\frac{s}{2}\leq u\leq s\leq t<\Psi(R)$,
$w\in(\partial B(y_{0},r_{n+1}))\setminus N\subset U\setminus N$ and
$d(w,x)\leq d(w,y_{0})+d(y_{0},x)=r_{n+1}+4r/\varepsilon<(2+4/\varepsilon)r$.
Then by \textup{(DB)$_{\Psi}$} and the assumption that $t<\Psi((2+4/\varepsilon)r)$,
\begin{equation*}
\frac{F_{u}(w,y)}{F_{t}(x,y)}
	\leq c_{F}\Bigl(\frac{t\vee\Psi(d(w,x))}{u}\Bigr)^{\alpha_{F}}
	\leq c_{F}\Bigl(\frac{\Psi((2+4/\varepsilon)r)}{s/2}\Bigr)^{\alpha_{F}}
\end{equation*}
for any $y\in U$, which together with $A\subset U$ and \textup{(DU)$_{F}^{U,R}$} yields
\begin{align}
\functionspace{P}^{U}_{s-\sigma_{n+1}(\omega')}(X_{\sigma_{n+1}}(\omega'),A)
	=\functionspace{P}^{U}_{u}(w,A)&\leq\int_{A}F_{u}(w,y)\,d\mu(y)\notag\\
&\leq c_{F}\Bigl(\frac{\Psi((2+4/\varepsilon)r)}{s/2}\Bigr)^{\alpha_{F}}
	\int_{A}F_{t}(x,y)\,d\mu(y).
\label{eq:PUtxA-expansion-n-step3}
\end{align}
Recalling that
$z=X_{\sigma_{n}}(\omega)\in(\partial B(y_{0},r_{n}))\setminus N$, we have
$(z,r_{n}-r_{n+1})\in(U\setminus N)\times(0,R)$, $B(z,r_{n}-r_{n+1})\subset U$,
and $\tau_{B(z,r_{n}-r_{n+1})}\leq\sigma_{n+1}$ by
$B(y_{0},r_{n+1})\subset M\setminus B(z,r_{n}-r_{n+1})$.
Therefore it follows from \eqref{eq:PUtxA-expansion-n-step2},
\eqref{eq:PUtxA-expansion-n-step3} and
\textup{(P)$_{\Psi}^{U,R}$} for $(z,r_{n}-r_{n+1})$ that
\begin{align}
&\mathbb{E}_{z}[\ind{A}(X_{s})\ind{\{2\sigma_{n+1}\leq s<\tau_{U}\}}]\notag\\
&\leq c_{F}\Bigl(\frac{\Psi((2+4/\varepsilon)r)}{s/2}\Bigr)^{\alpha_{F}}
	\mathbb{P}_{z}[\sigma_{n+1}\leq\tau_{U}\wedge{\textstyle\frac{s}{2}}]\int_{A}F_{t}(x,y)\,d\mu(y)\notag\\
&\leq c_{F}\Bigl(\frac{\Psi((2+4/\varepsilon)r)}{s/2}\Bigr)^{\alpha_{F}}
	\mathbb{P}_{z}[\tau_{B(z,r_{n}-r_{n+1})}\leq{\textstyle\frac{s}{2}}]\int_{A}F_{t}(x,y)\,d\mu(y)\notag\\
&\leq cc_{F}\Bigl(\frac{\Psi((2+4/\varepsilon)r)}{s/2}\Bigr)^{\alpha_{F}}
	\exp\bigl(-\Phi(\gamma(r_{n}-r_{n+1}),{\textstyle\frac{s}{2}})\bigr)\int_{A}F_{t}(x,y)\,d\mu(y).
\label{eq:PUtxA-expansion-n-step4}
\end{align}

We easily see from $\omega\in\Omega'_{n}\subset\Omega_{n}$ and \eqref{eq:Omega-n}
that $0<s=t-\sigma_{n}(\omega)\leq 2^{1-n}t$, and then by $t<\Psi((2+4/\varepsilon)r)$
we have $\Psi((2+4/\varepsilon)r)/(2^{n/2}s/2)\geq 2^{n/2}>1$,
which together with \eqref{eq:upper-bound-function-Phi-ULE},
$r_{n}-r_{n+1}=(1-2^{-1/(2\beta_{2})})2^{-n/(2\beta_{2})}r$,
\eqref{eq:upper-bound-function-Psi} and $1<\beta_{1}\leq\beta_{2}$ implies that
\begin{align}
\Phi(\gamma(r_{n}-r_{n+1}),{\textstyle\frac{s}{2}})
	&\geq(c_{\Psi}2^{\beta_{1}})^{-\frac{1}{\beta_{1}-1}}
		\min_{k\in\{1,2\}}\Bigl(\frac{\Psi(\gamma(r_{n}-r_{n+1}))}{s/2}\Bigr)^{\frac{1}{\beta_{k}-1}}\notag\\
&\geq c_{\varepsilon,1}\min_{k\in\{1,2\}}\Bigl(\frac{\Psi((2+4/\varepsilon)r)}{2^{n/2}s/2}\Bigr)^{\frac{1}{\beta_{k}-1}}\notag\\
&=c_{\varepsilon,1}\Bigl(\frac{\Psi((2+4/\varepsilon)r)}{2^{n/2}s/2}\Bigr)^{\frac{1}{\beta_{2}-1}}
	\geq c_{\varepsilon,1}2^{\frac{n}{2(\beta_{2}-1)}}
\label{eq:PUtxA-expansion-n-step5}
\end{align}
for some $c_{\varepsilon,1}\in(0,\infty)$ explicit in $c_{\Psi},\beta_{1},\beta_{2},\gamma,\varepsilon$.
\eqref{eq:PUtxA-expansion-n-step5} in turn yields
\begin{align}
&\Bigl(\frac{\Psi((2+4/\varepsilon)r)}{s/2}\Bigr)^{\alpha_{F}}
	\exp\bigl(-\Phi(\gamma(r_{n}-r_{n+1}),{\textstyle\frac{s}{2}})\bigr)\notag\\
&\leq\Bigl(\frac{\Psi((2+4/\varepsilon)r)}{s/2}\Bigr)^{\alpha_{F}}
	\exp\biggl(-c_{\varepsilon,1}\Bigl(\frac{\Psi((2+4/\varepsilon)r)}{2^{n/2}s/2}\Bigr)^{\frac{1}{\beta_{2}-1}}\biggr)\notag\\
&\leq\Bigl(\frac{\Psi((2+4/\varepsilon)r)}{2^{n/2}s/2}\Bigr)^{\alpha_{F}}2^{\alpha_{F}n/2}
	\exp\biggl(-\frac{c_{\varepsilon,1}}{2}\Bigl(\frac{\Psi((2+4/\varepsilon)r)}{2^{n/2}s/2}\Bigr)^{\frac{1}{\beta_{2}-1}}
		-\frac{c_{\varepsilon,1}}{2}2^{\frac{n}{2(\beta_{2}-1)}}\biggr)\notag\\
&\leq c_{\varepsilon,2}2^{-\alpha_{F}n/2},
\label{eq:PUtxA-expansion-n-step6}
\end{align}
where
$c_{\varepsilon,2}:=2^{5\alpha_{F}(\beta_{2}-1)}
	\bigl(\alpha_{F}(\beta_{2}-1)/(ec_{\varepsilon,1})\bigr)^{3\alpha_{F}(\beta_{2}-1)}$.
By \eqref{eq:PUtxA-expansion-n-step4} and \eqref{eq:PUtxA-expansion-n-step6},
\begin{equation}\label{eq:PUtxA-expansion-n-step7}
\mathbb{E}_{z}[\ind{A}(X_{s})\ind{\{2\sigma_{n+1}\leq s<\tau_{U}\}}]
	\leq\frac{cc_{F}c_{\varepsilon,2}}{2^{\alpha_{F}n/2}}\int_{A}F_{t}(x,y)\,d\mu(y)
\end{equation}
for $s=t-\sigma_{n,t}(\omega)$ and $z=X_{\sigma_{n,t}}(\omega)$ for any
$\omega\in\Omega'_{n}\cap\{X_{\sigma_{n}}\in(\partial B(y_{0},r_{n}))\setminus N\}$,
and therefore from \eqref{eq:PUtxA-expansion-n-step1}, \eqref{eq:PUtxA-expansion-n-step7}
and $\Omega'_{n}\subset\{\sigma_{n}\leq t\}$ we obtain
\begin{multline}\label{eq:PUtxA-expansion-n-step8}
\mathbb{P}_{x}\bigl[\{X_{t}\in A,\,t<\tau_{U}\}\cap\bigl(\Omega_{n}\cap\{\sigma_{n+1}\leq{\textstyle\frac{1}{2}}(\sigma_{n}+t)\}\bigr)\bigr]\\
	\leq\frac{cc_{F}c_{\varepsilon,2}}{2^{\alpha_{F}n/2}}\mathbb{P}_{x}[\sigma_{n}\leq t]\int_{A}F_{t}(x,y)\,d\mu(y).
\end{multline}

To conclude \eqref{eq:HKUB-localized-off-diag-local} from \eqref{eq:PUtxA-expansion}
and \eqref{eq:PUtxA-expansion-n-step8}, we show that
\begin{equation}\label{eq:PUtxA-expansion-n-step9}
\mathbb{P}_{x}[\sigma_{n}\leq t]\leq c\exp(-\Phi(\gamma r,t)).
\end{equation}
Indeed, setting $\sigma:=\dot{\sigma}_{B(y_{0},3r)}$, we have $\sigma\leq\sigma_{n}$
by $B(y_{0},r_{n})\subset B(y_{0},3r)$ and hence
$\sigma_{n}=\sigma+\sigma_{n}\circ\theta_{\sigma}$. Therefore
$\{\sigma_{n}\leq t\}\subset\{\sigma\leq t,\,\sigma_{n}\circ\theta_{\sigma}\leq t\}$,
and then by the strong Markov property \cite[Theorem A.1.21]{CF} of $X$ at time $\sigma$,
\begin{align}
\mathbb{P}_{x}[\sigma_{n}\leq t]
	\leq\mathbb{P}_{x}[\sigma\leq t,\,\sigma_{n}\circ\theta_{\sigma}\leq t]
	&=\mathbb{E}_{x}[\ind{\{\sigma\leq t\}}(\ind{\{\sigma_{n}\leq t\}}\circ\theta_{\sigma})]\notag\\
&=\mathbb{E}_{x}\bigl[\ind{\{\sigma\leq t\}}\mathbb{E}_{X_{\sigma}}[\ind{\{\sigma_{n}\leq t\}}]\bigr]\notag\\
&=\mathbb{E}_{x}\bigl[\ind{\{\sigma\leq t,\,X_{\sigma}\in(\partial B(y_{0},3r))\setminus N\}}
		\mathbb{P}_{X_{\sigma}}[\sigma_{n}\leq t]\bigr],
\label{eq:PUtxA-expansion-n-step10}
\end{align}
where the last equality follows since
$\ind{\{\sigma\leq t\}}=\ind{\{\sigma\leq t,\,X_{\sigma}\in(\partial B(y_{0},3r))\setminus N\}}$
$\mathbb{P}_{x}$-a.s.\ by $x\in M\setminus(N\cup B(y_{0},4r))$, $\mathbb{P}_{x}[X_{0}=x]=1$,
\eqref{eq:X-continuous}, \eqref{eq:X-continuous-entrance} and \eqref{eq:properly-exceptional}.
Moreover, for $z\in(\partial B(y_{0},3r))\setminus N$,
$B(y_{0},r_{n})\subset M\setminus B(z,r)$ by $r_{n}<2r$,
hence $\sigma_{n}\geq\tau_{B(z,r)}$, and therefore noting that
$(z,r)\in(U\setminus N)\times(0,R)$ and that $B(z,r)\subset B(y_{0},4r)\subset U$,
we see from \textup{(P)$_{\Psi}^{U,R}$} for $(z,r)$ that
$\mathbb{P}_{z}[\sigma_{n}\leq t]\leq\mathbb{P}_{z}[\tau_{B(z,r)}\leq t]
	\leq c\exp(-\Phi(\gamma r,t))$,
which together with \eqref{eq:PUtxA-expansion-n-step10} yields
\eqref{eq:PUtxA-expansion-n-step9}.

Now \eqref{eq:HKUB-localized-off-diag-local} with
$c'_{\varepsilon}:=c^{2}c_{F}c_{\varepsilon,2}/(2^{\alpha_{F}/2}-1)$ is
immediate from \eqref{eq:PUtxA-expansion}, \eqref{eq:PUtxA-expansion-n-step8},
\eqref{eq:PUtxA-expansion-n-step9} and the fact that
$d(x,y)\leq d(x,y_{0})+d(y_{0},y)<4r/\varepsilon+r<5r/\varepsilon$
for any $y\in A$ by $A\subset B(y_{0},r)$.

Finally, we prove \eqref{eq:HKUB-localized-off-diag-local} for general
$A\in\sigmafield{B}(U^{\circ}_{\varepsilon R})$. Note that
\eqref{eq:HKUB-localized-off-diag-local} holds for $A=\{x\}$ by
\textup{(DU)$_{F}^{U,R}$}. In particular, \eqref{eq:HKUB-localized-off-diag-local}
is valid for any $A\in\sigmafield{B}(U^{\circ}_{\varepsilon R})$ if
$U^{\circ}_{\varepsilon R}\setminus\{x\}=\emptyset$, and thus we may assume
$U^{\circ}_{\varepsilon R}\setminus\{x\}\not=\emptyset$. Let $\{y_{k}\}_{k\in\mathbb{N}}$
be a countable dense subset of $U^{\circ}_{\varepsilon R}\setminus\{x\}$
and set $B_{0}:=U^{\circ}_{\varepsilon R}\cap\{x\}$,
$B_{1}:=B(y_{1},\frac{1}{4}\varepsilon d(x,y_{1}))$ and
$B_{k}:=B(y_{k},\frac{1}{4}\varepsilon d(x,y_{k}))\setminus\bigcup_{j=1}^{k-1}B(y_{j},\frac{1}{4}\varepsilon d(x,y_{j}))$
for $k\in\mathbb{N}\setminus\{1\}$. Then
$\{B_{k}\}_{k\in\mathbb{N}\cup\{0\}}\subset\sigmafield{B}(U)$, and it is easy to
see that $U^{\circ}_{\varepsilon R}\subset\bigcup_{k\in\mathbb{N}\cup\{0\}}B_{k}$,
where the union is disjoint. Now for any $A\in\sigmafield{B}(U^{\circ}_{\varepsilon R})$,
since $A\cap B_{0}\in\{\emptyset,\{x\}\}$ and
$A\cap B_{k}\in\sigmafield{B}\bigl(B(y_{k},\frac{1}{4}\varepsilon d(x,y_{k}))\bigr)$
for $k\in\mathbb{N}$, we have already proved \eqref{eq:HKUB-localized-off-diag-local}
with $A\cap B_{k}$ in place of $A$ for any $k\in\mathbb{N}\cup\{0\}$, and therefore
\begin{align}
\functionspace{P}^{U}_{t}(x,A)
	&=\functionspace{P}^{U}_{t}\Bigl(x,\bigcup\nolimits_{k\in\mathbb{N}\cup\{0\}}(A\cap B_{k})\Bigr)
	=\sum_{k\in\mathbb{N}\cup\{0\}}\functionspace{P}^{U}_{t}(x,A\cap B_{k})\notag\\
&\leq\sum_{k\in\mathbb{N}\cup\{0\}}\int_{A\cap B_{k}}c'_{\varepsilon}F_{t}(x,y)\exp\bigl(-\Phi(\gamma_{\varepsilon}d(x,y),t)\bigr)\,d\mu(y)\notag\\
&=\int_{A}c'_{\varepsilon}F_{t}(x,y)\exp\bigl(-\Phi(\gamma_{\varepsilon}d(x,y),t)\bigr)\,d\mu(y)
\label{eq:HKUB-localized-off-diag-local-end}
\end{align}
by monotone convergence, completing the proof of Proposition \ref{prop:HKUB-localized}.
\end{proof}

Theorems \ref{thm:HKUB-localized} and \ref{thm:HKUB-localized-global}
are easy consequences of Propositions \ref{prop:HK-existence},
\ref{prop:HKUB-localized}, and \ref{prop:HKUB-localized-difference} below.

\begin{proposition}\label{prop:HKUB-localized-difference}
Under the same assumptions as those of Theorem \textup{\ref{thm:HKUB-localized}},
there exists $c''_{\varepsilon}\in(0,\infty)$ explicit in
$c_{\Psi},\beta_{1},\beta_{2},c_{F},\alpha_{F},c,\gamma,\varepsilon$ such that, with
$\gamma_{\varepsilon}:=\frac{1}{5}\varepsilon\gamma$,
for any $(t,x)\in(0,\Psi(R))\times(M\setminus N)$ and
any $A\in\sigmafield{B}(U^{\circ}_{\varepsilon R})$,
\begin{equation}\label{eq:HKUB-localized-difference}
\functionspace{P}_{t}(x,A)\leq\functionspace{P}^{U}_{t}(x,A)
	+c''_{\varepsilon}(\inf\nolimits_{U\times U}F_{t})
	\exp(-\Phi(\gamma_{\varepsilon}R,t))\mu(A).
\end{equation}
\end{proposition}

\begin{proof}
If $U=M$, then \eqref{eq:HKUB-localized-difference} is trivially valid since
$X_{t}=X^{U}_{t}$ and hence $\functionspace{P}_{t}=\functionspace{P}^{U}_{t}$
for any $t\in[0,\infty)$. Therefore we may assume $U\not=M$.
Set $B:=U^{\circ}_{(\varepsilon/2)R}$, so that $B$ is open in $M$ and
$\overline{B}\subset U$, and define $\sigmafield{F}_{*}$-stopping times $\tau_{n}$
and $\sigma_{n}$, $n\in\mathbb{N}$, by \eqref{eq:multiple-DH-stopping-times}.
For each $n\in\mathbb{N}$, as noted at the beginning of the proof of
Theorem \ref{thm:multiple-DH}, on $\{\sigma_{n}<\infty\}$ we have
$X_{\sigma_{n}}\in\overline{B}\subset U$, $\tau_{n}\leq\sigma_{n}<\zeta$,
hence $X_{\tau_{n}}\in M\setminus U$ and $\tau_{n}<\sigma_{n}$
by the sample path right-continuity of $X$, and we also easily see that
\begin{equation}\label{eq:multiple-DH-sigma-n-boundary}
X_{\sigma_{n}}\in(\partial B)\setminus N\quad\textrm{on}\quad
	\{\sigma_{n}<\infty=\dot{\sigma}_{N},\,[0,\zeta)\ni t\mapsto X_{t}\in M\textrm{ is continuous}\}.
\end{equation}

Let $(t,x)\in(0,\Psi(R))\times(M\setminus N)$ and
$A\in\sigmafield{B}(U^{\circ}_{\varepsilon R})$. Since
$\ind{A}|_{M\setminus B}=0$ by $A\subset U^{\circ}_{\varepsilon R}\subset B$,
from Theorem \ref{thm:multiple-DH} with $u=\ind{A}$ we obtain
\begin{equation}\label{eq:multiple-DH-apply}
\functionspace{P}_{t}(x,A)=\functionspace{P}^{U}_{t}(x,A)
	+\sum_{n\in\mathbb{N}}\mathbb{E}_{x}\bigl[\ind{\{\sigma_{n}\leq t\}}
		\functionspace{P}^{U}_{t-\sigma_{n}}(X_{\sigma_{n}},A)\bigr].
\end{equation}

Noting \eqref{eq:multiple-DH-sigma-n-boundary}, to estimate each term of the series
in \eqref{eq:multiple-DH-apply} let $s\in[0,t]$, $z\in(\partial B)\setminus N$ and
let $c'_{\varepsilon}\in(0,\infty)$ and $\gamma_{\varepsilon}=\frac{1}{5}\varepsilon\gamma$
be as in Proposition \ref{prop:HKUB-localized}. We claim that
\begin{equation}\label{eq:multiple-DH-apply-remainder-proof1}
\functionspace{P}^{U}_{s}(z,A)\leq c'_{\varepsilon}c_{F}c_{\varepsilon,3}
	(\inf\nolimits_{U\times U}F_{t})\mu(A)
\end{equation}
for some $c_{\varepsilon,3}\in(0,\infty)$ explicit in
$c_{\Psi},\beta_{1},\beta_{2},\alpha_{F},\gamma,\varepsilon$.
Indeed, \eqref{eq:multiple-DH-apply-remainder-proof1} trivially holds for $s=0$
since $\functionspace{P}^{U}_{0}(z,A)=\mathbb{P}_{z}[X_{0}\in A,\,0<\tau_{U}]=0$
by $\mathbb{P}_{z}[X_{0}=z]=1$ and $z\not\in B\supset A$, and thus we may assume
$s\in(0,t]$. Then $s\in(0,\Psi(R))$, $z\in U\setminus N$ by $\overline{B}\subset U$,
and hence an application of Proposition \ref{prop:HKUB-localized} yields
\eqref{eq:HKUB-localized-off-diag-local} with $(s,z)$ in place of $(t,x)$.
Let $y\in U^{\circ}_{\varepsilon R}$ and $x_{0},y_{0}\in U$.
By \textup{(DB)$_{\Psi}$}, $0<s\leq t<\Psi(R)$ and $\diam U\leq R$ we have
$F_{s}(z,y)\leq c_{F}(\Psi(R)/s)^{\alpha_{F}}F_{t}(x_{0},y_{0})$,
and furthermore we easily see from
$z\in\partial B=\partial U^{\circ}_{(\varepsilon/2)R}$ that
$d(z,y)>\frac{1}{2}\varepsilon R$, so that
$\exp\bigl(-\Phi(\gamma_{\varepsilon}d(z,y),s)\bigr)
	\leq\exp\bigl(-\Phi(\frac{1}{2}\varepsilon\gamma_{\varepsilon}R,s)\bigr)$
by the monotonicity of $\Phi(\cdot,s)$.
These facts, \eqref{eq:upper-bound-function-Phi-ULE}
and \eqref{eq:upper-bound-function-Psi} together imply that
\begin{align*}
&F_{s}(z,y)\exp\bigl(-\Phi(\gamma_{\varepsilon}d(z,y),s)\bigr)\\
&\leq c_{F}\Bigl(\frac{\Psi(R)}{s}\Bigr)^{\alpha_{F}}F_{t}(x_{0},y_{0})
	\exp\bigl(-\Phi({\textstyle\frac{1}{2}}\varepsilon\gamma_{\varepsilon}R,s)\bigr)\\
&\leq c_{F}\Bigl(\frac{\Psi(R)}{s}\Bigr)^{\alpha_{F}}F_{t}(x_{0},y_{0})
	\exp\biggl(-(c_{\Psi}2^{\beta_{1}})^{-\frac{1}{\beta_{1}-1}}
		\min_{k\in\{1,2\}}\Bigl(\frac{\Psi(\frac{1}{2}\varepsilon\gamma_{\varepsilon}R)}{s}\Bigr)^{\frac{1}{\beta_{k}-1}}\biggr)\\
&\leq c_{F}c_{\varepsilon,3}F_{t}(x_{0},y_{0}),
\end{align*}
and taking the infimum in $(x_{0},y_{0})\in U\times U$ shows that for any $y\in U^{\circ}_{\varepsilon R}$,
\begin{equation}\label{eq:multiple-DH-apply-remainder-proof2}
F_{s}(z,y)\exp\bigl(-\Phi(\gamma_{\varepsilon}d(z,y),s)\bigr)
	\leq c_{F}c_{\varepsilon,3}(\inf\nolimits_{U\times U}F_{t}).
\end{equation}
Then \eqref{eq:multiple-DH-apply-remainder-proof1} is immediate from
\eqref{eq:HKUB-localized-off-diag-local} with $(s,z)$ in place of $(t,x)$,
$A\subset U^{\circ}_{\varepsilon R}$ and \eqref{eq:multiple-DH-apply-remainder-proof2}.

Let $n\in\mathbb{N}$. By \eqref{eq:properly-exceptional},
\eqref{eq:X-continuous}, \eqref{eq:multiple-DH-sigma-n-boundary} and
\eqref{eq:multiple-DH-apply-remainder-proof1},
\begin{align}
\mathbb{E}_{x}\bigl[\ind{\{\sigma_{n}\leq t\}}\functionspace{P}^{U}_{t-\sigma_{n}}(X_{\sigma_{n}},A)\bigr]
	&=\mathbb{E}_{x}\bigl[\ind{\{\sigma_{n}\leq t,\,X_{\sigma_{n}}\in(\partial B)\setminus N\}}\functionspace{P}^{U}_{t-\sigma_{n}}(X_{\sigma_{n}},A)\bigr]\notag\\
&\leq c'_{\varepsilon}c_{F}c_{\varepsilon,3}(\inf\nolimits_{U\times U}F_{t})\mu(A)
	\mathbb{P}_{x}[\sigma_{n}\leq t],
\label{eq:multiple-DH-apply-remainder-proof3}
\end{align}
and we need to estimate $\mathbb{P}_{x}[\sigma_{n}\leq t]$. Recall that
$\tau_{n}\leq\sigma_{n}\leq\tau_{n+1}=\sigma_{n}+\tau_{U}\circ\theta_{\sigma_{n}}$
as mentioned in the proof of Theorem \ref{thm:multiple-DH}. Assume $n\geq 2$.
For each $\omega\in\{\sigma_{n}\leq t\}$, since
$0\leq\sigma_{k}(\omega)\leq\sigma_{n}(\omega)\leq t$
for any $k\in\{1,\dots,n\}$, we have
$t\geq\sigma_{n}(\omega)\geq\sigma_{n}(\omega)-\sigma_{1}(\omega)
	=\sum_{k=1}^{n-1}(\sigma_{k+1}(\omega)-\sigma_{k}(\omega))$
and therefore
$\tau_{U}\circ\theta_{\sigma_{k}}(\omega)=\tau_{k+1}(\omega)-\sigma_{k}(\omega)
	\leq\sigma_{k+1}(\omega)-\sigma_{k}(\omega)\leq\frac{t}{n-1}$
for some $k\in\{1,\dots,n-1\}$. Thus
$\{\sigma_{n}\leq t\}\subset\bigcup_{k=1}^{n-1}\{\sigma_{k}\leq t,\,\tau_{U}\circ\theta_{\sigma_{k}}\leq\frac{t}{n-1}\}$
and hence
\begin{equation}\label{eq:multiple-DH-apply-remainder-proof4}
\mathbb{P}_{x}[\sigma_{n}\leq t]
	\leq\mathbb{P}_{x}\Biggl[\bigcup_{k=1}^{n-1}\{\sigma_{k}\leq t,\,\tau_{U}\circ\theta_{\sigma_{k}}\leq{\textstyle\frac{t}{n-1}}\}\Biggr]
	\leq\sum_{k=1}^{n-1}\mathbb{P}_{x}[\sigma_{k}\leq t,\,\tau_{U}\circ\theta_{\sigma_{k}}\leq{\textstyle\frac{t}{n-1}}].
\end{equation}
Furthermore by using first the strong Markov property \cite[Theorem A.1.21]{CF}
of $X$ at time $\sigma_{k}$ and then \eqref{eq:properly-exceptional},
\eqref{eq:X-continuous} and \eqref{eq:multiple-DH-sigma-n-boundary}
we see that for any $k\in\{1,\dots,n-1\}$,
\begin{align}
\mathbb{P}_{x}[\sigma_{k}\leq t,\,\tau_{U}\circ\theta_{\sigma_{k}}\leq{\textstyle\frac{t}{n-1}}]
	&=\mathbb{E}_{x}[\ind{\{\sigma_{k}\leq t\}}(\ind{\{\tau_{U}\leq t/(n-1)\}}\circ\theta_{\sigma_{k}})]\notag\\
&=\mathbb{E}_{x}\bigl[\ind{\{\sigma_{k}\leq t\}}\mathbb{P}_{X_{\sigma_{k}}}[\tau_{U}\leq{\textstyle\frac{t}{n-1}}]\bigr]\notag\\
&=\mathbb{E}_{x}\bigl[\ind{\{\sigma_{k}\leq t,\,X_{\sigma_{k}}\in(\partial B)\setminus N\}}\mathbb{P}_{X_{\sigma_{k}}}[\tau_{U}\leq{\textstyle\frac{t}{n-1}}]\bigr]\notag\\
&\leq c\exp\bigl(-\Phi({\textstyle\frac{1}{2}}\varepsilon\gamma R,{\textstyle\frac{t}{n-1}})\bigr);
\label{eq:multiple-DH-apply-remainder-proof5}
\end{align}
here the last inequality follows from the fact that for any
$z\in(\partial B)\setminus N$, $\tau_{B(z,\varepsilon R/2)}\leq\tau_{U}$
by $B(z,\frac{1}{2}\varepsilon R)\subset U$ and hence
$\mathbb{P}_{z}[\tau_{U}\leq{\textstyle\frac{t}{n-1}}]
	\leq\mathbb{P}_{z}[\tau_{B(z,\varepsilon R/2)}\leq{\textstyle\frac{t}{n-1}}]
	\leq c\exp\bigl(-\Phi({\textstyle\frac{1}{2}}\varepsilon\gamma R,{\textstyle\frac{t}{n-1}})\bigr)$
by \textup{(P)$_{\Psi}^{U,R}$} for $(z,\frac{1}{2}\varepsilon R)$.
Also, by the monotonicity of $\Phi(\cdot,\frac{t}{n-1})$ and
$\Phi(\gamma_{\varepsilon}R,\cdot)$, \eqref{eq:upper-bound-function-Phi-constant},
\eqref{eq:upper-bound-function-Phi-ULE}, \eqref{eq:upper-bound-function-Psi},
$1<\beta_{1}\leq\beta_{2}$ and $t<\Psi(R)$,
for some $c_{\varepsilon,4}\in(0,\infty)$ explicit in
$c_{\Psi},\beta_{1},\beta_{2},\gamma_{\varepsilon}$,
\begin{align}
\Phi({\textstyle\frac{1}{2}}\varepsilon\gamma R,{\textstyle\frac{t}{n-1}})
	&\geq\Phi(2\gamma_{\varepsilon}R,{\textstyle\frac{t}{n-1}})
	\geq 2\Phi(\gamma_{\varepsilon}R,{\textstyle\frac{t}{n-1}})\notag\\
&\geq\Phi(\gamma_{\varepsilon}R,t)+(c_{\Psi}2^{\beta_{1}})^{-\frac{1}{\beta_{1}-1}}
	\min_{k\in\{1,2\}}\Bigl(\frac{\Psi(\gamma_{\varepsilon}R)}{t/(n-1)}\Bigr)^{\frac{1}{\beta_{k}-1}}\notag\\
&\geq\Phi(\gamma_{\varepsilon}R,t)+c_{\varepsilon,4}(n-1)^{\frac{1}{\beta_{2}-1}}.
\label{eq:multiple-DH-apply-remainder-proof6}
\end{align}
From \eqref{eq:multiple-DH-apply-remainder-proof4},
\eqref{eq:multiple-DH-apply-remainder-proof5} and
\eqref{eq:multiple-DH-apply-remainder-proof6} we conclude that
for any $n\in\mathbb{N}\setminus\{1\}$,
\begin{align}
\mathbb{P}_{x}[\sigma_{n}\leq t]
	&\leq c(n-1)e^{-c_{\varepsilon,4}(n-1)^{\frac{1}{\beta_{2}-1}}}\exp(-\Phi(\gamma_{\varepsilon}R,t))\notag\\
&\leq cc_{\varepsilon,5}(n-1)^{-2}\exp(-\Phi(\gamma_{\varepsilon}R,t)),
\label{eq:multiple-DH-apply-remainder-proof7}
\end{align}
where $c_{\varepsilon,5}:=\bigl(3(\beta_{2}-1)/(ec_{\varepsilon,4})\bigr)^{3(\beta_{2}-1)}$.
For $\mathbb{P}_{x}[\sigma_{1}\leq t]$, set $B':=U^{\circ}_{(\varepsilon/4)R}$
and $\sigma:=\dot{\sigma}_{B',\tau_{U}}=\dot{\sigma}_{B',\tau_{1}}$
(recall Definition \ref{dfn:entrance-exit-after-sigma}), so that we have
\eqref{eq:multiple-DH-sigma-n-boundary} with $B'$ and $\sigma$ in place of $B$ and
$\sigma_{n}$, respectively, by substituting $\frac{1}{2}\varepsilon$ for $\varepsilon$.
Noting that $\sigma_{1}=\sigma+\dot{\sigma}_{B}\circ\theta_{\sigma}$ by
$B\subset B'$ and thus that
$\{\sigma_{1}\leq t\}\subset\{\sigma\leq t,\,\dot{\sigma}_{B}\circ\theta_{\sigma}\leq t\}$,
from the strong Markov property \cite[Theorem A.1.21]{CF} of
$X$ at time $\sigma$, \eqref{eq:properly-exceptional}, \eqref{eq:X-continuous}
and \eqref{eq:multiple-DH-sigma-n-boundary} we obtain
\begin{align}
\mathbb{P}_{x}[\sigma_{1}\leq t]
	\leq\mathbb{P}_{x}[\sigma\leq t,\,\dot{\sigma}_{B}\circ\theta_{\sigma}\leq t]
	&=\mathbb{E}_{x}[\ind{\{\sigma\leq t\}}(\ind{\{\dot{\sigma}_{B}\leq t\}}\circ\theta_{\sigma})]\notag\\
&=\mathbb{E}_{x}\bigl[\ind{\{\sigma\leq t\}}\mathbb{P}_{X_{\sigma}}[\dot{\sigma}_{B}\leq t]\bigr]\notag\\
&=\mathbb{E}_{x}\bigl[\ind{\{\sigma\leq t,\,X_{\sigma}\in(\partial B')\setminus N\}}\mathbb{P}_{X_{\sigma}}[\dot{\sigma}_{B}\leq t]\bigr]\notag\\
&\leq c\exp(-\Phi(\gamma_{\varepsilon}R,t));
\label{eq:multiple-DH-apply-remainder-proof8}
\end{align}
here, similarly to \eqref{eq:multiple-DH-apply-remainder-proof5},
the last inequality holds since for any
$z\in(\partial B')\setminus N$, $\tau_{B(z,\varepsilon R/4)}\leq\dot{\sigma}_{B}$
by $B(z,\frac{1}{4}\varepsilon R)\subset U\setminus B$ and hence
$\mathbb{P}_{z}[\dot{\sigma}_{B}\leq t]
	\leq\mathbb{P}_{z}[\tau_{B(z,\varepsilon R/4)}\leq t]
	\leq c\exp(-\Phi(\gamma_{\varepsilon}R,t))$
by \textup{(P)$_{\Psi}^{U,R}$} for $(z,\frac{1}{4}\varepsilon R)$ and
the monotonicity of $\Phi(\cdot,t)$.

Now \eqref{eq:HKUB-localized-difference} with
$c''_{\varepsilon}:=cc'_{\varepsilon}c_{F}c_{\varepsilon,3}(2c_{\varepsilon,5}+1)$
is immediate from \eqref{eq:multiple-DH-apply},
\eqref{eq:multiple-DH-apply-remainder-proof3},
\eqref{eq:multiple-DH-apply-remainder-proof7} and
\eqref{eq:multiple-DH-apply-remainder-proof8},
completing the proof of Proposition \ref{prop:HKUB-localized-difference}.
\end{proof}

\begin{proof}[Proof of Theorem \textup{\ref{thm:HKUB-localized}}]
Let $c'_{\varepsilon},c''_{\varepsilon}\in(0,\infty)$ be as in
Propositions \ref{prop:HKUB-localized} and \ref{prop:HKUB-localized-difference},
respectively, and let $\gamma_{\varepsilon}:=\frac{1}{5}\varepsilon\gamma$.
We show that Theorem \ref{thm:HKUB-localized} can be concluded from
Proposition \ref{prop:HK-existence} applied to $I=(0,\infty)$,
$V=M\setminus N$, $W=U^{\circ}_{\varepsilon R}$, $M$ in place of $U$, and
$H=H_{t}(x,y):(0,\infty)\times(M\setminus N)\times U^{\circ}_{\varepsilon R}\to[0,\infty)$
given by
\begin{equation}\label{eq:HKUB-localized-dfn-H}
H_{t}(x,y):=
	\begin{cases}
	(c'_{\varepsilon}+c''_{\varepsilon})F_{t}(x,y)\exp\bigl(-\Phi(\gamma_{\varepsilon}d(x,y),t)\bigr)
		&\textrm{if }t<\Psi(R)\textrm{ and }x\in U,\\
	c''_{\varepsilon}c_{F}2^{\alpha_{F}}(\inf_{U\times U}F_{\Psi(R)/2^{n}})\exp(-\Phi(\gamma_{\varepsilon}R,t))
		&\textrm{if }\frac{\Psi(R)}{2^{n+1}}\leq t<\frac{\Psi(R)}{2^{n}}\textrm{ and }x\not\in U,\\
	c_{\varepsilon}(\inf_{U\times U}F_{\Psi(R)}) &\textrm{if }t\geq\Psi(R),
	\end{cases}
\end{equation}
where $n\in\mathbb{N}\cup\{0\}$ in the second line and
$c_{\varepsilon}:=((c'_{\varepsilon}+c''_{\varepsilon})c_{F}2^{\alpha_{F}})
	\vee(c''_{\varepsilon}c_{F}^{2}2^{2\alpha_{F}})$.
Obviously $H=H_{t}(x,y)$ is Borel measurable, and by using \textup{(DB)$_{\Psi}$}
it is easily seen to be less than or equal to the right-hand side of
\eqref{eq:HKUB-localized-off-diag-U}, so that it remains to verify that
\begin{equation}\label{eq:HK-existence-global}
\functionspace{P}_{t}(x,A)\leq\int_{A}H_{t}(x,y)\,d\mu(y)
\end{equation}
for any $(t,x)\in(0,\infty)\times(M\setminus N)$ and any
$A\in\sigmafield{B}(U^{\circ}_{\varepsilon R})$. Note that by
\textup{(DB)$_{\Psi}$} and $\diam U\leq R$ we also have
\begin{equation}\label{eq:HKUB-localized-proof}
H_{\Psi(R)/2}(z,y)\leq c_{\varepsilon}(\inf\nolimits_{U\times U}F_{\Psi(R)})\qquad
	\textrm{for any }(z,y)\in(M\setminus N)\times U^{\circ}_{\varepsilon R}.
\end{equation}

Let $(t,x)\in(0,\infty)\times(M\setminus N)$ and $A\in\sigmafield{B}(U^{\circ}_{\varepsilon R})$.
If $t<\Psi(R)$ and $x\in U$, then \eqref{eq:HK-existence-global} easily follows from
Propositions \ref{prop:HKUB-localized} and \ref{prop:HKUB-localized-difference}
in view of the fact that $\Phi(\gamma_{\varepsilon}R,t)\geq\Phi(\gamma_{\varepsilon}d(x,y),t)$
for any $y\in U^{\circ}_{\varepsilon R}$ by $\diam U\leq R$ and the monotonicity of
$\Phi(\cdot,t)$. If $t<\Psi(R)$ and $x\not\in U$, then we see from \textup{(DB)$_{\Psi}$} that
$c''_{\varepsilon}(\inf\nolimits_{U\times U}F_{t})\exp(-\Phi(\gamma_{\varepsilon}R,t))\leq H_{t}(x,y)$
for any $y\in U^{\circ}_{\varepsilon R}$, which together with
Proposition \ref{prop:HKUB-localized-difference} and
$\functionspace{P}^{U}_{t}(x,A)=0$ immediately implies \eqref{eq:HK-existence-global}.

Now assume $t\geq\Psi(R)$. Since
$\functionspace{P}_{t-\Psi(R)/2}(x,N)=\mathbb{P}_{x}[X_{t-\Psi(R)/2}\in N]=0$
by \eqref{eq:properly-exceptional} and
\begin{equation*}
\functionspace{P}_{\Psi(R)/2}(z,A)\leq\int_{A}H_{\Psi(R)/2}(z,y)\,d\mu(y)\leq
	\int_{A}H_{t}(x,y)\,d\mu(y)
\end{equation*}
for any $z\in M\setminus N$ by the previous paragraph,
\eqref{eq:HKUB-localized-proof} and \eqref{eq:HKUB-localized-dfn-H},
from \eqref{eq:Pt-semigroup} we get
\begin{align*}
\functionspace{P}_{t}(x,A)
	=\functionspace{P}_{t-\Psi(R)/2}(\functionspace{P}_{\Psi(R)/2}\ind{A})(x)
	&=\int_{M\setminus N}\functionspace{P}_{\Psi(R)/2}(z,A)\functionspace{P}_{t-\Psi(R)/2}(x,dz)\\
&\leq\int_{M\setminus N}\biggl(\int_{A}H_{t}(x,y)\,d\mu(y)\biggr)\functionspace{P}_{t-\Psi(R)/2}(x,dz)\\
&\leq\int_{A}H_{t}(x,y)\,d\mu(y).
\end{align*}

Thus \eqref{eq:HK-existence-global} has been proved and hence
Theorem \ref{thm:HKUB-localized} follows from Proposition \ref{prop:HK-existence}.
\end{proof}

\begin{proof}[Proof of Theorem \textup{\ref{thm:HKUB-localized-global}}]
Define $H=H_{t}(x,y):(0,\infty)\times(M\setminus N)\times M\to[0,\infty)$
by the right-hand side of \eqref{eq:HKUB-localized-off-diag}, so that it is
clearly Borel measurable. Thanks to Proposition \ref{prop:HK-existence},
it suffices to show \eqref{eq:HK-existence-global} for any
$(t,x)\in(0,\infty)\times(M\setminus N)$ and any $A\in\sigmafield{B}(M)$
for some $c'\in(0,\infty)$ explicit in
$c_{\Psi},\beta_{1},\beta_{2},c_{F},\alpha_{F},c,\gamma$.
For applications of Theorem \ref{thm:HKUB-localized} and
Proposition \ref{prop:HKUB-localized-difference},
we remark that for any $(y_{0},R')\in M\times(0,\infty)$,
\begin{equation}\label{eq:U-ball}
\textrm{if we set}\quad U:=B(y_{0},{\textstyle\frac{R'}{2}})
	\quad\textrm{then}\quad\diam U\leq R'\quad\textrm{and}\quad
	B(y_{0},{\textstyle\frac{R'}{4}})\subset U^{\circ}_{(1/4)R'}.
\end{equation}

Let $(t,x)\in(0,\infty)\times(M\setminus N)$. If $R=\infty$, then for any
$A\in\sigmafield{B}(M)$ and any $n\in\mathbb{N}$ with $n>\Psi^{-1}(t)$,
in view of \eqref{eq:U-ball} we can apply Theorem \ref{thm:HKUB-localized} with
$n,B(x,\frac{n}{2}),\frac{1}{4}$ in place of $R,U,\varepsilon$ respectively
and $A\cap B(x,\frac{n}{4})$ in place of $A$ in
\eqref{eq:HKUB-localized-existence} and obtain
\begin{equation*}
\functionspace{P}_{t}(x,A\cap B(x,{\textstyle\frac{n}{4}}))
	\leq\int_{A\cap B(x,n/4)}H_{t}(x,y)\,d\mu(y)
\end{equation*}
with $c'=c_{1/4}$, which yields \eqref{eq:HK-existence-global}
by using monotone convergence to let $n\to\infty$.

Thus we may assume $R<\infty$.
Let $y_{0}\in M$ and $A\in\sigmafield{B}(B(y_{0},\frac{R}{4}))$.
We claim that \eqref{eq:HK-existence-global} holds for such $A$.
Indeed, setting $U:=B(y_{0},\frac{R}{2})$,
we have \eqref{eq:HK-existence-global} with $H_{t}(x,y)$ replaced by
\begin{equation}\label{eq:HKUB-localized-global-Htilde}
\widetilde{H}_{t}(x,y):=
	\begin{cases}
	c_{1/4}F_{t}(x,y)\exp\bigl(-\Phi(\frac{1}{20}\gamma d(x,y),t)\bigr) &\textrm{if }t<\Psi(R)\textrm{ and }x\in U,\\
	c''_{1/4}(\inf\nolimits_{U\times U}F_{t})\exp(-\Phi(\frac{1}{20}\gamma R,t))
		&\textrm{if }t<\Psi(R)\textrm{ and }x\not\in U,\\
	c_{1/4}(\inf_{U\times U}F_{\Psi(R)}) &\textrm{if }t\geq\Psi(R)
	\end{cases}
\end{equation}
since Theorem \ref{thm:HKUB-localized} and
Proposition \ref{prop:HKUB-localized-difference}
with $\varepsilon=\frac{1}{4}$ are applicable by \eqref{eq:U-ball}
and $\functionspace{P}^{U}_{t}(x,A)=0$ if $x\not\in U$.
Moreover, if $t\leq\Psi(R)$ then for any $y\in M$ and any $z,w\in U$,
\begin{align}\label{eq:HKUB-localized-global-Rfinite1}
\inf\nolimits_{U\times U}F_{t}\leq F_{t}(z,w)
	&\leq c_{F}\Bigl(\frac{t\vee\Psi(d(x,z))\vee\Psi(d(y,w))}{t}\Bigr)^{\alpha_{F}}F_{t}(x,y)\notag\\
&\leq c_{F}\Bigl(\frac{\Psi(\delta^{-1}R)}{t}\Bigr)^{\alpha_{F}}F_{t}(x,y)\notag\\
&\leq c_{F}c_{\Psi}^{\alpha_{F}}\delta^{-\beta_{2}\alpha_{F}}
	\Bigl(\frac{\Psi(R)}{t}\Bigr)^{\alpha_{F}}F_{t}(x,y)
\end{align}
by \textup{(DB)$_{\Psi}$}, $\diam M\leq\delta^{-1}R$ and
\eqref{eq:upper-bound-function-Psi} (even if $x\in N$), hence
\begin{equation}\label{eq:HKUB-localized-global-Rfinite2}
c_{1/4}(\inf\nolimits_{U\times U}F_{\Psi(R)})
	\leq c_{1/4}c_{F}c_{\Psi}^{\alpha_{F}}\delta^{-\beta_{2}\alpha_{F}}
		(\inf\nolimits_{M\times M}F_{\Psi(R)}),
\end{equation}
and we also easily see from \eqref{eq:HKUB-localized-global-Rfinite1},
\eqref{eq:upper-bound-function-Phi-constant},
\eqref{eq:upper-bound-function-Phi-ULE}, \eqref{eq:upper-bound-function-Psi}
and $R\geq\delta\diam M$ that
\begin{align}
&c''_{1/4}(\inf\nolimits_{U\times U}F_{t})\exp(-\Phi({\textstyle\frac{1}{20}}\gamma R,t))\notag\\
&\leq c''_{1/4}c_{F}c_{\Psi}^{\alpha_{F}}\delta^{-\beta_{2}\alpha_{F}}
	\Bigl(\frac{\Psi(R)}{t}\Bigr)^{\alpha_{F}}F_{t}(x,y)\exp(-2\Phi({\textstyle\frac{1}{40}}\gamma R,t))\notag\\
&\leq c''\delta^{-\beta_{2}\alpha_{F}}F_{t}(x,y)\exp\bigl(-\Phi(\gamma'_{\delta}d(x,y),t)\bigr)
\label{eq:HKUB-localized-global-Rfinite3}
\end{align}
for any $y\in M$ for some $c''\in(0,\infty)$ explicit in
$c_{\Psi},\beta_{1},\beta_{2},c_{F},\alpha_{F},c,\gamma$.
Therefore putting $c':=c''\vee(c_{1/4}c_{F}c_{\Psi}^{\alpha_{F}})$,
we have $\widetilde{H}_{t}(x,y)\leq H_{t}(x,y)$ for any $y\in M$
by \eqref{eq:HKUB-localized-global-Htilde},
\eqref{eq:HKUB-localized-global-Rfinite2} and
\eqref{eq:HKUB-localized-global-Rfinite3}, and thus the inequality
\eqref{eq:HK-existence-global} follows from that with
$\widetilde{H}_{t}(x,y)$ in place of $H_{t}(x,y)$.

Now let $\{y_{k}\}_{k\in\mathbb{N}}$ be a countable dense
subset of $M$ and set $B_{1}:=B(y_{1},\frac{R}{4})$ and
$B_{k}:=B(y_{k},\frac{R}{4})\setminus\bigcup_{j=1}^{k-1}B(y_{j},\frac{R}{4})$
for $k\in\mathbb{N}\setminus\{1\}$, so that
$\{B_{k}\}_{k\in\mathbb{N}}\subset\sigmafield{B}(M)$ and
$M=\bigcup_{k\in\mathbb{N}}B_{k}$, where the union is disjoint.
Then for any $A\in\sigmafield{B}(M)$, for each $k\in\mathbb{N}$ we have
\eqref{eq:HK-existence-global} with $A\cap B_{k}$ in place of $A$ by the
previous paragraph and $A\cap B_{k}\in\sigmafield{B}(B(y_{k},\frac{R}{4}))$,
from which \eqref{eq:HK-existence-global} follows
in exactly the same way as \eqref{eq:HKUB-localized-off-diag-local-end},
completing the proof of Theorem \ref{thm:HKUB-localized-global}.
\end{proof}

\section{Exit probability estimates for diffusions}\label{sec:exit-probability}

As already mentioned at the beginning of Section \ref{sec:HKUB},
the purpose of this section is to provide reasonable sufficient conditions
for the exit probability estimate \textup{(P)$_{\Psi}^{U,R}$} of
Theorem \ref{thm:HKUB-localized}. Recall that since Section \ref{sec:HKUB}
we have fixed $\Psi,c_{\Psi},\beta_{1},\beta_{2}$ and
$\Phi=\Phi_{\Psi}$ as in Lemma \ref{lem:upper-bound-function-Psi}.

\emph{In this section, we fix an arbitrary properly exceptional set
$N\in\sigmafield{B}(M)$ for $X$ satisfying
\begin{equation}\label{eq:X-continuous-no-killing-inside}
\textrm{both}\quad\eqref{eq:X-continuous}\quad\textrm{and}\quad
	\mathbb{P}_{x}[\zeta<\infty,\,X_{\zeta-}\in M]=0
\end{equation}
for any $x\in M\setminus N$,}
where $X_{\zeta-}(\omega):=X_{\zeta(\omega)-}(\omega)$
($X_{0-}:=X_{0}$, $X_{\infty-}:=\cemetery=X_{\infty}$), so that
$X_{\zeta-}:\Omega\to M_{\cemetery}$ is
$\sigmafield{F}_{\infty}/\sigmafield{B}(M_{\cemetery})$-measurable by the
left-continuity of $[0,\infty)\ni t\mapsto X_{t-}(\omega)\in M_{\cemetery}$.
By \cite[Theorem 4.5.3]{FOT}, such $N$ exists if and only if
$(\functionspace{E},\functionspace{F})$ is \emph{strongly local},
i.e., $\functionspace{E}(u,v)=0$ for any $u,v\in\functionspace{F}$
with $\supp_{\mu}[u],\supp_{\mu}[v]$ compact and $u=c$ $\mu$-a.e.\ on a
neighborhood of $\supp_{\mu}[v]$ for some $c\in\mathbb{R}$.

Below we will also consider the situation where the set $N$
fixed above satisfies
\begin{equation}\label{eq:X-continuous-conservative}
\textrm{both}\quad\eqref{eq:X-continuous}\quad\textrm{and}\quad
	\mathbb{P}_{x}[\zeta<\infty]=0
\end{equation}
for any $x\in M\setminus N$,
more strongly than \eqref{eq:X-continuous-no-killing-inside}.
By \cite[Theorem 4.5.1 and Exercise 4.5.1]{FOT},
such a properly exceptional set $N\in\sigmafield{B}(M)$ for $X$ exists if and
only if $(\functionspace{E},\functionspace{F})$ is local and \emph{conservative},
i.e., $T_{t}\ind{}=\ind{}$ $\mu$-a.e.\ for any (or equivalently, for some)
$t\in(0,\infty)$, where $\ind{}:=\ind{M}$; recall
(see, e.g., \cite[(1.1.9) and (1.1.11)]{CF}) that
for a Markovian bounded linear operator $T:L^{2}(M,\mu)\to L^{2}(M,\mu)$,
$T|_{L^{2}(M,\mu)\cap L^{\infty}(M,\mu)}$ can be uniquely extended to
a linear operator $T:L^{\infty}(M,\mu)\to L^{\infty}(M,\mu)$ such that
$\lim_{n\to\infty}Tu_{n}=Tu$ $\mu$-a.e.\ for any $u\in L^{\infty}(M,\mu)$
and any $\{u_{n}\}_{n\in\mathbb{N}}\subset L^{\infty}(M,\mu)$ with
$u_{n}\leq u_{n+1}$ $\mu$-a.e.\ for any $n\in\mathbb{N}$ and
$\lim_{n\to\infty}u_{n}=u$ $\mu$-a.e.

\begin{remark}\label{rmk:exit-probability}
In fact, \emph{Theorems \textup{\ref{thm:exit-probability}} and
\textup{\ref{thm:exit-time-probability}} below apply, without any changes in
the proofs, to any locally compact separable metric space $(M,d)$,
any Hunt process $X$ on $(M,\sigmafield{B}(M))$ and any
$N\in\sigmafield{B}(M)$ satisfying \eqref{eq:properly-exceptional} and
\eqref{eq:X-continuous-no-killing-inside} for any $x\in M\setminus N$}.
\end{remark}

The main result of this section is the following theorem, which is a localized version
of an unpublished result \cite[Theorem 9.1]{Gri:HKfractal} by the first named author.
We refer the reader to \cite[Subsection 5.4]{GriHu:Upper} for an alternative
analytic approach. We set $e^{-\infty}:=0$.

\begin{theorem}\label{thm:exit-probability}
Let $U$ be a non-empty open subset of $M$ and let $R\in(0,\infty]$.
Then among the following seven conditions, the latter six \textup{(2)--(7)}
are equivalent and imply \textup{(1)}:
\begin{itemize}[label=\textup{(1)},align=left,leftmargin=*]
\item[\textup{(1)}]
	There exist $\varepsilon\in(0,\frac{1}{2})$ and $\delta\in(0,\infty)$ such that
	for any $(x,r)\in(U\setminus N)\times(0,R)$ with $B(x,r)\subset U$
	and $\overline{B(x,r)}$ compact and for any $t\in(0,\delta\Psi(r)]$,
	\begin{equation}\label{eq:exit-probability1}
		\mathbb{P}_{x}[X_{t}\in M_{\cemetery}\setminus B(x,r)]\leq\varepsilon.
	\end{equation}
\item[\textup{(2)}]
	There exist $\varepsilon\in(0,1)$ and $\delta\in(0,\infty)$ such that
	for any $(x,r)\in(U\setminus N)\times(0,R)$ with $B(x,r)\subset U$
	and $\overline{B(x,r)}$ compact,
	\begin{equation}\label{eq:exit-probability2}
		\mathbb{P}_{x}[\tau_{B(x,r)}\leq\delta\Psi(r)]\leq\varepsilon.
	\end{equation}
\item[\textup{(3)}]
	There exists $\varepsilon\in(0,\infty)$ such that
	for any $(x,r)\in(U\setminus N)\times(0,R)$ with $B(x,r)\subset U$
	and $\overline{B(x,r)}$ compact,
	\begin{equation}\label{eq:exit-probability3}
		\mathbb{E}_{x}[\tau_{B(x,r)}\wedge\Psi(r)]\geq\varepsilon\Psi(r).
	\end{equation}
\item[\textup{(4)}]
	There exist $\varepsilon\in(0,1)$ and $\delta\in(0,\infty)$ such that
	for any $(x,r)\in(U\setminus N)\times(0,R)$ with $B(x,r)\subset U$
	and $\overline{B(x,r)}$ compact,
	\begin{equation}\label{eq:exit-probability4}
		\mathbb{E}_{x}\Bigl[\exp\Bigl(-\frac{\tau_{B(x,r)}}{\delta\Psi(r)}\Bigr)\Bigr]\leq\varepsilon.
	\end{equation}
\item[\textup{(5)}]
	There exist $c,\gamma\in(0,\infty)$ such that for any
	$(x,r)\in(U\setminus N)\times(0,R)$ with $B(x,r)\subset U$
	and $\overline{B(x,r)}$ compact and for any $\lambda\in(0,\infty)$,
	\begin{equation}\label{eq:exit-probability5}
		\mathbb{E}_{x}[e^{-\lambda\tau_{B(x,r)}}]\leq c\exp\Bigl(-\frac{\gamma r}{\Psi^{-1}(\lambda^{-1})}\Bigr).
	\end{equation}
\item[\textup{(6)}]
	There exist $c,\gamma\in(0,\infty)$ such that for any
	$(x,r)\in(U\setminus N)\times(0,R)$ with $B(x,r)\subset U$
	and $\overline{B(x,r)}$ compact and for any $t\in(0,\infty)$,
	\begin{equation}\label{eq:exit-probability6}
		\mathbb{P}_{x}[\tau_{B(x,r)}\leq t]\leq c\exp(-\Phi(\gamma r,t)).
	\end{equation}
\item[\textup{(7)}]
	There exist $c,\gamma\in(0,\infty)$ such that for any
	$(x,r)\in(U\setminus N)\times(0,R)$ with $B(x,r)\subset U$
	and $\overline{B(x,r)}$ compact and for any $t\in(0,\infty)$,
	\begin{equation}\label{eq:exit-probability7}
		\mathbb{P}_{x}[\tau_{B(x,r)}\leq t]\leq
		c\exp\biggl(-\gamma\Bigl(\frac{\Psi(r)}{t}\Bigr)^{\frac{1}{\beta_{2}-1}}\biggr).
	\end{equation}
\end{itemize}

Moreover, if $N$ satisfies \eqref{eq:X-continuous-conservative} for any
$x\in M\setminus N$, then with ``and $\overline{B(x,r)}$ compact" all removed,
still the conditions \textup{(2)--(7)} are equivalent, imply \textup{(1)}
and are implied by the following condition \textup{(1)$'$}:
\begin{itemize}[label=\textup{(1)$'$},align=left,leftmargin=*]
\item[\textup{(1)$'$}]
	There exist $\varepsilon\in(0,\frac{1}{2})$ and $\delta\in(0,\infty)$ such that
	for any $(x,r)\in(\overline{U}\setminus N)\times(0,\frac{R}{2})$
	and any $t\in(0,\delta\Psi(r)]$,
	\begin{equation}\label{eq:exit-probability1'}
		\mathbb{P}_{x}[X_{t}\in M\setminus B(x,r)]\leq\varepsilon.
	\end{equation}
\end{itemize}
\end{theorem}

\begin{proof}
We follow \cite[Proof of Theorem 9.1]{Gri:HKfractal}; for the implications
\textup{(4)}$\Rightarrow$\textup{(5)}$\Rightarrow$\textup{(6)}$\Rightarrow$\textup{(7)}
see also \cite[Proofs of Lemma 3.14, Theorem 3.15 and Corollary 3.20]{GT}.

We treat the two cases simultaneously, one with ``and $\overline{B(x,r)}$ compact"
kept and without \eqref{eq:X-continuous-conservative} and the other with
``and $\overline{B(x,r)}$ compact" removed and \eqref{eq:X-continuous-conservative} assumed.
Let $(x,r)\in(U\setminus N)\times(0,R)$ satisfy $B(x,r)\subset U$
and set $\tau:=\tau_{B(x,r)}$. We assume in the former case that
$\overline{B(x,r)}$ is compact, while not in the latter case.
It easily follows either from \eqref{eq:X-continuous-no-killing-inside} and the
compactness of $\overline{B(x,r)}$ or from \eqref{eq:X-continuous-conservative},
together with $\mathbb{P}_{x}[X_{0}=x]=1$ and \eqref{eq:properly-exceptional}, that
\begin{equation}\label{eq:exit-ball-boundary}
\mathbb{P}_{x}\bigl[\tau_{B(x,\rho)}<\infty,\,X_{\tau_{B(x,\rho)}}\not\in(\partial B(x,\rho))\setminus N\bigr]=0
	\qquad\textrm{for any }\rho\in(0,r].
\end{equation}

\textup{(2)}$\Rightarrow$\textup{(3)}:
Since
$\mathbb{P}_{x}[\tau>\delta\Psi(r)]=1-\mathbb{P}_{x}[\tau\leq\delta\Psi(r)]\geq 1-\varepsilon$
by \eqref{eq:exit-probability2},
\begin{equation*}
\mathbb{E}_{x}[\tau\wedge\Psi(r)]
	\geq\mathbb{E}_{x}[(\tau\wedge\Psi(r))\ind{\{\tau>\delta\Psi(r)\}}]
	\geq(\delta\wedge 1)\Psi(r)\mathbb{P}_{x}[\tau>\delta\Psi(r)]
	\geq(1-\varepsilon)(\delta\wedge 1)\Psi(r).
\end{equation*}

\textup{(3)}$\Rightarrow$\textup{(4)}:
For $\lambda,t\in(0,\infty)$, by considering the case of $\tau\geq t$ and
that of $\tau\leq t$ separately we easily see that
$e^{-\lambda\tau}\leq 1-\lambda e^{-\lambda t}(\tau\wedge t)$,
and therefore for \emph{any} $\delta\in(0,\infty)$, setting $\lambda:=(\delta\Psi(r))^{-1}$
and $t:=\Psi(r)$, taking $\mathbb{E}_{x}[(\cdot)]$ and applying \eqref{eq:exit-probability3}, we obtain
\begin{equation*}
\mathbb{E}_{x}\Bigl[\exp\Bigl(-\frac{\tau}{\delta\Psi(r)}\Bigr)\Bigr]
	\leq 1-\frac{1}{\delta\Psi(r)}e^{-1/\delta}\mathbb{E}_{x}[\tau\wedge\Psi(r)]
	\leq 1-\frac{\varepsilon}{\delta}e^{-1/\delta}.
\end{equation*}

\textup{(4)}$\Rightarrow$\textup{(5)}:
Let $\lambda\in[(\delta\Psi(r))^{-1},\infty)$, set
$n:=\max\{k\in\mathbb{N}\mid\lambda\delta\Psi(r/k)\geq 1\}$ and $\rho:=r/n$.
Also set $B_{k}:=B(x,k\rho)$ and $\tau_{k}:=\tau_{B_{k}}$ for $k\in\{1,\dots,n\}$.
We claim that
\begin{equation}\label{eq:exit-probability-iteration}
\mathbb{E}_{x}[e^{-\lambda\tau_{k+1}}]\leq\varepsilon\mathbb{E}_{x}[e^{-\lambda\tau_{k}}]
	\qquad\textrm{for any }k\in\{1,\dots,n\}\textrm{ with }k<n.
\end{equation}
To see \eqref{eq:exit-probability-iteration}, let $k\in\{1,\dots,n\}$ satisfy $k<n$
and let $y\in(\partial B_{k})\setminus N$. Then obviously
$(y,\rho)\in(U\setminus N)\times(0,R)$,
$B(y,\rho)\subset B_{k+1}\subset B(x,r)\subset U$,
hence $\tau_{B(y,\rho)}\leq\tau_{k+1}$, and
$\overline{B(y,\rho)}$ is compact if $\overline{B(x,r)}$ is.
Thus \textup{(4)} applies to $(y,\rho)$, so that from
\eqref{eq:exit-probability4} we obtain
\begin{equation}\label{eq:exit-probability4-apply}
\mathbb{E}_{y}[e^{-\lambda\tau_{k+1}}]
	\leq\mathbb{E}_{y}[e^{-\lambda\tau_{B(y,\rho)}}]
	\leq\mathbb{E}_{y}\Bigl[\exp\Bigl(-\frac{\tau_{B(y,\rho)}}{\delta\Psi(\rho)}\Bigr)\Bigr]
	\leq\varepsilon,
\end{equation}
noting that $\lambda\geq(\delta\Psi(\rho))^{-1}$ by the choice of $n$.
Now since $\tau_{k+1}=\tau_{k}+\tau_{k+1}\circ\theta_{\tau_{k}}$,
it follows by the strong Markov property \cite[Theorem A.1.21]{CF} of $X$,
\eqref{eq:exit-ball-boundary} and \eqref{eq:exit-probability4-apply} that
\begin{align*}
\mathbb{E}_{x}[e^{-\lambda\tau_{k+1}}]
	=\mathbb{E}_{x}[e^{-\lambda\tau_{k}}(e^{-\lambda\tau_{k+1}}\circ\theta_{\tau_{k}})]
	&=\mathbb{E}_{x}\bigl[e^{-\lambda\tau_{k}}\mathbb{E}_{x}[e^{-\lambda\tau_{k+1}}\circ\theta_{\tau_{k}}\mid\sigmafield{F}_{\tau_{k}}]\bigr]\\
&=\mathbb{E}_{x}\bigl[e^{-\lambda\tau_{k}}\mathbb{E}_{X_{\tau_{k}}}[e^{-\lambda\tau_{k+1}}]\bigr]\\
&=\mathbb{E}_{x}\bigl[\ind{\{X_{\tau_{k}}\in(\partial B_{k})\setminus N\}}e^{-\lambda\tau_{k}}\mathbb{E}_{X_{\tau_{k}}}[e^{-\lambda\tau_{k+1}}]\bigr]\\
&\leq\varepsilon\mathbb{E}_{x}[e^{-\lambda\tau_{k}}].
\end{align*}

\begin{figure}[t]\centering
	\includegraphics[height=180pt]{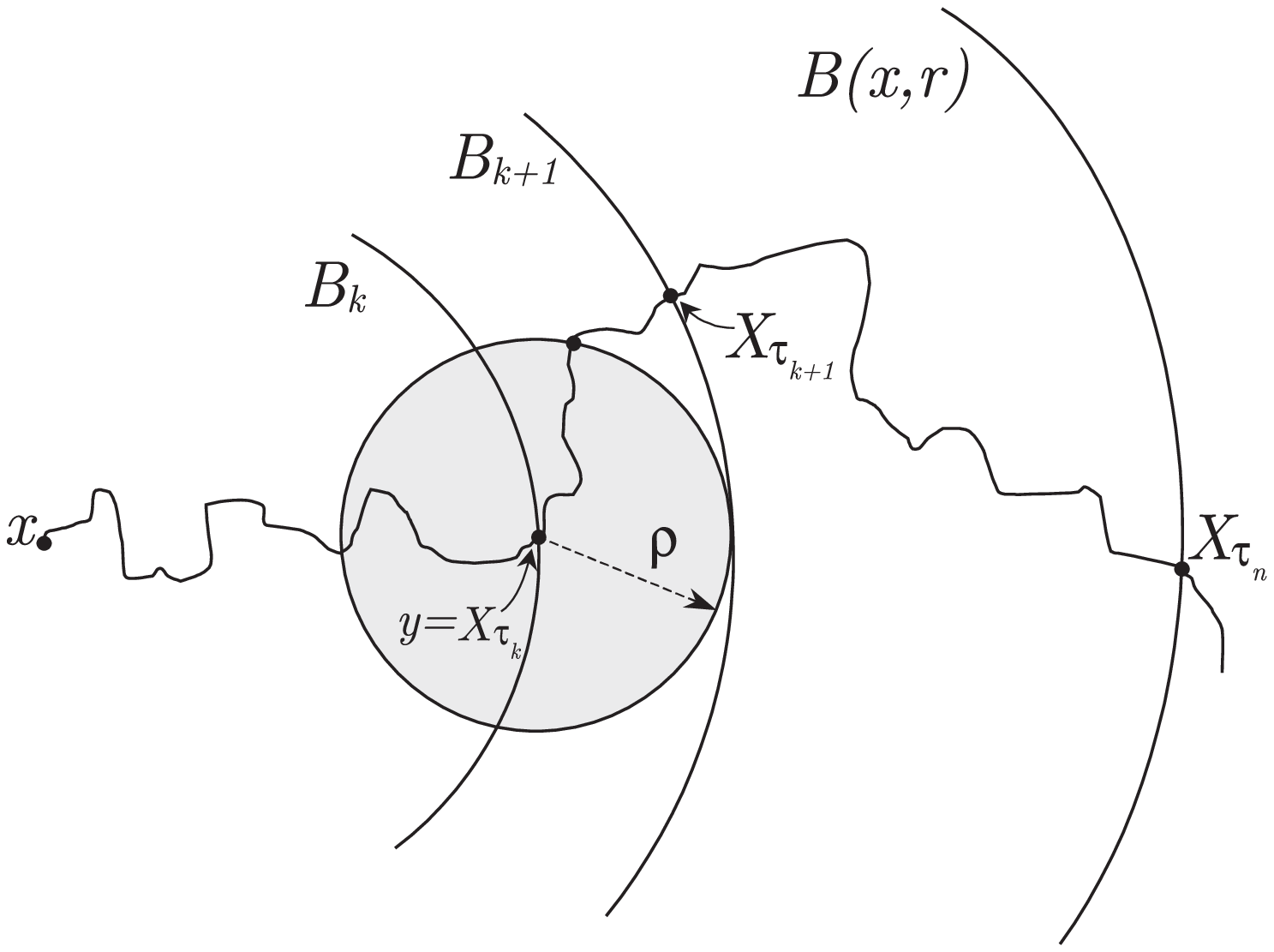}
	\caption{Proof of \textup{(4)}$\Rightarrow$\textup{(5)}: the exit times $\tau_{k}$, $\tau_{k+1}$ and $\tau_{B(y,\rho)}$}
	\label{fig:exit-probability-iteration}
\end{figure}
Thus we have proved \eqref{eq:exit-probability-iteration}, which together with
$\tau=\tau_{n}$ and \eqref{eq:exit-probability4} for $(x,\rho)$ yields
\begin{align*}
\mathbb{E}_{x}[e^{-\lambda\tau}]\leq\varepsilon^{n-1}\mathbb{E}_{x}[e^{-\lambda\tau_{1}}]
	\leq\varepsilon^{n-1}\mathbb{E}_{x}\Bigl[\exp\Bigl(-\frac{\tau_{1}}{\delta\Psi(\rho)}\Bigr)\Bigr]
	\leq\varepsilon^{n}
	<\varepsilon^{-1}\exp\Bigl(-\frac{\gamma r}{\Psi^{-1}(\lambda^{-1})}\Bigr),
\end{align*}
where $\gamma:=\eta\log(\varepsilon^{-1})$ with
$\eta:=(\delta/c_{\Psi})^{1/\beta_{1}}\wedge 1$ and the last inequality follows since
$1>\lambda\delta\Psi(\frac{r}{n+1})\geq\lambda\Psi(\frac{\eta r}{n+1})$
by the choice of $n$ and \eqref{eq:upper-bound-function-Psi} and hence
$n+1>\eta r/\Psi^{-1}(\lambda^{-1})$.

\eqref{eq:exit-probability5} therefore holds for $\lambda\in[(\delta\Psi(r))^{-1},\infty)$.
On the other hand, for $\lambda\in(0,(\delta\Psi(r))^{-1})$, since
$\lambda^{-1}>\delta\Psi(r)\geq\Psi(\eta r)$ by \eqref{eq:upper-bound-function-Psi}
and hence $\Psi^{-1}(\lambda^{-1})>\eta r$, we have
$\mathbb{E}_{x}[e^{-\lambda\tau}]\leq 1
	<e^{\gamma/\eta}\exp\bigl(-\gamma r/\Psi^{-1}(\lambda^{-1})\bigr)$,
completing the proof of \textup{(4)}$\Rightarrow$\textup{(5)}.

\textup{(5)}$\Rightarrow$\textup{(6)}:
For any $t,\lambda\in(0,\infty)$, we see from \eqref{eq:exit-probability5} that
\begin{equation*}
\mathbb{P}_{x}[\tau\leq t]=\mathbb{P}_{x}[e^{-\lambda t}\leq e^{-\lambda\tau}]
	\leq e^{\lambda t}\mathbb{E}_{x}[e^{-\lambda\tau}]
	\leq c\exp\Bigl(\lambda t-\frac{\gamma r}{\Psi^{-1}(\lambda^{-1})}\Bigr),
\end{equation*}
and taking the infimum of the right-hand side in $\lambda\in(0,\infty)$ shows
\eqref{eq:exit-probability6} in view of \eqref{eq:upper-bound-function-Phi}.

\textup{(6)}$\Rightarrow$\textup{(7)}:
Since $\Psi(\gamma r)\geq c_{\Psi}^{-1}(\gamma^{\beta_{2}}\wedge 1)\Psi(r)$ with
$\gamma\in(0,\infty)$ as in \textup{(6)} by \eqref{eq:upper-bound-function-Psi},
if $\Psi(\gamma r)\geq t$ then \eqref{eq:exit-probability7} is immediate from
\eqref{eq:exit-probability6} and the lower inequality in
\eqref{eq:upper-bound-function-Phi-ULE}, whereas if $\Psi(\gamma r)<t$ then we have
$\mathbb{P}_{x}[\tau\leq t]\leq 1\leq c'\exp\bigl(-\gamma'(\Psi(r)/t)^{\frac{1}{\beta_{2}-1}}\bigr)$
for some $c',\gamma'\in(0,\infty)$ explicit in $c_{\Psi},\beta_{1},\beta_{2},\gamma$.

\textup{(7)}$\Rightarrow$\textup{(2)},\textup{(1)}:
For \emph{any} $\varepsilon\in(0,c\wedge\frac{1}{2})$, setting
$\delta:=\bigl(\gamma/\log(c/\varepsilon)\bigr)^{\beta_{2}-1}\in(0,\infty)$,
for any $t\in(0,\delta\Psi(r)]$ we see from
$\{X_{t}\in M_{\cemetery}\setminus B(x,r)\}\subset\{\tau\leq\delta\Psi(r)\}$
and \eqref{eq:exit-probability7} that
\begin{align*}
\mathbb{P}_{x}[X_{t}\in M_{\cemetery}\setminus B(x,r)]
	\leq\mathbb{P}_{x}[\tau\leq\delta\Psi(r)]
	\leq c\exp\bigl(-\gamma\delta^{-\frac{1}{\beta_{2}-1}}\bigr)
	=\varepsilon.
\end{align*}

\emph{\textup{(1)$'$}$\Rightarrow$\textup{(2)} under \eqref{eq:X-continuous-conservative}}:
Note that \eqref{eq:exit-probability1'} is valid also for $t=0$ since $\mathbb{P}_{y}[X_{0}=y]=1$
for $y\in M$. Let $t:=c_{\Psi}^{-1}2^{-\beta_{2}}\delta\Psi(r)$, so that
$t\in(0,\delta\Psi(\frac{r}{2})]$ by \eqref{eq:upper-bound-function-Psi}.
We first show that
\begin{equation}\label{eq:exit-probability1'-proof1}
\mathbb{P}_{x}[\tau\leq t,\,X_{t}\in B(x,{\textstyle\frac{r}{2}})]\leq\varepsilon.
\end{equation}
Indeed, if $y\in(\partial B(x,r))\setminus N$, then
clearly $B(x,\frac{r}{2})\subset M\setminus B(y,\frac{r}{2})$,
$(y,\frac{r}{2})\in(\overline{U}\setminus N)\times(0,\frac{R}{2})$
by $B(x,r)\subset U$ and hence \textup{(1)$'$} applies to $(y,\frac{r}{2})$,
so that \eqref{eq:exit-probability1'} yields
\begin{equation}\label{eq:exit-probability1'-proof2}
\mathbb{P}_{y}[X_{s}\in B(x,{\textstyle\frac{r}{2}})]
	\leq\mathbb{P}_{y}[X_{s}\in M\setminus B(y,{\textstyle\frac{r}{2}})]
	\leq\varepsilon
\end{equation}
for any $y\in(\partial B(x,r))\setminus N$
and any $s\in[0,t]\subset[0,\delta\Psi(\frac{r}{2})]$.
Then since $\{\tau\leq t\}\in\sigmafield{F}_{\tau\wedge t}$ by \cite[Lemma 1.2.16]{KS},
it follows from Proposition \ref{prop:strong-Markov}, \eqref{eq:exit-ball-boundary}
and \eqref{eq:exit-probability1'-proof2} that
\begin{align*}
\mathbb{P}_{x}[\tau\leq t,\,X_{t}\in B(x,{\textstyle\frac{r}{2}})]
	&=\mathbb{E}_{x}[\ind{\{\tau\leq t\}}\ind{B(x,r/2)}(X_{t})]\\
&=\mathbb{E}_{x}\bigl[\ind{\{\tau\leq t\}}\mathbb{E}_{x}[\ind{B(x,r/2)}(X_{t})\mid\sigmafield{F}_{\tau\wedge t}]\bigr]\\
&=\int_{\{\tau\leq t\}}\mathbb{E}_{X_{\tau\wedge t}(\omega)}[\ind{B(x,r/2)}(X_{t-\tau(\omega)\wedge t})]\,d\mathbb{P}_{x}(\omega)\\
&=\int_{\{\tau\leq t,\,X_{\tau}\in(\partial B(x,r))\setminus N\}}
	\mathbb{P}_{X_{\tau}(\omega)}[X_{t-\tau(\omega)}\in B(x,{\textstyle\frac{r}{2}})]\,d\mathbb{P}_{x}(\omega)\\
&\leq\varepsilon\mathbb{P}_{x}[\tau\leq t,\,X_{\tau}\in(\partial B(x,r))\setminus N]\leq\varepsilon.
\end{align*}

Now noting that $\mathbb{P}_{x}[X_{t}=\cemetery]=0$ by \eqref{eq:X-continuous-conservative},
from \eqref{eq:exit-probability1'} for $(x,\frac{r}{2})$ and
\eqref{eq:exit-probability1'-proof1} we obtain
\begin{align*}
\mathbb{P}_{x}[\tau\leq t]
	&=\mathbb{P}_{x}[\tau\leq t,\,X_{t}=\cemetery]
	+\mathbb{P}_{x}[\tau\leq t,\,X_{t}\in M\setminus B(x,{\textstyle\frac{r}{2}})]
	+\mathbb{P}_{x}[\tau\leq t,\,X_{t}\in B(x,{\textstyle\frac{r}{2}})]\\
&\leq\mathbb{P}_{x}[X_{t}=\cemetery]+\mathbb{P}_{x}[X_{t}\in M\setminus B(x,{\textstyle\frac{r}{2}})]+\varepsilon\\
&\leq 2\varepsilon<1,
\end{align*}
which, in view of $t=c_{\Psi}^{-1}2^{-\beta_{2}}\delta\Psi(r)$,
completes the proof of \textup{(1)$'$}$\Rightarrow$\textup{(2)}
under \eqref{eq:X-continuous-conservative}.
\end{proof}

At the last of this paper, as an application of Theorem \ref{thm:exit-probability}
we state and prove a localized version of the well-known fact that the comparability
of the mean exit time $\mathbb{E}_{x}[\tau_{B(x,r)}]$ to $\Psi(r)$ implies the exit
probability estimate \eqref{eq:exit-probability6}. This fact was first observed by
M.\ T.\ Barlow as treated in \cite[Proof of Theorem 3.11]{Bar}, and
the proof below is also based on an idea of his in \cite{Bar}.

\begin{theorem}\label{thm:exit-time-probability}
Let $U$ be a non-empty open subset of $M$, let $R\in(0,\infty]$,
and assume that there exists $c_{\mathrm{E}}\in(0,\infty)$ such that
for any $(x,r)\in(U\setminus N)\times(0,2R)$,
\begin{equation}\label{eq:exit-time-upper}
\mathbb{E}_{x}[\tau_{B(x,r)}]\leq c_{\mathrm{E}}\Psi(r),
\end{equation}
and for any $(x,r)\in(U\setminus N)\times(0,R)$ with
$B(x,r)\subset U$ and $\overline{B(x,r)}$ compact,
\begin{equation}\label{eq:exit-time-lower}
\mathbb{E}_{x}[\tau_{B(x,r)}]\geq c_{\mathrm{E}}^{-1}\Psi(r).
\end{equation}
Then Theorem \textup{\ref{thm:exit-probability}-(6)} holds.

Moreover, additionally if $N$ satisfies \eqref{eq:X-continuous-conservative}
for any $x\in M\setminus N$ and if \eqref{eq:exit-time-lower} holds for
any $(x,r)\in(U\setminus N)\times(0,R)$ with $B(x,r)\subset U$, then
Theorem \textup{\ref{thm:exit-probability}-(6)} with ``and $\overline{B(x,r)}$ compact" removed holds.
\end{theorem}

\begin{proof}
As in the proof of Theorem \ref{thm:exit-probability}, we treat the two cases
simultaneously, one with ``and $\overline{B(x,r)}$ compact" kept and without
\eqref{eq:X-continuous-conservative} and the other with ``and $\overline{B(x,r)}$ compact"
removed and \eqref{eq:X-continuous-conservative} assumed.
Let $(x,r)\in(U\setminus N)\times(0,R)$ satisfy $B(x,r)\subset U$.
We assume in the former case that
$\overline{B(x,r)}$ is compact, while not in the latter case. We claim that
\begin{equation}\label{eq:exit-time-probability-claim}
\mathbb{P}_{x}[\tau_{B(x,r)}\leq{\textstyle\frac{1}{2}}c_{\mathrm{E}}^{-1}\Psi(r)]
	\leq 1-(c_{\mathrm{E}}^{2}c_{\Psi}2^{\beta_{2}+1})^{-1},
\end{equation}
which together with the implication \textup{(2)}$\Rightarrow$\textup{(6)}
of Theorem \ref{thm:exit-probability} shows the assertions.

To see \eqref{eq:exit-time-probability-claim} we follow \cite[Proof of Lemma 3.16]{Bar}.
Set $\tau:=\tau_{B(x,r)}$ and $t:=\frac{1}{2}c_{\mathrm{E}}^{-1}\Psi(r)$.
By using \eqref{eq:exit-time-lower}, the obvious relation
$\tau\leq t+(\tau-t)\ind{\{t<\tau\}}=t+(\tau\circ\theta_{t})\ind{\{t<\tau\}}$
and the Markov property \cite[Theorem A.1.21]{CF} of $X$ at time $t$, we have
\begin{equation}\label{eq:exit-time-probability-proof}
2t=c_{\mathrm{E}}^{-1}\Psi(r)\leq\mathbb{E}_{x}[\tau]
	\leq t+\mathbb{E}_{x}[(\tau\circ\theta_{t})\ind{\{t<\tau\}}]
	=t+\mathbb{E}_{x}\bigl[\ind{\{t<\tau\}}\mathbb{E}_{X_{t}}[\tau]\bigr].
\end{equation}
Note that $X_{t}\in B(x,r)$ on $\{t<\tau\}$, that
$\mathbb{P}_{x}[X_{t}\in N]=0$ by \eqref{eq:properly-exceptional},
and that for any $y\in B(x,r)\setminus N$,
$\tau\leq\tau_{B(y,2r)}$ by $B(x,r)\subset B(y,2r)$ and hence
$\mathbb{E}_{y}[\tau]\leq\mathbb{E}_{y}[\tau_{B(y,2r)}]
	\leq c_{\mathrm{E}}\Psi(2r)\leq c_{\mathrm{E}}c_{\Psi}2^{\beta_{2}}\Psi(r)$
by \eqref{eq:exit-time-upper} and \eqref{eq:upper-bound-function-Psi}.
It follows from \eqref{eq:exit-time-probability-proof} and these facts that
\begin{equation*}
t\leq\mathbb{E}_{x}\bigl[\ind{\{t<\tau\}}\mathbb{E}_{X_{t}}[\tau]\bigr]
	=\mathbb{E}_{x}\bigl[\ind{\{t<\tau,\,X_{t}\in B(x,r)\setminus N\}}\mathbb{E}_{X_{t}}[\tau]\bigr]
	\leq c_{\mathrm{E}}c_{\Psi}2^{\beta_{2}}\Psi(r)(1-\mathbb{P}_{x}[\tau\leq t]),
\end{equation*}
which immediately implies \eqref{eq:exit-time-probability-claim}.
\end{proof}

\begin{acknowledgements}
Essential parts of this paper were written while the second named author was
visiting the University of Bonn from March to September 2014. He would like to
thank Kobe University for its financial and administrative supports for his visit.
He also would like to express his deepest gratitude toward the members of the
stochastics research groups of the University of Bonn for their heartfelt hospitality.

The authors would like to thank Dr.\ Sebastian Andres and Dr.\ Kouji Yano
for their valuable comments on earlier versions of the manuscript.
\end{acknowledgements}

\end{document}